\documentclass[pdflatex,sn-mathphys-num]{sn-jnl}

\usepackage{mathtools}          %
\usepackage{amssymb,amsfonts}
\usepackage{mathrsfs}
\usepackage{bm}

\usepackage{graphicx}
\usepackage{epstopdf}
\usepackage{booktabs}
\usepackage{multirow}
\usepackage{diagbox}

\usepackage{algorithmic}

\usepackage{enumitem}
\usepackage{comment}
\usepackage{verbatim}
\usepackage{empheq}
\usepackage{colonequals}
\usepackage[group-separator={,}]{siunitx}

\usepackage{tikz-cd}
\usepackage{quiver}

\usepackage{bbm}               %

\DeclareFontFamily{U}{rsfs}{\skewchar\font127}
\DeclareFontShape{U}{rsfs}{m}{n}{%
  <-6> rsfs5
  <6-8> rsfs7
  <8-> rsfs10
}{}
\DeclareFontFamily{U}{bbold}{}
\DeclareFontShape{U}{bbold}{m}{n}{%
  <-6> bbold5
  <6-9> gen * bbold
  <9-12> bbold10
  <12-17.28> bbold12
  <17.28-> bbold17
}{}

\usepackage{xcolor}

\definecolor{cblue}{rgb}{0.121569,0.466667,0.705882}
\definecolor{corange}{rgb}{1.000000,0.498039,0.054902}
\definecolor{cgreen}{rgb}{0.172549,0.627451,0.172549}

\usepackage{hyperref}
\hypersetup{
  colorlinks = true,
  linkcolor  = cblue,
  citecolor  = cgreen,
  urlcolor   = corange,
}

\usepackage{xr-hyper}

\theoremstyle{thmstyleone}
\newtheorem{theorem}{Theorem}[section]
\newtheorem{lemma}[theorem]{Lemma}
\newtheorem{proposition}[theorem]{Proposition}
\newtheorem{corollary}[theorem]{Corollary}

\theoremstyle{thmstylethree}
\newtheorem{definition}[theorem]{Definition}
\newtheorem{assumption}[theorem]{Assumption}
\newtheorem{example}[theorem]{Example}

\theoremstyle{thmstyletwo}

\makeatletter
\newcommand*{\addFileDependency}[1]{%
  \typeout{(#1)}%
  \@addtofilelist{#1}%
  \IfFileExists{#1}{}{\typeout{No file #1.}}%
}
\makeatother

\renewcommand{\bar}{\overline}

\renewcommand{\emptyset}{\varnothing}

\makeatletter
\def\moverlay{\mathpalette\mov@rlay}
\def\mov@rlay#1#2{\leavevmode\vtop{%
   \baselineskip\z@skip \lineskiplimit-\maxdimen
   \ialign{\hfil$\m@th#1##$\hfil\cr#2\crcr}}}
\newcommand{\charfusion}[3][\mathord]{
    #1{\ifx#1\mathop\vphantom{#2}\fi
        \mathpalette\mov@rlay{#2\cr#3}
      }
    \ifx#1\mathop\expandafter\displaylimits\fi}
\makeatother

\newcommand{\PP}{\mathbb{P}}

\newcommand{\RR}{\mathbb{R}}
\renewcommand{\SS}{\mathbb{S}}

\DeclareSymbolFont{bbold}{U}{bbold}{m}{n}
\DeclareSymbolFontAlphabet{\mathbbold}{bbold}

\newcommand{\sA}{\mathcal{A}}
\newcommand{\sB}{\mathcal{B}}
\newcommand{\sC}{\mathcal{C}}
\newcommand{\sD}{\mathcal{D}}
\newcommand{\sE}{\mathcal{E}}

\newcommand{\sG}{\mathcal{G}}
\newcommand{\sH}{\mathcal{H}}
\newcommand{\sI}{\mathcal{I}}
\newcommand{\sJ}{\mathcal{J}}
\newcommand{\sK}{\mathcal{K}}
\newcommand{\sL}{\mathcal{L}}
\newcommand{\sM}{\mathcal{M}}
\newcommand{\sN}{\mathcal{N}}
\newcommand{\sO}{\mathcal{O}}
\newcommand{\sP}{\mathcal{P}}

\newcommand{\sQ}{\mathcal{Q}}
\newcommand{\sR}{\mathcal{R}}
\newcommand{\sS}{\mathscr{S}}
\newcommand{\sT}{\mathcal{T}}
\newcommand{\sU}{\mathcal{U}}
\newcommand{\sV}{\mathcal{V}}
\newcommand{\sW}{\mathcal{W}}
\newcommand{\sX}{\mathcal{X}}
\newcommand{\sY}{\mathcal{Y}}
\newcommand{\sZ}{\mathcal{Z}}

\renewcommand{\sl}{\mathfrak{sl}}

\DeclareSymbolFont{sfoperators}{OT1}{cmss}{m}{n}
\DeclareSymbolFontAlphabet{\mathsf}{sfoperators}
\makeatletter
\renewcommand{\operator@font}{\mathgroup\symsfoperators}
\makeatother

\DeclareMathOperator{\rank}{rank}
\DeclareMathOperator{\diag}{diag}

\DeclareMathOperator{\conv}{conv}

\DeclareMathOperator{\supp}{supp}
\DeclareMathOperator{\Span}{span}
\DeclareMathOperator{\gph}{Gph}

\renewcommand\Im{\operatorname{Im}}
\newcommand{\Ex}{\mathop{\mathbb{E}}}

\DeclareMathOperator{\dist}{dist}

\DeclareMathOperator{\sign}{sign} 
\DeclareMathOperator{\dom}{dom} 
\DeclareMathOperator{\epi}{epi}

\DeclareMathOperator{\relint}{relint} 
\DeclareMathOperator{\col}{col}

\DeclarePairedDelimiterX{\inp}[2]{\langle}{\rangle}{#1, #2}%

\newcommand{\eps}{\varepsilon}

\newcommand{\norm}[1]{\left\lVert #1\right\rVert}

\newif\ifshowcomments
\showcommentstrue

\newcommand{\cl}{\operatorname{cl}}
\usepackage[normalem]{ulem}
\usepackage{svg}
\usepackage{float} 
\usepackage{array}
\usepackage{graphicx}
\usepackage{subcaption}

\usepackage{tcolorbox}
\newtcolorbox{theorembox}{
	colback=blue!4,
	colframe=blue!30!black!40,
	boxrule=0.5pt,
	arc=3pt,
	left=10pt,right=10pt,top=8pt,bottom=8pt,
}

\title{On the Absence of Identifiable Manifolds in Finite-Max Composite Optimization}

\author*[1]{\fnm{Yifan} \sur{Wang}}
\author*[2]{\fnm{Jianhao} \sur{Ma}}
\author*[1]{\fnm{Salar} \sur{Fattahi}}
\affil[1]{\orgdiv{Department of Industrial and Operations Engineering}, \orgname{University of Michigan}, \orgaddress{\city{Ann Arbor}, \state{MI}, \country{USA}}}
\affil[2]{\orgdiv{Department of Industrial Engineering}, \orgname{Tsinghua University}, \orgaddress{\state{Beijing}, \country{China}}}

\abstract{In nonsmooth optimization, identifiable sets describe the local region eventually reached by sequences converging to a prescribed critical point. When such a set is a $C^2$ manifold on which the objective restricts to a $C^2$ function, it is called an identifiable manifold. Their appeal lies in what they enable: many first-order methods identify these manifolds in finitely many iterations, after which the iterates enter a region in which the problem is effectively smooth. Consequently, many powerful tools and guarantees from smooth optimization transplant naturally to the nonsmooth setting. Owing to these properties, much existing work has focused on characterizing conditions that guarantee their existence.
	In this work, we study a complementary question: under what conditions is a critical point devoid of any identifiable manifold? We answer this by developing a deterministic branching criterion for a broad class of finite-max composite optimization problems, characterizing when a critical point admits no identifiable manifold. This criterion is surprisingly mild in certain classes of problems: it holds with high probability for overparameterized robust low-rank recovery and almost surely at common interpolators of random minimax regression, suggesting that the absence of identifiable manifolds may be the rule rather than the exception in modern optimization.
}


\hypersetup{
	pdftitle={Absence of Identifiable Manifolds in Finite-Max Optimization},
	pdfauthor={Yifan Wang, Jianhao Ma, Salar Fattahi}
}

\ifdefined\pdfsuppresswarningdupdest
\fi
\ifdefined\pdfsuppresswarningpagegroup
\fi

\begin{document}
	\raggedbottom
	\maketitle
	
	\section{Introduction}\label{sec:introduction}

	A recurring theme in nonsmooth optimization is that apparent nonsmoothness often conceals a smoother underlying structure. The idea is to identify a local set $\sM$ near a critical point $\bar{x}$ that captures all nearly critical behavior: whenever $x_\nu\to\bar{x}$ and $v_\nu\in\partial f(x_\nu)$ satisfies $v_\nu\to0$, the points $x_\nu$ eventually lie in $\sM$. The most powerful case is when $\sM$ is a smooth manifold on which $f$ restricts smoothly, so that the local analysis reduces to smooth analysis on $\sM$. Moreover, once a first-order method identifies $\sM$, its subsequent iterates evolve on the manifold and effectively optimize the smooth restricted objective, enabling guarantees such as convergence to minimizers and avoidance of strict saddle points~\citep{lee2016gradient,lee2019first}. This perspective underlies a long line of work on partial smoothness~\citep{lewis2002active,hare2004identifying,hare2007identifying}, active-set methods~\citep{wright1993identifiable,burke1988identification,burke1990identification}, sensitivity theory~\citep{bonnans2000perturbation,rockafellar2009variational,dontchev2022lectures}, and, more recently, active-strict-saddle theory for weakly convex optimization~\citep{davis2022proximal,bianchi2024stochastic}.
	
	The interest in identifiable manifolds is driven by what they make possible at a critical point. Identifiable manifolds enable smooth local analysis at critical points. Because the restriction $f|_{\sM}$ is $C^2$, conditions such as positive-definiteness of the restricted Hessian can certify local minimality, and the restriction becomes the natural object for sensitivity and tilt-stability analysis~\citep{lewis2002active,drusvyatskiy2014optimality,bonnans2000perturbation}. Proximal-based and dual averaging methods can identify the manifold in finitely many iterations and subsequently inherit smooth convergence rates~\citep{hare2007identifying,lewis2016proximal,liang2017activity,lee2012manifold}, while manifold-restricted Newton updates attain fast local convergence~\citep{lewis2016proximal,lewis2002active}. Identifiable manifolds also underlie the theory of {active strict saddles}~\citep{davis2022proximal}: critical points lying on a $C^2$ identifiable manifold along which the objective decreases smoothly and quadratically in some tangent direction. Such points are provably avoided by several algorithms, including stochastic subgradient~\citep{bianchi2024stochastic} and proximal-type methods~\citep{davis2022proximal}.
	
	Owing to these benefits, the existing literature has focused almost entirely on sufficient conditions for the \emph{existence} of identifiable manifolds. The strongest such result is a genericity theorem: for almost every linear perturbation of a weakly convex semialgebraic function, the perturbed function admits an identifiable manifold at each of its critical points~\citep[Corollary 4.8]{drusvyatskiy2016generic}. A critical point with no identifiable manifold thus appears to be a measure-zero event, removable by a generic perturbation, leading most existing theory to exclude such points altogether. This view is mathematically sound, but it may understate the practical cost of perturbing the objective: perturbation does not merely alter the local geometry around a critical point but reshapes the geometry globally. As the following example illustrates, even an arbitrarily small linear perturbation can destroy useful global structure; a bounded problem with a benign landscape can become unbounded below.

	\begin{example}\label{ex:perturbation-cost}
		Let $r:=\sqrt{x^2+y^2}$ and consider
		\[
		h(x,y)
		=
		\chi(r)\left(5+(|x|+|y|)^2-2x^2\right),
		\]
		where $\chi$ is the $C^2$ cutoff
		\[
		\chi(r)
		=
		\begin{cases}
			1,
			& 0\leq r\leq \frac{1}{2},\\[1mm]
			\omega\left(\frac{2r-1}{3}\right),
			& \frac{1}{2}<r<2,\\[1mm]
			0,
			& r\geq 2,
		\end{cases}
		\qquad
		\omega(t):=1-10t^3+15t^4-6t^5.
		\]
		The local minimizers of $h$ are exactly the points satisfying $x^2+y^2\geq 4$, all of which attain the global minimum value $0$, while every other critical point, including the origin, is a strict saddle; see Figure~\ref{fig:ex11-global-landscape}. Specifically, near the origin, $h(x,y)=5+y^2-x^2+2|xy|,$ so $h(t,0)=5-t^2$, and hence the origin is a strict saddle. For every sequence $(x_k,y_k)\to(0,0)$, one can choose subgradients $v_k\in\partial h(x_k,y_k)$ with $v_k\to(0,0)$. Consequently, any identifiable set at the origin must contain a full neighborhood of the origin. Since $h$ is not $C^2$ on any such neighborhood, the origin admits no $C^2$ identifiable manifold and is therefore an inactive strict saddle; see Figure~\ref{fig:ex11-local-landscape}. 
		Any linear perturbation, no matter how small, makes the problem unbounded below. To see this, for any $a\neq 0$, define $h_a(z):=h(z)-\langle a,z\rangle.$ Since $h(ta)=0$ for all sufficiently large $t$,
		\[
		h_a(ta)=-t\|a\|^2\longrightarrow-\infty, \quad \text{as } t\to+\infty.
		\]
	\end{example}

	\begin{figure}
		\centering
		\begin{subfigure}[t]{0.49\textwidth}
			\centering
			\includegraphics[width=\textwidth]{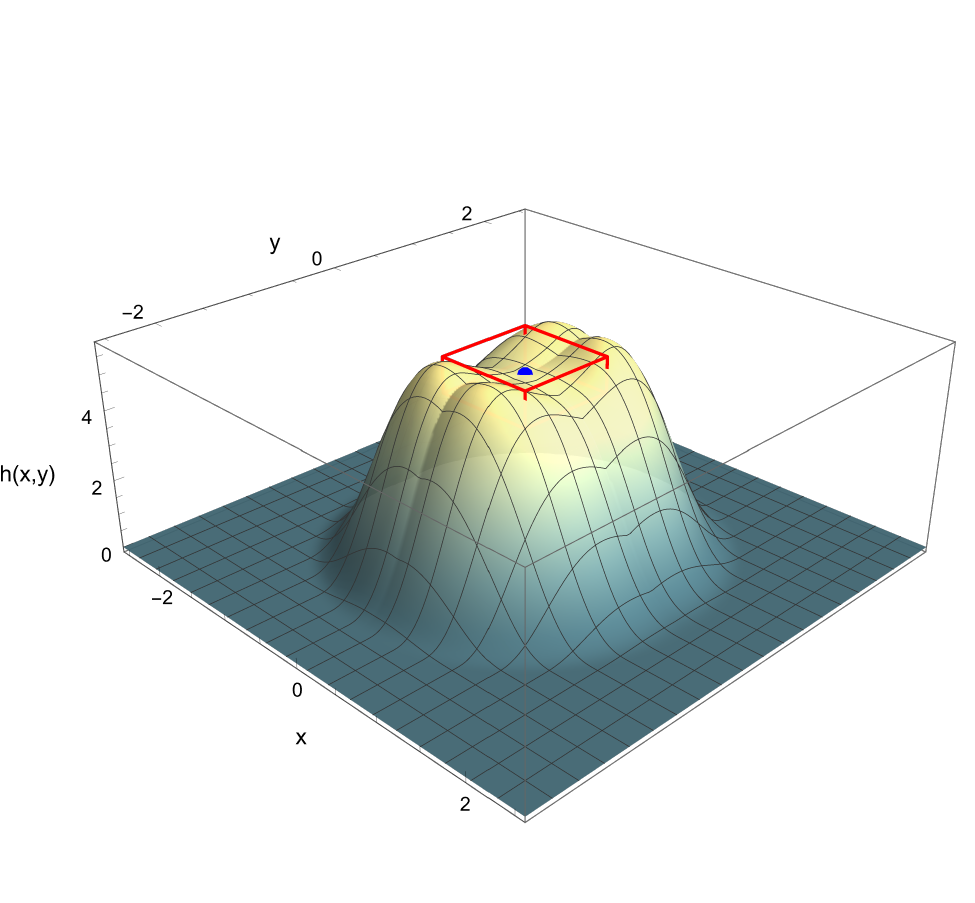}
			\caption{Global landscape}
			\label{fig:ex11-global-landscape}
		\end{subfigure}
		\hfill
		\begin{subfigure}[t]{0.49\textwidth}
			\centering
			\includegraphics[width=\textwidth]{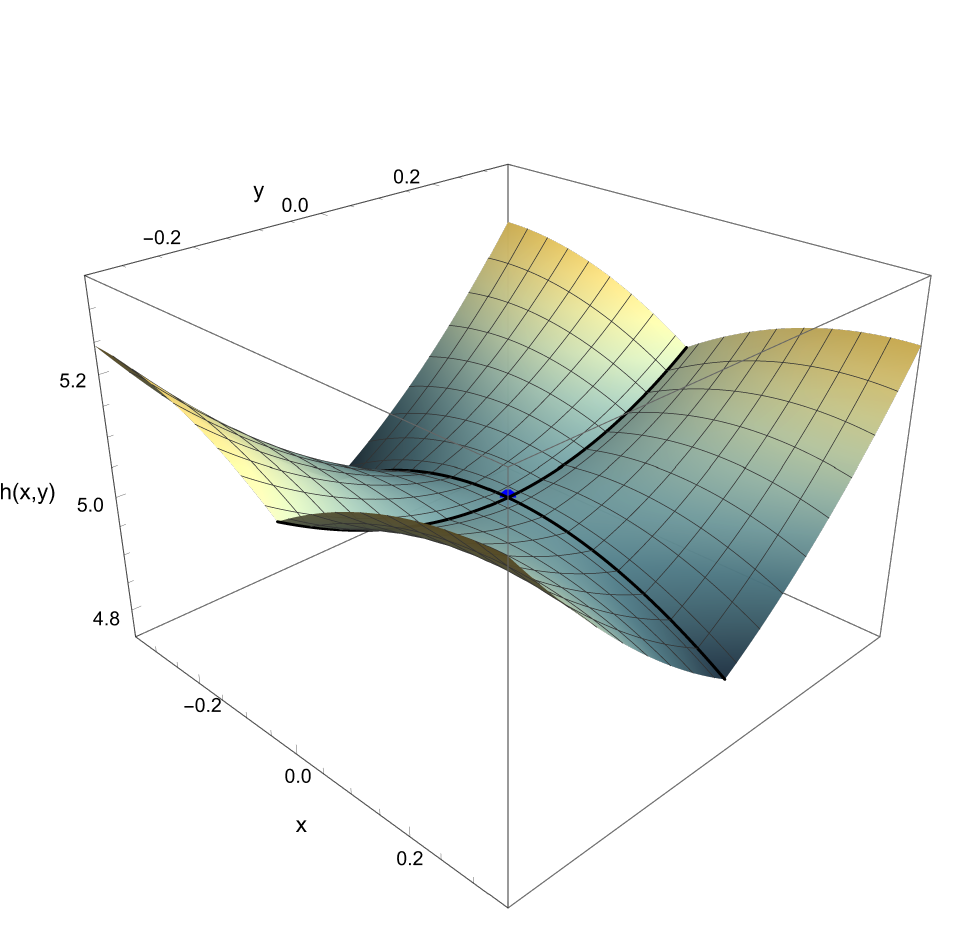}
			\caption{Local landscape}
			\label{fig:ex11-local-landscape}
		\end{subfigure}
		
		\caption{Global and local landscape of the function $h$ in Example~\ref{ex:perturbation-cost}. \emph{Left:} the benign landscape with the localized nonsmooth landscape at the origin; the red box marks the region enlarged in the right panel. \emph{Right:} the corresponding local landscape near the origin, where the origin is an inactive strict saddle. The two black curves mark the coordinate axes, the nonsmooth locus of the local term $|xy|$.}
		\label{fig:ex11-landscapes}
	\end{figure}
	
	\begin{figure}[t]
		\centering
		\begin{subfigure}[t]{0.48\linewidth}
			\centering
			\includegraphics[width=\linewidth]{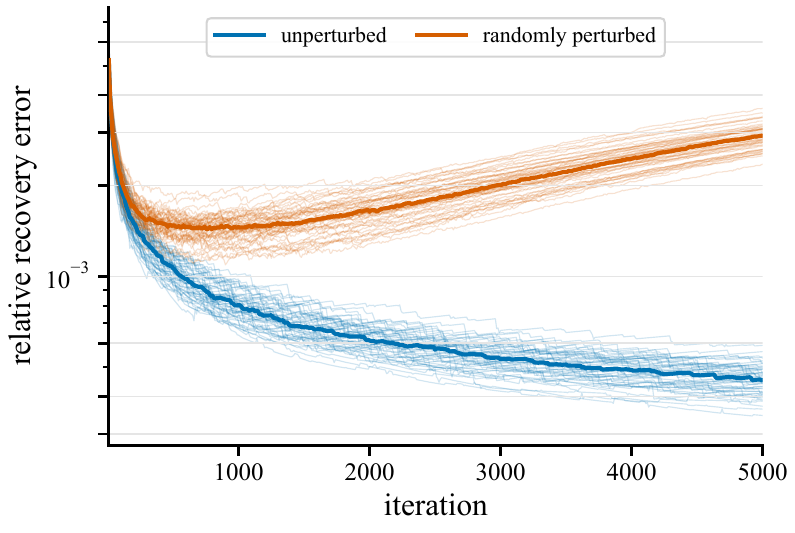}
			\caption{Relative recovery error}
			\label{fig:small-tilt-error}
		\end{subfigure}
		\hfill
		\begin{subfigure}[t]{0.48\linewidth}
			\centering
			\includegraphics[width=\linewidth]{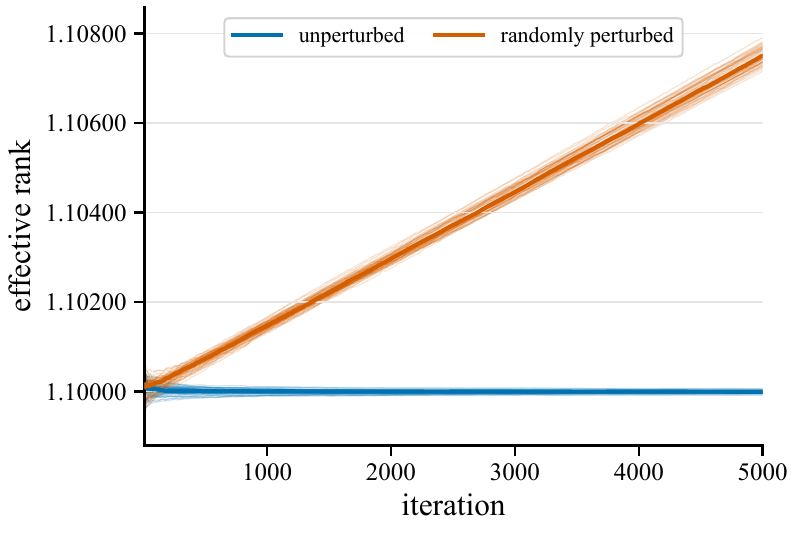}
			\caption{Effective rank}
			\label{fig:small-tilt-rank}
		\end{subfigure}
		\caption{Behavior of the subgradient method on overparameterized low-rank matrix sensing with $\ell_1$-loss. The ground-truth factor $U^\star$, satisfying $U^\star{U^\star}^\top=X^\star$, is a critical point admitting no identifiable manifold. From local random initializations, the unperturbed method converges sublinearly to $X^\star$ and retains effective rank near $1.1$, whereas a small random perturbation prevents convergence and consistently increases the effective rank. The relative recovery error and factor effective rank are
			$\norm{U_tU_t^\top-X^\star}_F/\norm{X^\star}_F$ and
			$\norm{U_t}_F^2/\norm{U_t}_{\mathrm{op}}^2$, respectively. Thin curves show the $50$ paired trials, and thick curves show their coordinatewise medians. See Appendix~\ref{app:low-rank-experiment} for details.}
		\label{fig:small-random-tilt-dynamics}
	\end{figure}
	
	The above example is not an isolated case. Figure~\ref{fig:small-random-tilt-dynamics} shows that the same tension arises---more sharply on the algorithmic side---in an important class of problems: robust low-rank matrix sensing with overparameterized rank, details of which we defer to Section~\ref{sec:low-rank-specialization}. There, the subgradient method converges to the true low-rank solution, which we show is a critical point admitting no identifiable manifold, yet convergence fails once the problem is subject to a small random perturbation.
	
	This motivates the main question studied in this paper:\vspace{3mm}
	
	\begin{tcolorbox}
		{\it Under what conditions does the original, unperturbed problem \textit{fail} to admit a $C^2$-identifiable manifold at a critical point?}
	\end{tcolorbox}\vspace{3mm}
	Beyond a few special cases, this question has not been studied systematically; see Section~\ref{subsec::related-work}. We address it for finite-max composite problems, a broad class encompassing robust estimation with $\ell_1$-loss~\citep{charisopoulos2021low,wright2010dense}, ReLU and hinge losses, finite distributionally robust optimization~\citep{wiesemann2014distributionally}, and exact-penalty reformulations of constrained problems~\citep{wright1999numerical}.
	For this class, we provide deterministic geometric conditions, called the \emph{branching criterion}, which when satisfied at a critical point rule out the existence of any identifiable manifold, and with it the algorithmic guarantees described above. We summarize the contributions of this paper below.

	\begin{enumerate}
		\item \textit{A deterministic obstruction.} For finite-max composite optimization, we provide a branching criterion that, when satisfied at a given critical point, rules out the existence of any $C^2$ identifiable manifold (Theorem~\ref{thm:branching-criterion}). The key challenge is that establishing the nonexistence of a $C^2$ identifiable manifold is a priori challenging, as one must rule out all possible manifolds simultaneously. We address this via a contrapositive argument: if a $C^2$ identifiable manifold exists, then a minimal one exists and must be locally unique. We explicitly characterize this hypothetical minimal manifold and show that the branching criterion precisely excludes it, thereby certifying nonexistence. This criterion admits a particularly geometric interpretation for stratifiable functions: it requires that two strata of the identifiable set, reachable along a common set of tangent directions, induce different quadratic forms on those directions. The existence of such incompatible directions precludes an identifiable manifold from existing.
		\item \textit{Applications.} We show that the branching criterion is far from being pathological: it is satisfied with high probability or almost surely for several important statistical estimation problems, including overparameterized nonnegative sparse regression (Theorem~\ref{thm:nonnegative-sparse}), symmetric matrix sensing (Theorem~\ref{thm:sym-MS}), asymmetric matrix sensing (Theorem~\ref{thm:asym-MS}), and finite-scenario minimax optimization (Theorem~\ref{thm:minimax-interpolation} and Corollary~\ref{cor:random-minimax-regression}). These results point to a common structural phenomenon across overparameterized learning tasks: critical points without identifiable manifolds may be universal rather than exceptional.
	\end{enumerate}
	
	\subsection{Related Work}\label{subsec::related-work}
	
	\paragraph*{Identifiability and partial smoothness} 
	The idea that nonsmooth optimization problems often carry a hidden smooth structure goes back at least to the 1990s, with the identifiable surfaces introduced by~\citet{wright1993identifiable} for constrained optimization and the active-set identification results of~\citet{burke1988identification} and~\citet{burke1990identification}, which showed that the optimal active set is correctly identified once iterates are close enough to a solution. \citet{lewis2002active} unified these threads under the notion of \textit{partial smoothness}, isolating the geometric conditions---smoothness along a manifold and sharpness normal to it---that make such identification possible; \citet{hare2004identifying, hare2007identifying} then tied this geometry to concrete algorithms. Two parallel perspectives refine the same picture: the $\mathcal{U}$--$\mathcal{V}$ decomposition of~\citet{lemarechal2000U}, which splits the space into directions of smoothness and of nonsmoothness, and the cone-reducibility theory of~\citet{bonnans2000perturbation}, developed for sensitivity analysis.~\citet{drusvyatskiy2012optimalityidentifiabilitysensitivity, drusvyatskiy2014optimality} recast the theory in terms of \textit{identifiable sets}, which need not be assumed manifolds at the outset. We adopt this point of view in our paper.

	\paragraph*{Active-set and active-manifold algorithms} 
	The terms \emph{active manifold} (from partial-smoothness theory) and \emph{identifiable manifold} (from the identifiable-set theory) refer to the same object at a critical pair: \citet[Proposition~10.8]{drusvyatskiy2012optimalityidentifiabilitysensitivity} show that the active manifold of a partly smooth function is identifiable, and conversely a $C^2$ identifiable manifold is active in the partial-smoothness sense. 
	A parallel line of work studies algorithms that identify this manifold in finite time and then exploit its smooth structure. Classical results include those of~\citet{burke1988identification} and~\citet{calamai1987projected} for projected gradient methods on polyhedral constraints. In the nonsmooth setting,~\citet{hare2007identifying} showed that the proximal point method identifies the active manifold of a partly smooth function, and~\citet{lewis2016proximal} extended this to proximal-gradient methods for composite problems.~\citet{liang2017activity} established analogous identification results for forward-backward and primal-dual splitting methods, and~\citet{lee2012manifold} obtained identification guarantees for regularized dual averaging in stochastic online learning. All of these results presuppose the existence of a smooth active manifold; the present paper studies the obstruction to that presupposition. 
	In the context of constrained optimization, both the second-order methods of~\citet{wright2002modifying} and~\citet{wright2003constraint} and the more recent first-order methods of~\citet{diaz2026active} guarantee convergence to degenerate solutions where strict complementarity fails. Although such degenerate solutions are often associated with the absence of an identifiable manifold, this connection is not causal: degeneracy alone does not guarantee absence of an identifiable manifold, as we discuss next.
	
	\paragraph*{Absence of identifiable manifolds for composite functions}
	For a composite function $h = g \circ F$, the absence of identifiable manifolds is often associated with degeneracy of the critical points, namely when the Jacobian $DF$ loses rank at a critical point~\citep{diaz2025preconditioned}. This degeneracy is closely related to {\it overparameterization} in modern statistical models, especially factored low-rank optimization and rank-overspecified robust recovery~\citep{charisopoulos2021low, chi2019nonconvex, ding2021rank, tong2021accelerating, ma2023global, ma2025understanding}. 
	Local-search algorithms often converge to such degenerate points~\citep{tong2021accelerating, diaz2025preconditioned, ma2018implicit, zhang2023preconditioned, ma2025understanding, ma2023global, qian2021structures}.
	However, rank deficiency of $DF$ alone does not preclude the existence of an identifiable manifold for $h$: it only means that $F$ cannot serve as a local defining map for such a manifold~\citep[Definition~3.1]{boumal2023introduction}. For instance, the $x$-axis $\sM=\{(x,y):y=0\}$ is a smooth manifold with valid local defining map $h_1(x,y):=y$ with a full-rank Jacobian, satisfying $\sM=h_1^{-1}(0)$. The map $h_2(x,y):=y^2$ also satisfies $\sM=h_2^{-1}(0)$ near the origin, but its Jacobian vanishes there, disqualifying it as a local defining map. This shows that the rank deficiency of one representation of $\sM$ does not preclude the existence of a valid one. 
	
	Ruling out a $C^2$ identifiable manifold is therefore considerably more delicate: one must exclude every valid local defining map, not merely show that the natural candidate $F$ fails. To the best of our knowledge, the only existing sufficient conditions are first-order. Specifically, \citet[Proposition~10.12]{drusvyatskiy2012optimalityidentifiabilitysensitivity} rule out an identifiable manifold when zero does not lie in the relative interior of the Fr\'echet subdifferential, while \citet[Theorem~4.11]{drusvyatskiy2016generic} do so when the critical cone is not a linear subspace. These conditions are equivalent~\citep[Proposition~2.1]{borwein1992partially}.
	In our setting, however, both conditions are too restrictive to be applicable: Proposition~\ref{prop:interior-point-event} shows that they fail with high probability in robust low-rank recovery. Thus, first-order geometry alone does not detect inactivity in our setting. The obstruction instead appears at second order: distinct accessible branches induce incompatible quadratic expansions that cannot arise from a single $C^2$ restriction. Capturing this incompatibility is the purpose of our branching criterion.

	\subsection{Organization}
	Section~\ref{sec:preliminaries} reviews the necessary background and introduces finite-max composite problems. Section~\ref{sec:deterministic-criterion} characterizes their locally minimal identifiable sets and establishes the branching criterion (Theorem~\ref{thm:branching-criterion}). Section~\ref{sec:minimax-interpolation} applies the branching criterion to common interpolators of finite-scenario minimax problems. Section~\ref{sec:low-rank-specialization} specializes this criterion to overparameterized robust low-rank recovery, reducing it to a geometric cone-selection test (Theorem~\ref{thm:spectrahedral-deterministic}) and deriving model-specific guarantees for nonnegative sparse regression, symmetric matrix sensing, and asymmetric matrix sensing (Theorems~\ref{thm:nonnegative-sparse},~\ref{thm:sym-MS}, and~\ref{thm:asym-MS}). Concluding remarks are provided in Section~\ref{sec:conclusion}.

	\section{Notation and Basic Constructions}\label{sec:preliminaries}

	All vector and matrix spaces are finite-dimensional Euclidean spaces over
	$\RR$, equipped with their standard inner products and induced norms. On
	matrix spaces, these are the Frobenius inner product and norm.
	We write $\|A\|_{\mathrm{op}}$ for the operator norm. The space of
	$n\times n$ symmetric matrices is denoted by $\SS^n$. For
	$X\in\SS^n$, the notation $X\succ0$ and $X\succeq0$ means that $X$ is
	positive definite and positive semidefinite, respectively. We write $I_n$
	for the $n\times n$ identity matrix and $e_i$ for the $i$th standard basis
	vector.
	
	For a positive integer $m$, let $[m]:=\{1,\ldots,m\}$. For
	$S\subseteq[m]$, we write $S^c:=[m]\setminus S$ and use $x_S$ to denote
	the restriction of a vector $x\in\RR^m$ to the coordinates in $S$. The
	sign function $\sign(x)\in\{-1,0,+1\}^m$ is applied componentwise. The
	Hadamard product is denoted by $\odot$, while $\circ$ denotes function
	composition.
	
	For Euclidean spaces $\sV$ and $\sW$ and a linear map
	$T:\sV\to\sW$, we write $T^*$, $\Im T$, $\ker T$, and $\rank T$ for its
	adjoint, range, nullspace, and rank, respectively. For a linear subspace
	$\sU\subseteq\sV$, we write $\sU^\perp$ for its orthogonal complement and
	$P_{\sU}$ for the orthogonal projection onto $\sU$. For a matrix $M$,
	$\col M$ denotes its column space, and $\Span\mathcal S$ denotes the
	linear span of a set $\mathcal S\subseteq\sV$.
	For a set $\sX\subseteq\RR^n$, we write $\cl\sX$, $\relint\sX$, and
	$\conv\sX$ for its closure, relative interior, and convex hull,
	respectively.
	For a polytope $\sP$, $\operatorname{Vert}(\sP)$ denotes its set of
	vertices.
	Given a finite index set $\sI$, the probability simplex on $\sI$ is
	\[
	\Delta_{\sI}
	:=
	\left\{
	\lambda\in\RR_+^{\sI}:
	\sum_{i\in\sI}\lambda_i=1
	\right\},
	\qquad
	\RR_+:=\{t\in\RR:t\geq0\}.
	\]
	For $\lambda\in\RR^{\sI}$, write $\supp\lambda:=\{i\in\sI:\lambda_i\neq0\}$ and $\|\lambda\|_0:=|\supp\lambda|$. For a function $h:\RR^d\to\RR\cup\{+\infty\}$, its domain, graph, and
	epigraph are
	\begin{align*}
		\dom h
		&:=\{x\in\RR^d:h(x)<+\infty\},\\
		\gph h
		&:=\{(x,h(x))\in\RR^d\times\RR:x\in\dom h\},\\
		\epi h
		&:=\{(x,r)\in\RR^d\times\RR:r\geq h(x)\}.
	\end{align*}
	For a set-valued map
	$\Phi:\RR^d\rightrightarrows\RR^m$, we define $\gph\Phi
	:=
	\{(x,v)\in\RR^d\times\RR^m:v\in\Phi(x)\}.$
	
	We write $N(\mu,\Sigma)$ for the multivariate normal distribution with mean $\mu$ and covariance $\Sigma$, and $\mathrm{GOE}(n)$ for the Gaussian Orthogonal Ensemble of $n\times n$ symmetric matrices with i.i.d.\ $N(0,1/2)$ off-diagonal entries and i.i.d.\ $N(0,1)$ diagonal entries. Probability is denoted by $\PP$ and expectation by $\Ex$.
	
	\subsection{Manifolds and Definable Sets}\label{subsec:manifolds-definable}
	
	We recall basic facts from differential geometry. For smooth manifolds, we follow~\citet{boumal2023introduction} and~\citet{lee2012introduction}; for o-minimal structures and stratifications, we follow~\citet{van1998tame} and~\citet{coste2000introduction}.

	For a $C^k$ map $F:\RR^d\to\RR^m$, we denote its derivative at $x$ by a linear map $DF(x):\RR^d\to\RR^m$. When $m=1$, we identify $DF(x)$ with the gradient $\nabla F(x)\in\RR^d$ and denote the Hessian by $\nabla^2F(x)$. If $F$ is $C^2$, we write $D^2F(x)[v,w]$ for its second derivative evaluated along $v,w\in\RR^d$ and $D^2F(x)[v,v]$ for the associated quadratic form along $v$.
	
	A set $\sM\subset\RR^d$ is a \emph{$C^k$ embedded submanifold} (or simply a \emph{manifold}) of dimension $p$ around $\bar x\in\sM$ if it is defined locally by a smooth nondegenerate equation: there is a neighborhood $\sU$ of $\bar x$ and a $C^k$ map $\Phi:\sU\to\RR^{d-p}$ with surjective derivative such that $\sM\cap \sU=\Phi^{-1}(0)$. Its \emph{tangent} and \emph{normal} spaces at $\bar x$ are then simply $T_{\bar x}\sM:=\ker D\Phi(\bar x)$ and $N_{\bar x}\sM:=(T_{\bar x}\sM)^\perp$, both independent of the defining map $\Phi$, and $P_{T_{\bar x}\sM}$ denotes the orthogonal projection onto the tangent space.
	
	Let $h$ be $C^1$ on $\sM$ near $\bar x$. Its \emph{Riemannian gradient} is $\nabla_{\sM}h(\bar x):=P_{T_{\bar x}\sM}\nabla\widetilde h(\bar x)$, where $\widetilde h$ is any $C^1$ extension of $h$; this definition is independent of the extension. Thus, $\bar x$ is critical for $h|_{\sM}$ when $\nabla_{\sM}h(\bar x)=0$. If $\sM$ is $C^2$, the second-order term of a local tangent-space parametrization captures its curvature, as formalized in the next lemma.
	
	\begin{lemma}[Second-order local parametrization, Chapter~5 of~\citet{boumal2023introduction}]\label{lem:chart-second-order}
		Let $\sM$ be a $C^2$ manifold around $\bar x$. There is a $C^2$ map $\gamma$ from a neighborhood of $0$ in $T_{\bar x}\sM$ onto a neighborhood of $\bar x$ in $\sM$, with $\gamma(0)=\bar x$, of the form
		\begin{equation}\label{eq:chart-second-order}
			\gamma(v)=\bar x+v+Z_{\sM}(v)+o(\|v\|^2),
			\qquad v\in T_{\bar x}\sM,
		\end{equation}
		where $Z_{\sM}:T_{\bar x}\sM\to N_{\bar x}\sM$ is a homogeneous quadratic map.
	\end{lemma}
	
	Many sets in optimization are not manifolds, but their singularities often admit a structured decomposition into smooth manifolds, formalized by the notion of \emph{stratification}.
	
	\begin{definition}[$C^k$ stratification]\label{def:stratification}
		A \emph{$C^k$ stratification} $(k\ge1)$ of a set $\sQ\subset\RR^d$ is a partition of $\sQ$ into finitely many connected $C^k$ manifolds, called \emph{strata}, such that the closure of any stratum meets a second stratum only by containing it:
		\[
		\text{for all strata }\sX,\sY\subseteq \sQ:\qquad \sY\cap\cl\sX\ne\varnothing
		\quad\Longrightarrow\quad
		\sY\subseteq\cl\sX.
		\]
		The stratification is \emph{compatible} with a family of sets if every stratum is either contained in or disjoint from each member of the family.
	\end{definition}
	
	We focus on $C^2$ stratifications, henceforth called simply stratifications. To ensure their existence, we assume throughout that our objectives are definable in a fixed {\it o-minimal} structure. Such structures are closed under Boolean operations, products, and coordinate projections, while their definable subsets of $\RR$ are precisely finite unions of points and intervals, ruling out pathologies such as infinite oscillation. They include the semialgebraic, globally subanalytic, and $\RR_{\mathrm{an},\exp}$ structures~\citep{van1998tame,coste2000introduction}. A set is \emph{definable} if it belongs to the chosen structure, and a map is definable if its graph is definable. All examples considered in this paper are semialgebraic and hence definable.
	
	\begin{theorem}[Theorem~1.3 and Corollary~1.11 of~\citet{loi1997verdier}]\label{thm:definable-stratification}
		For every $k\ge1$ and every finite family $\{\sQ_1,\dots,\sQ_p\}$ of definable subsets of $\RR^d$, the union $\bigcup_{j=1}^p\sQ_j$ admits a $C^k$ stratification compatible with the family\footnote{\citet[Theorem 1.3]{loi1997verdier} shows the existence of $C^k$ stratification for all of $\RR^d$ compatible with the family $\{\sQ_1,\dots, \sQ_p\}$, which implies the existence of such stratification for the union $\bigcup_{j=1}^p\sQ_j$.}, in the sense of Definition~\ref{def:stratification}.
	\end{theorem}
	
	We use the compatibility of this stratification in Section~\ref{sec:deterministic-criterion} to decompose an identifiable set into finitely many strata on which the objective is smooth.

	\subsection{Normal Cones and Subdifferentials}\label{subsec:variational-background}
	
	We next recall the variational objects used throughout; see~\cite{rockafellar2009variational,dontchev2022lectures} for background.
	
	Two normal cones to a closed set $\sQ\subset\RR^d$ at $\bar x\in \sQ$ underlie the subdifferentials below. The \emph{Fr\'echet normal cone} is defined as
	\[
	\widehat N_{\sQ}(\bar{x})
	:=\left\{v\in\RR^d:
	\limsup_{\substack{x\to\bar{x},\,x\in \sQ}}
	\frac{\inp{v}{x-\bar{x}}}{\|x-\bar{x}\|}\le 0\right\}.
	\]
	The \emph{limiting normal cone} enlarges it by including the limits of nearby Fr\'echet normals:
	\[
	N_{\sQ}(\bar{x})
	:=\left\{v:\exists\,x_\nu\to\bar{x},\ v_\nu\to v,
	\ v_\nu\in\widehat N_{\sQ}(x_\nu)\right\}.
	\]
	Indeed, $\widehat N_{\sQ}(\bar x)\subseteq N_{\sQ}(\bar x)$.
	Subgradients are defined through these cones via the epigraph. For lower semicontinuous $h:\RR^d\to\RR\cup\{+\infty\}$ and $\bar x\in\dom h$, the \emph{Fr\'echet} and \emph{limiting subdifferentials} are
	\[
	\widehat\partial h(\bar{x})
	:=\{v:(v,-1)\in \widehat N_{\epi h}(\bar{x},h(\bar{x}))\},
	\qquad
	\partial h(\bar{x})
	:=\{v:(v,-1)\in N_{\epi h}(\bar{x},h(\bar{x}))\},
	\]
	and the \emph{horizon subdifferential} is $\partial^\infty h(\bar{x}):=\{v:(v,0)\in N_{\epi h}(\bar{x},h(\bar{x}))\}$; we write $\gph\partial h$ for the graph of $\partial h$. For convex $h$ both subdifferentials reduce to the convex subdifferential, and for $C^1$ functions to the gradient. The horizon subdifferential records recession directions of the subdifferential graph and vanishes whenever $h$ is locally Lipschitz around $\bar x$.

	\begin{example}[Subdifferential of the $\ell_1$-norm]\label{ex:l1-subdifferential}
		Let $h(z)=\|z\|_1=\sum_{i=1}^m|z_i|$ on $\RR^m$. Since the coordinates separate, the subdifferential is
		\[
		\partial h(z)=\{y\in[-1,1]^m:y_i=\sign(z_i)\text{ for every }i\text{ with }z_i\ne0\},
		\]
		a face of the unit cube $[-1,1]^m$ whose dimension equals the number of free coordinates of $z$. In particular, $\partial h(0)=[-1,1]^m$ is the full cube, while $\partial h(z)=\{\sign(z)\}$ is a single vertex when $z$ has no zero entry. Since $h$ is Lipschitz, $\partial^\infty h(z)=\{0\}$ for all $z$. This cube-face structure reappears in Section~\ref{sec:low-rank-specialization}, where the multipliers at a critical point form a slice of the cube by an affine subspace.
	\end{example}
	
	A point $\bar x\in\dom h$ is a \emph{critical point} of $h$ if $0\in\partial h(\bar x)$, and a \emph{Fr\'echet-critical point} if $0\in\widehat\partial h(\bar x)$; the function is \emph{subdifferentially regular}
	at $\bar x$ if $\widehat\partial h(\bar x)=\partial h(\bar x)$. For locally Lipschitz functions, this is the standard Clarke regularity condition; in that case the Fr\'echet, limiting, and Clarke subdifferentials agree at $\bar x$.
	Regularity holds for every function class in this paper.
	A critical point $\bar x$ is a \emph{strict saddle point} of $h$ if, for some unit direction $w$, the standard second-order subderivative of $h$ at $\bar x$ for the subgradient $0$ satisfies $d^2 h(\bar x\mid 0)(w)<0$~\citep{rockafellar2009variational,davis2022proximal,bianchi2024stochastic}.

	\subsection{Identifiable Sets and Identifiable Manifolds}\label{subsec:identifiable}
	
	We adopt the definitions of~\citet{drusvyatskiy2012optimalityidentifiabilitysensitivity, drusvyatskiy2014optimality} for identifiable sets and manifolds.

	\begin{definition}[Identifiable set]
		Let $h:\RR^d\to\RR$ be locally Lipschitz, and let $(\bar{x},\bar{v})\in\gph\partial h$.  A set $\sM\subset\RR^d$ containing $\bar{x}$ is identifiable for $h$ at $\bar{x}$ for $\bar{v}$ if, for every sequence $x_\nu\to\bar{x}$ and every sequence $v_\nu\to\bar{v}$ with $v_\nu\in\partial h(x_\nu)$, all sufficiently large indices satisfy $x_\nu\in\sM$.
	\end{definition}
	
	\begin{definition}[Locally minimal identifiable set]
		An identifiable set $\sM$ for $h$ at $\bar{x}$ for $\bar{v}$ is locally minimal if, for every other identifiable set $\sN$ for the same pair, there exists a neighborhood $\sO$ of $\bar{x}$ such that
		\[
		\sM\cap\sO\subseteq \sN\cap\sO.
		\]
	\end{definition}
	
	\begin{definition}[$C^2$ identifiable manifold]\label{def:c2-identifiable-manifold}
		Let $h:\RR^d\to\RR$ be locally Lipschitz, and let $(\bar{x},\bar{v})\in\gph\partial h$.  A set $\sM\subset\RR^d$ containing $\bar{x}$ is a $C^2$ identifiable manifold for $h$ at $\bar{x}$ for $\bar{v}$ if it is identifiable for $h$ at $\bar{x}$ for $\bar{v}$ and, in a neighborhood of $\bar{x}$, the set $\sM$ is a $C^2$ manifold of $\RR^d$ on which $h$ restricts to a $C^2$ function.
	\end{definition}
	
	\begin{proposition}[Local minimality and local uniqueness, Proposition~8.2 of~\citet{drusvyatskiy2014optimality}] \label{prop:auto-minimality}
		Let $h:\RR^d\to\RR$ be lower-semicontinuous and subdifferentially regular at some critical point $\bar{x}$. If $\sM$ is a $C^2$-identifiable manifold for $h$ at $\bar{x}$ for $0$, then $\sM$ is automatically a locally minimal identifiable set at $\bar{x}$ for $0$.
	\end{proposition}
	
	This proposition is particularly useful in our analysis: any identifiable manifold, upon existence, must be determined \textit{uniquely} by the locally minimal identifiable set.
	
	\begin{definition}[Inactive critical point]\label{def:inactive-critical-point}
		Let $0\in\partial h(\bar{x})$.  We say that $\bar{x}$ is inactive if no $C^2$ identifiable manifold exists for $h$ at $\bar{x}$ for $0$.  
	\end{definition}
	
	We next present a useful property of inactive critical points, the proof of which is presented in Appendix~\ref{app:orthogonal-invariance}.
	
	\begin{lemma}[Orthogonal invariance of criticality and identifiability]
		\label{lem:orthogonal-invariance}
		Let $\sX$ be a finite-dimensional Euclidean space, let $h:\sX\to\RR$ be locally Lipschitz, and let $\mathcal{T}:\sX\to\sX$ be an orthogonal linear transformation satisfying $h\circ\mathcal{T}=h$.  Then, for every $(x,v)\in\gph\partial h$, $\partial h(\mathcal{T}x)=\mathcal{T}\bigl(\partial h(x)\bigr)$. Moreover, a set $\sM\subset\sX$ is identifiable for $h$ at $x$ for
		$v$ if and only if $\mathcal{T}\sM$ is identifiable for $h$ at $\mathcal{T}x$ for $\mathcal{T}v$. The same equivalence holds for
		$C^2$ identifiable manifolds. Consequently, $x$ is an inactive critical point if and only if $\mathcal{T}x$ is an inactive critical point.
	\end{lemma}
	
	We now specialize to finite-max composite objectives, a broad class for which the locally minimal identifiable set admits an explicit description.
	
	\subsection{Finite-Max Composite Functions}\label{subsec:problem-formulation}
	In this paper, we study \textit{finite-max composite} objectives
	\[
	f(x)=g(F(x)),
	\qquad
	g(z):=\max_{i\in\sI}g_i(z),
	\]
	where $\sI$ is a finite index set, $F:\RR^d\to\RR^n$ is $C^2$ and definable, and each $g_i:\RR^n\to\RR$ is $C^2$ and definable. We write $f_i := g_i\circ F$ and call each $f_i$ a smooth \emph{piece} of $f$ and each $g_i$ a smooth \emph{piece} of $g$. The nonsmoothness of $f$ comes entirely from the outer pointwise maximum.
	Fix a critical point $\bar x$ of $f$ and set $\bar z := F(\bar x)$. 
	For $z$ near $\bar z$, the \emph{active set} is defined as
	\[
	I(z) := \{i\in\sI : g_i(z) = g(z)\}.
	\]
	Since $g$ is a finite maximum of $C^2$ functions, it is locally Lipschitz and subdifferentially regular, with
	\[
	\partial g(\bar z)=\conv\{\nabla g_i(\bar z):i\in  I(\bar z)\}
	\qquad\text{and}\qquad
	\partial^\infty g(\bar z)=\{0\}.
	\]
	The two results below characterize the locally minimal identifiable set for two building blocks of finite-max composite functions: finite-max of $C^1$ functions (Proposition~\ref{prop:finite-max-outer-identifiable}) and composite functions (Proposition~\ref{prop:chain-rule}). 
	
	\begin{proposition}[Identifiable set for a finite maximum, Theorem~7.3 of~\citet{lewis2025identifiability}]\label{prop:finite-max-outer-identifiable}
		Let $g_i:\RR^n\to \RR$ be $C^1$ for each $i\in\sI$, where $\sI$ is a finite index set, and define $g(z):=\max_{i\in\sI}g_i(z)$. For $\bar{z}\in\dom g$ and $\bar{v}\in\partial g(\bar{z})$, define the set of multipliers
		\begin{equation}\nonumber
			\Lambda_{\bar{v}} := \left\{\lambda\in \Delta_{\sI}: \bar{v} = \sum_{i\in I(\bar z)} \lambda_i \nabla g_i(\bar{z}), \supp \lambda \subseteq  I(\bar z)\right\}.
		\end{equation}
		Then
		\begin{equation}\nonumber
			\sM_{\bar{v}}:= \bigcup_{\lambda\in \Lambda_{\bar{v}}} \left\{z\in \RR^n: \supp \lambda \subseteq I(z)\right\}
		\end{equation}
		is the locally minimal identifiable set for $g$ at $\bar z$ for $\bar{v}$.
	\end{proposition}
	
	\begin{example}[Identifiable set of the $\ell_1$-norm]\label{ex:l1-outer}
		Take $g=\|\cdot\|_1$ on $\RR^m$. Its pieces are the affine maps $g_\sigma(z)=\inp{\sigma}{z}$ over $\sigma\in\{\pm1\}^m=:\sI$, with $\nabla g_\sigma = \sigma$. Moreover, $\sigma$ is an element of the active set $I(z)$ exactly when $\sigma_i=\sign(z_i)$ on every coordinate with $z_i\ne0$. Fix a subgradient $\bar v\in\partial g(\bar z)$. By Proposition~\ref{prop:finite-max-outer-identifiable}, the locally minimal identifiable set for $g$ at $\bar z$ for $\bar v$ is
		\begin{equation}\label{eq:l1-outer-union}
			\sM_{\bar v}:=\bigcup_{\lambda\in\Lambda_{\bar v}}\{z\in\RR^m:\supp\lambda\subseteq I(z)\},
			\qquad
			\Lambda_{\bar v}=\Bigl\{\lambda\in\Delta_{\sI}:\bar v=\textstyle\sum_\sigma\lambda_\sigma\sigma\Bigr\},
		\end{equation}
		a union, over the representations of $\bar v$ as a convex combination of sign patterns, of the points $z$ at which every pattern in $\supp\lambda$ is simultaneously active.\footnote{The support condition $\supp\lambda\subseteq I(\bar z)$ of Proposition~\ref{prop:finite-max-outer-identifiable} is automatic here: a sign pattern disagreeing with $\bar v$ on a coordinate where $|\bar v_i|=1$ cannot appear in any representation of $\bar v$, while a coordinate where $|\bar v_i|<1$ satisfies $\bar z_i=0$ and therefore imposes no constraint on the support.} In Appendix~\ref{app::example-l1}, we show that this identifiable set reduces to:
		\begin{equation}\label{eq:Sy-section5}
			\mathcal{S}_{\bar v}:=\{z\in\RR^m:z_i=0\text{ if }|\bar v_i|<1,\ z_i\ge0\text{ if }\bar v_i=1,\ z_i\le0\text{ if }\bar v_i=-1\},
		\end{equation}
		which we call the {\it sign cone} of $\bar v$ henceforth.
		Equivalently, $\mathcal{S}_{\bar v}$ is the normal cone $N_{[-1,1]^m}(\bar v)$ to the cube $[-1,1]^m$ at $\bar v$. Call a coordinate $i$ \emph{saturated} for $\bar v$ when $|\bar v_i|=1$ and \emph{free} when $|\bar v_i|<1$. Then $\dim \mathcal{S}_{\bar v}$ is the number of saturated coordinates, complementary to the dimension of the cube face containing $\bar v$. For example, if $\bar v$ is a vertex of the cube, $\mathcal{S}_{\bar v}$ is a full-dimensional orthant; if $\bar v$ is the center, $\mathcal{S}_{\bar v}=\{0\}$. Figure~\ref{fig:ex_Sv} illustrates an example.
	\end{example}

	\begin{figure}[H]
		\centering
		\includegraphics{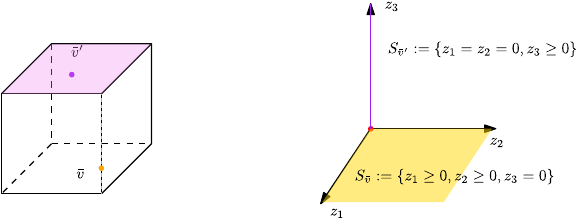}
		\caption{The locally minimal identifiable set for $g=\|\cdot\|_1$ on $\RR^3$. \emph{Left:} two subgradients $\bar v, \bar v'\in [-1,1]^3 = \partial g(0)$, where $\bar{v}$ lies on the front-right edge of the cube (two saturated coordinates) and $\bar{v}'$ lies on the top face (one saturated coordinate). \emph{Right:} the corresponding identifiable sets $\mathcal{S}_{\bar v}$ and $\mathcal{S}_{\bar v'}$, which are the normal cones to the cube at $\bar v$ and $\bar v'$, respectively. The edge point $\bar v$ (orange) gives the two-dimensional cone $\{z\in\RR^3:z_1\geq0,\,z_2\geq0,\,z_3=0\}$; the face point $\bar v'$ (purple) gives the ray $\{z\in\RR^3:z_1=z_2=0,\,z_3\geq0\}$.
        }
		\label{fig:ex_Sv}
	\end{figure}

	\begin{proposition}[Chain rule, Proposition 5.1 of~\citet{drusvyatskiy2012optimalityidentifiabilitysensitivity}]\label{prop:chain-rule}
		Consider a function $f = g\circ F$ defined on an open neighborhood $\sV\subseteq \RR^d$, where $F:\RR^d\to\RR^n$ is $C^1$ and $g:\RR^n\to\RR$ is lower-semicontinuous. Fix a critical point $\bar{x}\in \dom f$ with $\bar{z}:=F(\bar{x})$. Suppose the following qualification condition holds at $\bar{x}$
		\[
		\ker D F(\bar{x})^* \cap \partial^\infty g(F(\bar{x})) = \{0\},
		\]
		and hence the inclusion $\partial f(\bar{x}) \subseteq D F(\bar{x})^* (\partial g(\bar{z}))$ holds. Define the set of multipliers
		\[
		\Lambda:=\{v\in \partial g(\bar{z}): D F(\bar{x})^*(v)=0\} = \partial g(\bar{z}) \cap \ker D F(\bar{x})^*.
		\]
		Suppose 
		\begin{itemize}
			\item $g$ is Clarke regular at all points in $\dom g$ around $F(\bar{x})$;
			\item for each $v\in \Lambda$, there exists a locally minimal identifiable set $\sM_v$ for $g$ at $\bar{z}$ for $v$;
			\item the collection $\{\sM_v\}_{v\in \Lambda}$ is finite.
		\end{itemize}
		Then, the set 
		\[
		\sM:=\bigcup_{v\in \Lambda} F^{-1}(\sM_v)
		\]
		is locally minimal identifiable for $f$ at $\bar{x}$ for $0\in \partial f(\bar{x})$.
	\end{proposition}
	
	\begin{example}[Identifiable set of an $\ell_1$-composite]\label{ex:l1-lmi}
		Take $g=\|\cdot\|_1$ and $f=\|F(\cdot)\|_1$ for a $C^2$ map $F:\RR^d\to\RR^m$ with critical point $\bar x$ and $\bar z:=F(\bar x)$. The chain rule of Proposition~\ref{prop:chain-rule} applies: the qualification condition holds because $\partial^\infty g(\bar z)=\{0\}$, the norm is convex and hence Clarke regular, and the identifiable set of $g$ is the cone $\mathcal{S}_v$ defined in~\eqref{eq:Sy-section5}, taking finitely many values, one per sign pattern of $v$. Its relevant multipliers are
		\[
		\Lambda=\partial g(\bar z)\cap\ker DF(\bar x)^*
		=\bigl\{v\in[-1,1]^m:v_i=\sign(\bar z_i)\text{ for }\bar z_i\ne0,\ DF(\bar x)^*v=0\bigr\}.
		\]
		Since the preimage commutes with unions, the locally minimal identifiable set is
		\begin{equation}\label{eq:l1-E}
			\sM=\bigcup_{v\in\Lambda}F^{-1}(\mathcal{S}_v)=F^{-1}(\sE),
			\qquad
			\sE:=\bigcup_{v\in\Lambda}\mathcal{S}_v\subset\RR^m,
		\end{equation}
		where $\mathcal{S}_v$ is the sign cone of $v$ defined in Example~\ref{ex:l1-outer}.
		The set $\sE$ is a union of normal cones $\mathcal{S}_v = N_{[-1,1]^m}(v)$ to the cube, taken over the multiplier slice $\Lambda = \partial g(\bar z)\cap\ker DF(\bar x)^*$. Therefore, its geometry is determined by how the subspace $\ker DF(\bar x)^*$ intersects the cube. Figure~\ref{fig:ex_polyhedral_fan_hexagon} illustrates an example: depending on the orientation of the slice $\Lambda$, the geometry of the set $\sE$ can differ drastically. Whether its pullback via $F^{-1}(\sE)$ is a $C^2$ smooth manifold is precisely the question addressed by the branching criterion of Section~\ref{sec:deterministic-criterion}.

		\begin{figure}[htbp]
			\centering
			\includegraphics{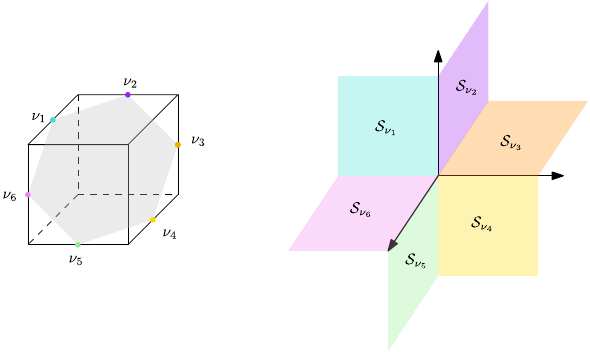}
			\caption{The identifiable set $\sE$ for a tilted slice, with $m=3$ and $\dim\Lambda=2$. \emph{Left:} the slice $\Lambda$ is a hexagon (gray), with six vertices $\nu_1,\dots,\nu_6$ (colored). \emph{Right:} the corresponding fan $\sE=\bigcup_{\nu\in\operatorname{Vert}(\Lambda)}\mathcal{S}_\nu$. The cone $\mathcal{S}_\nu=N_{[-1,1]^3}(\nu)$ for each vertex $\nu$ is a two-dimensional sector in the matching color, meeting along the rays where adjacent vertices share a saturated coordinate. The set $\sE$ is a nonconvex and nonsmooth polyhedral fan.}
			\label{fig:ex_polyhedral_fan_hexagon}
		\end{figure}
	\end{example}

	Together, these two propositions give a clean expression for the locally minimal identifiable set of the finite-max composite problem of Section~\ref{subsec:problem-formulation}. Proposition~\ref{prop:finite-max-outer-identifiable} considers the outer maximum: the identifiable set of $g$ at a subgradient $\bar v$ is the union, over multipliers $\lambda$ representing $\bar v$, of the regions where the pieces $g_i$ of $g$
	in $\supp\lambda$ stay active. Proposition~\ref{prop:chain-rule} pulls this set back through the smooth map $F$, keeping only the multipliers compatible with criticality at $\bar x$. The result is that the locally minimal identifiable set of $f$ is the $F$-preimage of a union of active-piece regions, which we derive explicitly in Section~\ref{subsec:lmi-composite} as the starting point for the branching criterion.

	\section{Deterministic Branching Criterion}\label{sec:deterministic-criterion}
	
	To derive the branching criterion for finite-max composite problems, we proceed as follows. In Section~\ref{subsec:lmi-composite}, we combine Propositions~\ref{prop:finite-max-outer-identifiable} and~\ref{prop:chain-rule} to obtain an explicit description of the locally minimal identifiable set together with a compatible stratification. In Section~\ref{subsec:general-second-order}, we work under the hypothesis that this set is a $C^2$ identifiable manifold and derive its second-order expansion. Section~\ref{subsec:accessible-strata} proves the resulting compatibility condition across accessible strata (Proposition~\ref{prop:branch-compatibility}) and obtains the branching criterion by contraposition (Theorem~\ref{thm:branching-criterion}).
	
	\subsection{Locally Minimal Identifiable Set for Finite-Max Composite Functions}\label{subsec:lmi-composite}
	The following proposition combines Propositions~\ref{prop:finite-max-outer-identifiable} and~\ref{prop:chain-rule} to express the locally minimal identifiable set of the composite $f = g\circ F$ in closed form.
	
	\begin{proposition}[Locally minimal identifiable set for finite-max composites]\label{prop:lmi-finite-max}
		Consider the finite-max composite function $f=g\circ F$ of Section~\ref{subsec:problem-formulation}. Fix a critical point $\bar x$ with $\bar z:=F(\bar x)$ and active set $I(\bar z)$. Define the critical multiplier set
		\begin{equation}\label{eq::Lambda0}
			\Lambda_0:=\left\{\lambda\in\Delta_{\sI}:
			DF(\bar x)^*\Big(\textstyle\sum_{i\in  I(\bar z)}\lambda_i\nabla g_i(\bar z)\Big)=0, \supp \lambda \subseteq  I(\bar z)\right\},
		\end{equation}
		the family of active-support sets $\mathscr{K}:=\{\supp\lambda:\lambda\in\Lambda_0\}$, and for nonempty $\sK\subseteq  I(\bar z)$, the active locus
		\[
		\sA_{\sK}:=\{z:\sK\subseteq I(z)\}=\bigcap_{i\in \sK}\bigcap_{j\in\sI}\{z:g_i(z)\ge g_j(z)\}.
		\]
		Then the set
		\[
		\sM:=F^{-1}(\sE),
		\qquad
		\sE:=\bigcup_{\sK\in\mathscr{K}}\sA_{\sK},
		\]
		is the locally minimal identifiable set for $f$ at $\bar x$ for $0\in \partial f(\bar x)$.
	\end{proposition}
	
	\begin{proof}
		To apply Proposition~\ref{prop:chain-rule}, we first need to verify its conditions. Since $g$ is a finite maximum of $C^2$ functions, it is locally Lipschitz near $\bar z$, so $\partial^\infty g(\bar z)=\{0\}$. Hence $\ker DF(\bar x)^*\cap\partial^\infty g(\bar z)=\{0\}$ trivially, and the qualification condition holds. Moreover, $g$ is the finite maximum of $C^2$ functions, and hence Clarke regular, at every point near $\bar z$. Moreover, by Proposition~\ref{prop:finite-max-outer-identifiable}, the locally minimal identifiable set $\sM_v$ for $g$ at $\bar z$ for $v\in\partial g(\bar z)$ is
		\[
		\begin{aligned}
			\sM_v
			&=\bigcup_{\lambda\in\Lambda_v}\{z:\supp\lambda\subseteq I(z)\},\\
			\Lambda_v
			&=\Big\{\lambda\in\Delta_{\sI}:v=\textstyle\sum_{i\in  I(\bar z)}\lambda_i\nabla g_i(\bar z),
			\ \supp \lambda \subseteq  I(\bar z)\Big\}.
		\end{aligned}
		\]
		The set $\sM_v$ depends on $v$ only through the collection of supports $\{\supp\lambda:\lambda\in\Lambda_v\}$, which is a family of subsets of the finite set $ I(\bar z)$. There are finitely many such families, so $\{\sM_v\}_{v\in\Lambda}$ is finite. The conditions of Proposition~\ref{prop:chain-rule} are thus satisfied, and the locally minimal identifiable set of $f$ at $\bar x$ is
		\[
		\sM=\bigcup_{v\in\Lambda}F^{-1}(\sM_v), \qquad \Lambda=\partial g(\bar z)\cap\ker DF(\bar x)^*.
		\]
		Since $\partial g(\bar z)=\conv\{\nabla g_i(\bar z):i\in I(\bar z)\}$, every $v\in\partial g(\bar z)$ admits a representation $v=\sum_{i\in I(\bar z)}\lambda_i\nabla g_i(\bar z)$ for some $\lambda\in\Delta_{\sI}$ with $\supp\lambda\subseteq I(\bar z)$. The condition $v\in\ker DF(\bar x)^*$ then reads $DF(\bar x)^*\bigl(\sum_{i\in I(\bar z)}\lambda_i\nabla g_i(\bar z)\bigr)=0$. Hence the weights representing multipliers in $\Lambda$ are exactly the elements of $\Lambda_0$ defined in~\eqref{eq::Lambda0}.
		Finally, we focus on simplifying the description of $\sM$. Pulling $F^{-1}$ outside the union and substituting the expression for $\sM_v$ gives
		\[
		\sM
		=\bigcup_{v\in\Lambda}F^{-1}(\sM_v)
		=F^{-1}\!\left(\bigcup_{v\in\Lambda}\sM_v\right)
		=F^{-1}\!\left(\bigcup_{v\in\Lambda}\ \bigcup_{\lambda\in\Lambda_v}\{z:\supp\lambda\subseteq I(z)\}\right).
		\]
		As $v$ ranges over $\Lambda$ and $\lambda$ over $\Lambda_v$, the multiplier $\lambda$ ranges exactly over $\Lambda_0$. Writing $\sK=\supp\lambda$ and recalling $\sA_{\sK}=\{z:\sK\subseteq I(z)\}$ and $\mathscr{K}=\{\supp\lambda:\lambda\in\Lambda_0\}$,
		\[
		\bigcup_{v\in\Lambda}\bigcup_{\lambda\in\Lambda_v}\{z:\supp\lambda\subseteq I(z)\}
		=\bigcup_{\lambda\in\Lambda_0}\sA_{\supp\lambda}
		=\bigcup_{\sK\in\mathscr{K}}\sA_{\sK}
		=\sE.
		\]
		Therefore $\sM=F^{-1}(\sE)$, as claimed.
	\end{proof}
	
	The identity $\sE=\bigcup_{\sK\in\mathscr{K}}\sA_{\sK}$ expresses $\sE$ through the supports of the critical multipliers, each locus $\sA_{\sK}$ collecting the points where the pieces $g_i$, $i\in\sK$, stay active. We now refine it into a description of the stratification of $\sE$. Being definable, $\sE$ admits a $C^2$ stratification globally.
	Near $\bar z$, the active set is a subset of $ I(\bar z)$: an inactive piece with $g_i(\bar z)<g(\bar z)$ stays below $g$ around $\bar z$, so $I(z)\subseteq I(\bar z)$ on a neighborhood of $\bar z$. On such a neighborhood,
	the geometry of $\sE$ is governed by which subset of $ I(\bar z)$ stays active, and the following proposition shows this subset to be locally constant: on a neighborhood of $\bar z$, the strata of $\sE$ each carry a fixed active subset of $ I(\bar z)$. For $g=\|\cdot\|_1$, Example~\ref{ex:l1-lmi} makes the decomposition concrete: $\sE=\bigcup_{v\in\Lambda}\mathcal{S}_v$ is a polyhedral fan whose maximal cones are the top-dimensional strata, meeting along the lower-dimensional cones where sign patterns coincide (Figure~\ref{fig:ex_polyhedral_fan_hexagon}).
	
	\begin{proposition}[Stratification compatible with constant active set]\label{prop:compatible-stratification}
		Let $\sE=\bigcup_{\sK\in\mathscr{K}}\sA_{\sK}$ be as in Proposition~\ref{prop:lmi-finite-max}. For each $\sJ\subseteq I(\bar z)$, define the active-set region
		\[
		\sR_\sJ:=\{z:I(z)=\sJ\}.
		\]
		In a neighborhood of $\bar z$, the set $\sE$ admits a finite $C^2$ stratification $\mathscr{Z}$ compatible with the family $\{\sR_\sJ:\sJ\subseteq I(\bar z)\}$ in the sense of Definition~\ref{def:stratification}. Equivalently, for every stratum $\sZ\in\mathscr{Z}$ there is an active set $\sI_\sZ\subseteq I(\bar z)$ such that
		\[
		I(z)=\sI_\sZ
		\qquad\text{for all }z\in\sZ.
		\]
		In particular, $g$ agrees with each active piece $g_i$, $i\in\sI_\sZ$, on $\sZ$, so $g|_\sZ$ is the restriction of a $C^2$ function.
	\end{proposition}
	
	\begin{proof}
		Choose a definable open neighborhood $\sU$ of $\bar z$ with $I(z)\subseteq I(\bar z)$ for all $z\in\sU$. We work within $\sU$ throughout, intersecting every set with $\sU$ without further mention. At each point of $\sU$ the maximum defining $g$ is attained, so the active set is a nonempty subset of $ I(\bar z)$, and the family $\{\sR_\sJ:\sJ\subseteq I(\bar z), \sJ\ne \emptyset\}$ partitions $\sU$. Each $\sR_\sJ$ is definable, being cut out by the equalities $g_i=g$ for $i\in\sJ$ together with the strict inequalities $g_i<g$ for $i\in I(\bar z)\setminus\sJ$.
		
		We first rewrite $\sE$ in terms of the regions $\sR_\sJ$. By definition, $\sA_{\sK}=\{z:\sK\subseteq I(z)\}$, and a point lies in $\sA_{\sK}$ exactly when its active set is one of the subsets of $ I(\bar z)$ containing $\sK$. Therefore
		\[
		\sA_{\sK}=\bigcup_{\sK\subseteq\sJ\subseteq I(\bar z)}\sR_\sJ,
		\qquad\text{and hence}\qquad
		\sE=\bigcup_{\sK\in\mathscr{K}}\sA_{\sK}
		=\bigcup_{\sJ\in\mathscr{K}^{\uparrow}}\sR_\sJ,
		\]
		where $\mathscr{K}^{\uparrow}:=\{\sJ\subseteq I(\bar z):\sK\subseteq\sJ\text{ for some }\sK\in\mathscr{K}\}$ is the collection of active sets that contain some member of $\mathscr{K}$. Thus $\sE$ is a union of regions, and any region $\sR_\sJ$ with $\sJ\notin\mathscr{K}^{\uparrow}$ is disjoint from $\sE$.
		Now apply Theorem~\ref{thm:definable-stratification} to the finite family $\{\sR_\sJ:\sJ\subseteq I(\bar z), \sJ\ne \emptyset\}$, whose union is $\sU$. The resulting $C^2$ stratification of $\sU$ is compatible with this family, so each stratum is contained in a single $\sR_\sJ$. Since $\sE$ is a union of the sets $\sR_\sJ$, every stratum is either contained in or disjoint from $\sE$, and we let $\mathscr{Z}$ be those contained in $\sE$. Recalling that $\sE=\bigcup_{\sJ\in\mathscr{K}^{\uparrow}}\sR_\sJ$, the members of $\mathscr{Z}$ are exactly the strata contained in some $\sR_\sJ$ with $\sJ\in\mathscr{K}^{\uparrow}$, and they partition $\sE$.
		
		Fix a stratum $\sZ\in\mathscr{Z}$ and let $\sJ$ be the unique index with $\sZ\subseteq\sR_\sJ$. The active set is then constant on $\sZ$, equal to $\sI_\sZ:=\sJ$, so that $g_i=g$ on $\sZ$ for every $i\in\sI_\sZ$. It follows that $g|_\sZ$ coincides with the restriction of the $C^2$ function $g_i$ for any such $i$, and is in particular $C^2$.
	\end{proof}

	Each active locus is a union of the sets $\sR_\sJ$, namely $A_K=\bigcup_{K\subseteq\sJ}\sR_\sJ$, and each $\sR_\sJ$ is a union of strata of $\mathscr{Z}$. Consequently $\sE$ is a union of strata, and on each stratum a single active set, hence a single piece $g_i$, agrees with $g$. Example~\ref{ex:damek-inactive-saddle} below, and the robust low-rank recovery of Section~\ref{sec:low-rank-specialization}, identify the strata explicitly and compare the pieces $g_i$ across them.
	
	Propositions~\ref{prop:lmi-finite-max} and~\ref{prop:compatible-stratification} together concretize the question of whether a critical point $\bar x$ admits a $C^2$ identifiable manifold. Such a manifold, if it exists, must coincide near $\bar x$ with $\sM=F^{-1}(\sE)$, on which $f$ restricts to a $C^2$ function with a single second-order expansion at $\bar x$. The stratification of $\sE$ turns this into a condition on the individual strata: each contributes its own second-order behavior to $f$ along the manifold. The next subsection computes the second-order contribution of each stratum, and Section~\ref{subsec:accessible-strata} derives the compatibility that these contributions must satisfy. The branching criterion is then obtained by contraposition.
	
	\subsection{Second-Order Expansion Under the Activity Hypothesis}\label{subsec:general-second-order}
	
	We derive a necessary condition for activity under the following standing hypothesis.
	
	\begin{assumption}[Activity hypothesis]\label{ass:activity}
		The locally minimal identifiable set $\sM=F^{-1}(\sE)$ is a $C^2$ identifiable manifold for $f$ at $\bar x$ for $0\in \partial f(\bar x)$.
	\end{assumption}
	
	This hypothesis entails no loss in deriving a necessary condition: if any $C^2$ identifiable manifold exists, Proposition~\ref{prop:auto-minimality} and the local uniqueness of the locally minimal identifiable set force it to agree with $\sM$ near $\bar x$. In particular, under the activity hypothesis, $\sM$ is a $C^2$ manifold near $\bar x$ and $f|_{\sM}$ is $C^2$. All tangent spaces, local charts, and chart-dependent quantities introduced below are understood under this hypothesis until it is discharged in Theorem~\ref{thm:branching-criterion}.
	
	Each stratum of $\sE$ determines a quadratic form, the leading second-order term contributed by that stratum. The activity hypothesis requires these forms to agree whenever the corresponding strata are reached along a common family of tangent directions, because the $C^2$ restriction $f|_{\sM}$ has a unique second-order expansion at $\bar x$. We first compute the contribution of each active piece and then identify which strata are accessible from $\sM$.
	
	Writing $T_{\bar{x}}\sM$ for the tangent space, Lemma~\ref{lem:chart-second-order} expresses points of $\sM$ near $\bar{x}$ as
	\begin{equation}\label{eq:general-chart}
		\gamma(v)=\bar{x}+v+Z_{\sM}(v)+o(\|v\|^2),
		\qquad
		v\in T_{\bar{x}}\sM,
	\end{equation}
	where $Z_{\sM}:T_{\bar{x}}\sM\to N_{\bar{x}}\sM$ is the homogeneous quadratic second-order term. 
	Define
	\begin{equation}\label{eq:eta-main}
		\eta_{\sM}(v):=DF(\bar{x})[Z_{\sM}(v)]+\tfrac12 D^2F(\bar{x})[v,v],
	\end{equation}
	and, for each active piece $i\in I(\bar z)$,
	\begin{equation}\label{eq:zeta-main}
		\zeta_i(v)
		:=\inp{\nabla g_i(\bar{z})}{\eta_{\sM}(v)}
		+\tfrac12\inp{DF(\bar{x})[v]}{\nabla^2 g_i(\bar{z})\,DF(\bar{x})[v]}.
	\end{equation}
	The next proposition shows that $\eta_{\sM}(v)$ and $\zeta_i(v)$ capture the second-order deviations of $\gamma(v)$ when pushed forward through $F$ and $g_i\circ F$, respectively.
	
	\begin{proposition}[Tangential first-order and second-order expansions]\label{prop:taylor-expansion}
		Let $\bar{x}$ be a critical point of $f$, and define $\bar z = F(\bar x)$. Suppose the activity hypothesis (Assumption~\ref{ass:activity}) holds.  Then:
		\begin{enumerate}
			\item for every active index $i\in  I(\bar z)$,
			\[
			P_{T_{\bar{x}}\sM}\nabla(g_i\circ F)(\bar{x})=\nabla_{\sM}f(\bar{x})=0;
			\]
			\item for all sufficiently small $v\in T_{\bar{x}}\sM$,
			\[
			F(\gamma(v))=\bar{z}+DF(\bar{x})[v]+\eta_{\sM}(v)+o(\|v\|^2);
			\]
			\item for all sufficiently small $v\in T_{\bar{x}}\sM$,
			\[
			f(\gamma(v))=g(\bar{z})+
			\max_{i\in  I(\bar z)}\zeta_i(v)+o(\|v\|^2).
			\]
		\end{enumerate}
	\end{proposition}
	
	\begin{proof}
		Write $f_i := g_i\circ F$, so that $f = \max_{i\in I(\bar z)} f_i$ on a neighborhood of $\bar x$, where each $f_i$ is $C^2$ with $\nabla f_i(\bar x) = DF(\bar x)^*\nabla g_i(\bar z)$.
		
		\emph{Proof of (i).}
		Fix $v\in T_{\bar x}\sM$ and let $t\mapsto\gamma(t)$ be a $C^1$ curve in $\sM$ with $\gamma(0)=\bar x$ and $\dot\gamma(0)=v$. Since each $f_i$ is differentiable and, on a neighborhood of $\bar x$, the active indices $ I(\bar z)$ are the only ones attaining the maximum, the one-sided derivative of $f$ along $\gamma$ at $t=0$ is
		\[
		\lim_{t\downarrow 0}\frac{f(\gamma(t))-f(\bar x)}{t}
		=\max_{i\in  I(\bar z)}\inp{\nabla f_i(\bar x)}{v}.
		\]
		Since $f|_{\sM}$ is differentiable at $\bar x$, the left-hand side equals $\inp{\nabla_{\sM}f(\bar x)}{v}$, so
		\begin{equation}\label{eq:tangential-max}
			\inp{\nabla_{\sM}f(\bar x)}{v}
			=\max_{i\in \sI_\star}\inp{\nabla f_i(\bar x)}{v}
			\qquad\text{for all }v\in T_{\bar x}\sM.
		\end{equation}
		The left-hand side of \eqref{eq:tangential-max} is linear in $v$. Noting that $-v\in T_{\bar x}\sM$ and replacing $v$ by $-v$ gives
		\[
		\max_{i\in I(\bar z)}\inp{\nabla f_i(\bar x)}{v}
		=-\max_{i\in I(\bar z)}\inp{\nabla f_i(\bar x)}{-v}
		=\min_{i\in I(\bar z)}\inp{\nabla f_i(\bar x)}{v}.
		\]
		Thus the maximum and minimum over $i\in I(\bar z)$ agree for every $v\in T_{\bar x}\sM$, so all $\inp{\nabla f_i(\bar x)}{v}$ coincide and, by~\eqref{eq:tangential-max}, equal $\inp{\nabla_{\sM}f(\bar x)}{v}$. Equivalently, $P_{T_{\bar x}\sM}\nabla f_i(\bar x) = \nabla_{\sM}f(\bar x)$ for every $i\in I(\bar z)$. Since $0\in\partial f(\bar x) = \conv\{\nabla f_i(\bar x):i\in I(\bar z)\}$, projecting onto $T_{\bar x}\sM$ gives $\nabla_{\sM}f(\bar x)=0$.
		
		\emph{Proof of (ii).}
		Using Taylor expansion of $F$ along the chart \eqref{eq:general-chart} and grouping by order gives
		\begin{align*}
			F(\gamma(v))
			&=\bar z+DF(\bar x)[v]
			+\big(DF(\bar x)[Z_{\sM}(v)]+\tfrac12 D^2F(\bar x)[v,v]\big)
			+o(\|v\|^2)\\
			&=\bar{z}+DF(\bar{x})[v]+\eta_{\sM}(v)+o(\|v\|^2),
		\end{align*}
		where the second line uses the definition of $\eta_{\sM}(v)$ in~\eqref{eq:eta-main}.
		
		\emph{Proof of (iii).}
		Fix $i\in  I(\bar z)$. Using Taylor expansion of $g_i\circ F$ along the chart \eqref{eq:general-chart}, and using the conclusion (ii), we have
		\begin{align*}
			&g_i(F(\gamma(v)))\\
			&=g_i(\bar z)
			+\inp{\nabla g_i(\bar z)}{DF(\bar x)[v]}
			+\inp{\nabla g_i(\bar{z})}{\eta_{\sM}(v)}
			+\tfrac12\inp{DF(\bar{x})[v]}{\nabla^2 g_i(\bar{z})\,DF(\bar{x})[v]}\!+\!o(\|v\|^2)\\
			&=g_i(\bar z)
			+\inp{\nabla g_i(\bar z)}{DF(\bar x)[v]}+\zeta_i(v)+o(\|v\|^2)
		\end{align*}
		where the third line uses the definition of $\zeta_i$ in \eqref{eq:zeta-main}. By conclusion (i) the linear term satisfies $\inp{\nabla g_i(\bar z)}{DF(\bar x)[v]}=\inp{\nabla f_i(\bar x)}{v}=0$ for all $v\in T_{\bar x}\sM$, and $g_i(\bar z)=g(\bar z)$ since $i$ is active, so
		\begin{align}\label{eq::zeta_i-g}
			g_i(F(\gamma(v)))=g(\bar z)+\zeta_i(v)+o(\|v\|^2).
		\end{align}
		By continuity, each inactive piece stays strictly below $g$ on a neighborhood of $\bar z$, and $\sI$ is finite. Therefore, for small enough $v$, the active piece of $f$ is attained on $ I(\bar z)$. This implies
		\[
		f(\gamma(v))=\max_{i\in  I(\bar z)}\{g_i(F(\gamma(v)))\} = g(\bar{z})+
		\max_{i\in  I(\bar z)}\zeta_i(v)+o(\|v\|^2).
		\]
		This completes the proof.
	\end{proof}
	
	Under the activity hypothesis, conclusion (i) says that every active piece of $f$ has the same first-order variation along $T_{\bar x}\sM$, equal to that of $f|_{\sM}$, and this common value is zero. Crucially, this does \emph{not} require $DF(\bar x)[v]$ itself to vanish: the inner displacement may be nonzero, provided it is orthogonal to every active outer gradient $\nabla g_i(\bar z)$. Tangent directions are thus constrained only through the active gradients, and first-order information alone cannot pin down the manifold. This is why the second-order coefficients $\zeta_i$ of conclusion (iii) are needed.
	
	\subsection{Accessible Strata and Branching Criterion}\label{subsec:accessible-strata}
	We now identify which strata contribute to the local expansion under the activity hypothesis. For a comparison between two strata to be meaningful, the manifold must reach both strata, and do so along an open set of tangent directions rather than a single curve. We introduce the following definition to isolate this condition.
	
	\begin{definition}[Accessible stratum]\label{def:accessible-stratum}
		Under the activity hypothesis (Assumption~\ref{ass:activity}), let $\sW\subseteq T_{\bar{x}}\sM$ be a linear subspace. A stratum $\sZ\in\mathscr{Z}$ is \emph{accessible} from $\sM$ at $\bar{x}$ along $\sW$ if there is a nonempty relatively open cone $\sC\subset\sW\setminus\{0\}$ such that, for every unit vector $v\in \sC$, there is $\eps(v)>0$ with
		\[
		F(\gamma(tv))\in \sZ,
		\qquad 0<t<\eps(v).
		\]
	\end{definition}

	Note that the open cone is essential in this definition. Reaching a stratum along a single direction would describe $f$ only along that one direction, whereas reaching it along an open cone describes it on a full solid angle of directions in $\sW$. Intuitively, $\sZ$ is accessible when the curve $t\mapsto F(\gamma(tv))$ enters $\sZ$ for a non-zero measure of tangent directions $v$ in $\sW$.

	\begin{lemma}[Quadratic forms agreeing on an open cone]\label{lem:quadratic-open-cone}
		Let $q_1,q_2$ be quadratic forms on a finite-dimensional vector space $\sW$. If $q_1=q_2$ on a nonempty relatively open cone in $\sW$, then $q_1= q_2$ on $\sW$.
	\end{lemma}
	
	\begin{proof}
		The difference $q_1-q_2$ is a homogeneous polynomial of degree two. A polynomial that vanishes on a nonempty open subset of its domain vanishes identically.
	\end{proof}
	
	On an accessible stratum, all active pieces share a single quadratic form, making the following definition unambiguous.
	
	\begin{lemma}[Branch quadratic of a stratum]\label{lem:stratum-quadratic}
		Under the activity hypothesis (Assumption~\ref{ass:activity}), let $\sZ\in\mathscr{Z}$ be accessible from $\sM$ at $\bar x$ along a linear subspace $\sW\subseteq T_{\bar x}\sM$. Then the restrictions $\zeta_i|_{\sW}$, with $\zeta_i$ as in~\eqref{eq:zeta-main}, coincide for all $i\in\sI_{\sZ}$. We call this common quadratic form the \emph{branch quadratic} of $\sZ$ along $\sW$ and denote it $\zeta_{\sZ,\sW}$.
	\end{lemma}
	
	\begin{proof}
		Let $\sC\subset\sW\setminus\{0\}$ be an accessibility cone for $\sZ$, and take $i,i'\in \sI_{\sZ}$. For every unit $v\in \sC$ and all sufficiently small $t>0$, the point $F(\gamma(tv))$ lies in $\sZ$, and since the active set is constant on $\sZ$,
		\[
		g_i(F(\gamma(tv)))=g_{i'}(F(\gamma(tv))).
		\]
		By the per-branch expansion \eqref{eq::zeta_i-g}, we have $g_i(F(\gamma(tv)))=g(\bar z)+t^2\zeta_i(v)+o(t^2)$ for each active $i$, so
		\[
		0=g_i(F(\gamma(tv)))-g_{i'}(F(\gamma(tv)))
		=t^2\bigl(\zeta_i(v)-\zeta_{i'}(v)\bigr)+o(t^2).
		\]
		Dividing by $t^2$ and letting $t\downarrow0$ gives $\zeta_i(v)=\zeta_{i'}(v)$ for every unit $v\in \sC$, and hence on all of $\sC$ by homogeneity. Lemma~\ref{lem:quadratic-open-cone} then gives $\zeta_i|_{\sW}\equiv\zeta_{i'}|_{\sW}$, so $\zeta_{\sZ, \sW}$ is well-defined.
	\end{proof}
	
	The activity hypothesis imposes a single quadratic expansion across all accessible strata. This is the key necessary condition behind the branching criterion.
	
	\begin{proposition}[Compatibility of accessible branch quadratics]\label{prop:branch-compatibility}
		Under the activity hypothesis, let $\sW\subseteq T_{\bar x}\sM$ be a linear subspace. If $\sZ_1,\sZ_2\in\mathscr{Z}$ are accessible from $\sM$ at $\bar x$ along $\sW$, then $\zeta_{\sZ_1,\sW}=\zeta_{\sZ_2,\sW}$ on $\sW$.
	\end{proposition}
	
	\begin{proof}
		By Proposition~\ref{prop:taylor-expansion}, $\nabla_{\sM}f(\bar x)=0$. Since $f|_{\sM}$ is $C^2$ under the activity hypothesis, there is a unique quadratic form $q$ on $T_{\bar x}\sM$ such that
		\[
		f(\gamma(v))=f(\bar x)+q(v)+o(\|v\|^2).
		\]
		Fix $j\in\{1,2\}$ and let $\sC_j\subset\sW\setminus\{0\}$ be an accessibility cone for $\sZ_j$. For $v\in\sC_j$ and all sufficiently small $t>0$, the point $F(\gamma(tv))$ lies in $\sZ_j$. Hence, for any $i\in\sI_{\sZ_j}$, the per-branch expansion~\eqref{eq::zeta_i-g} and Lemma~\ref{lem:stratum-quadratic} give
		\[
		f(\gamma(tv))=g(\bar z)+t^2\zeta_{\sZ_j,\sW}(v)+o(t^2).
		\]
		Comparing the two expansions shows that $q(v)=\zeta_{\sZ_j,\sW}(v)$ on $\sC_j$. Lemma~\ref{lem:quadratic-open-cone} therefore yields $q|_{\sW}\equiv\zeta_{\sZ_j,\sW}$. Applying this identity for $j=1,2$ proves the result.
	\end{proof}

	We now discharge the activity hypothesis. The following contrapositive is the main result of this section.
	
	\begin{theorem}[Branching criterion]\label{thm:branching-criterion}
		Let $\bar x$ be a critical point of the finite-max composite $f=g\circ F$, with locally minimal identifiable set $\sM=F^{-1}(\sE)$ from Proposition~\ref{prop:lmi-finite-max} and compatible stratification $\mathscr{Z}$ from Proposition~\ref{prop:compatible-stratification}. Suppose that the activity hypothesis implies the existence of a linear subspace $\sW\subseteq T_{\bar x}\sM$ and two strata $\sZ_1,\sZ_2\in\mathscr{Z}$ such that
		\begin{enumerate}
			\item[(a)] both $\sZ_1$ and $\sZ_2$ are accessible from $\sM$ at $\bar x$ along $\sW$;
			\item[(b)] their branch quadratics disagree: $\zeta_{\sZ_1,\sW}\not\equiv\zeta_{\sZ_2,\sW}$.
		\end{enumerate}
		Then $f$ admits no $C^2$ identifiable manifold at $\bar x$ for $0\in\partial f(\bar x)$; that is, $\bar x$ is inactive.
	\end{theorem}
	
	\begin{proof}
		Suppose, toward a contradiction, that a $C^2$ identifiable manifold exists at $\bar x$ for $0\in\partial f(\bar x)$. By Proposition~\ref{prop:auto-minimality} and the local uniqueness of the locally minimal identifiable set, this manifold agrees near $\bar x$ with $\sM=F^{-1}(\sE)$. Thus the activity hypothesis holds. The premise of the theorem then provides $\sW$, $\sZ_1$, and $\sZ_2$ satisfying (a)--(b), whereas Proposition~\ref{prop:branch-compatibility} forces $\zeta_{\sZ_1,\sW}\equiv\zeta_{\sZ_2,\sW}$, a contradiction.
	\end{proof}

	We now illustrate the branching criterion of Theorem~\ref{thm:branching-criterion} with an example from \citep{davis2022proximal}, explicitly verifying each of its conditions.

	\begin{example}[A two-piece inactive critical point]\label{ex:damek-inactive-saddle}
		Consider 
		$$f(x_1,x_2) = (|x_1|+|x_2|)^2 - 2x_1^2.$$
		\citet[Figure 1a]{davis2022proximal} visually illustrate that the origin is an inactive critical point of $f$ (in fact, a strict saddle). We now confirm this rigorously by directly applying our branching criterion. This function can be rewritten as
		\[
		f(x_1,x_2)=x_2^2-x_1^2+2|x_1x_2| = \max\{x_2^2-x_1^2+2x_1x_2, x_2^2-x_1^2-2x_1x_2\}.
		\]
		Therefore, $f$ admits a finite-max composite form $f = g\circ F$, where $F:\RR^2\to\RR^2$ is the identity map, $g=\max\{g_1,g_2\}$, and $g_1(z_1,z_2) := z_2^2 - z_1^2 + 2z_1z_2$ and $g_2(z_1,z_2) := z_2^2 - z_1^2 - 2z_1z_2$, both $C^2$.
		To apply the branching criterion, we first obtain the locally minimal identifiable set $\sM$ at the critical point $\bar x=(0,0)$ via Proposition~\ref{prop:lmi-finite-max}. Both pieces and their gradients vanish at $\bar x$, so the active set $I(\bar z)$ at $\bar z = (0,0)$, the critical multiplier set $\Lambda_0$, and the family of active support sets $\mathscr{K}$ reduce to
		\[
		I(\bar z)=\{1,2\},
		\qquad
		\Lambda_0=\Delta_{\sI},
		\qquad
		\mathscr{K}=\{\{1\},\{2\},\{1,2\}\}.
		\]
		Since $g_1(z)-g_2(z)=4z_1z_2$, the active loci are $A_{\{1\}}=\{z_1z_2\ge0\}$, $A_{\{2\}}=\{z_1z_2\le0\}$, and $A_{\{1,2\}}=\{z_1z_2=0\}$. Together they cover $\RR^2$, so $\sE=\RR^2$ and $\sM=F^{-1}(\sE)=\RR^2$.
		
		We now identify the stratification of Proposition~\ref{prop:compatible-stratification}. The regions on which the active set is constant are
		\[
		\sR_{\{1\}}=\{z_1z_2>0\},
		\qquad
		\sR_{\{2\}}=\{z_1z_2<0\},
		\qquad
		\sR_{\{1,2\}}=\{z_1z_2=0\},
		\]
		on which the active piece is $g_1$, $g_2$, or both, respectively. Each is a union of strata: $\sR_{\{1\}}$ and $\sR_{\{2\}}$ split into their connected components, the open quadrants, while $\sR_{\{1,2\}}$ is the union of the two coordinate axes, which stratify into the four open half-axes and the origin. The resulting strata are the four open quadrants
		\[
		\sZ_1=\{z_1>0,\,z_2>0\},
		\sZ_2=\{z_1<0,\,z_2>0\},
		\sZ_3=\{z_1<0,\,z_2<0\},
		\sZ_4=\{z_1>0,\,z_2<0\},
		\]
		together with the four open half-axes, and the origin. The constant active sets are: $\{1\}$ on $\sZ_1$ and $\sZ_3$, $\{2\}$ on $\sZ_2$ and $\sZ_4$, and $\{1,2\}$ on the open half axes and the origin. We compare the two quadrants $\sZ_1$ and $\sZ_4$, with $\sI_{\sZ_1}=\{1\}$ and $\sI_{\sZ_4}=\{2\}$.
		
		To verify the implication required by Theorem~\ref{thm:branching-criterion}, suppose the activity hypothesis holds. Since $\sM=\RR^2$, the chart of Lemma~\ref{lem:chart-second-order} is $\gamma(v)=\bar x+v$, and we take $\sW=\RR^2$. Both quadrants are accessible from $\sM$ at $\bar x$ along $\sW$: the cones $\sC_1=\{v_1>0,\,v_2>0\}$ and $\sC_4=\{v_1>0,\,v_2<0\}$ are nonempty and relatively open, and for every unit vector $v\in \sC_1$ (resp.\ $v\in \sC_4$) and sufficiently small $t>0$ one has $F(\gamma(tv))=tv\in\sZ_1$ (resp.\ $\sZ_4$), which is the accessibility condition of Definition~\ref{def:accessible-stratum}. With $\gamma(v)=\bar x+v$, each piece is a homogeneous quadratic, so $\zeta_i(v)=g_i(v)$, and Lemma~\ref{lem:stratum-quadratic} identifies the branch quadratics
		\[
		\zeta_{\sZ_1,\sW}(v)=v_2^2-v_1^2+2v_1v_2,
		\qquad
		\zeta_{\sZ_4,\sW}(v)=v_2^2-v_1^2-2v_1v_2,
		\]
		which disagree on $\sW$. Thus the activity hypothesis implies the incompatible accessible branches required by Theorem~\ref{thm:branching-criterion}, and $\bar x$ is inactive.
	\end{example}

The example above is deliberately simple: inactivity can be verified by inspection, without invoking the branching criterion. The remainder of the paper demonstrates the criterion in two practically important problem classes: finite-scenario min--max optimization and robust low-rank recovery. The former models worst-case estimation across finitely many scenarios and provides a transparent first application: at a common interpolator, all scenario gradients vanish, so disagreement among their quadratic growth directly triggers the branching criterion. In the random overparameterized two-group model, every common interpolator is consequently an inactive global minimizer almost surely. Robust low-rank recovery, by contrast, seeks to recover a low-rank object from corrupted measurements and requires a substantially more involved reduction to a geometric cone-selection test and its probabilistic verification.

\section{Inactive Global Minimizers in Finite-Scenario Minimax Optimization}\label{sec:minimax-interpolation}

We consider the application of the branching criterion to finite-scenario minimax optimization~\cite{care2015scenario}, a standard instance of group distributionally robust optimization~\cite{sagawa2019distributionally}. Suppose the data are partitioned into $S$ prespecified groups, and let $R_s:\RR^d\to\RR^{m_s}$ denote the residual map for group $s$, which we assume to be $C^2$ and definable. We consider
\begin{equation}\label{eq:group-dro}
	f_{\mathrm{GDRO}}(x)
	:=
	\max_{s\in[S]}\ell_s(x),
	\qquad
	\ell_s(x):=\frac{1}{2}\|R_s(x)\|^2.
\end{equation}
Equivalently, $f_{\mathrm{GDRO}}(x)=\max_{p\in\Delta_{[S]}}\sum_{s=1}^S p_s\ell_s(x)$, matching the loss function of the group distributionally robust problem, wherein an adversary may choose any mixture of the empirical group distributions~\cite{sagawa2019distributionally}.
It has the finite-max composite form of Section~\ref{subsec:problem-formulation}: take \(F(x)=(R_1(x),\ldots,R_S(x))\) and, for \(z=(z_1,\ldots,z_S)\), define \(g_s(z):=\frac12\|z_s\|^2\) and \(g:=\max_{s\in[S]}g_s\).

We call \(\bar x\) a \emph{common interpolator} if \(R_s(\bar x)=0\) for every \(s\in[S]\). Such a point is a global minimizer of \(f_{\mathrm{GDRO}}\), and every smooth piece has zero gradient there. Nevertheless, the pieces may grow according to different quadratic forms around \(\bar x\). Such common interpolators arise naturally in modern overparameterized learning, where models often fit every training sample—and hence every prespecified group—simultaneously; see~\citep{zhang2016understanding,bartlett2020benign,sagawa2019distributionally}. 

The next result shows that whenever two different pieces dominate along different sets of directions, the common interpolator admits no $C^2$ identifiable manifold.
For \(\sJ_s[v]:=DR_s(\bar x)[v]\), define the scenario quadratic
\begin{equation}\label{eq:minimax-scenario-quadratic}
	q_s(v):=\frac12\left\|\sJ_s[v]\right\|^2,
	\qquad v\in\RR^d,
\end{equation}
and, for a linear subspace \(\sW\subseteq\RR^d\), the strict dominance cone
\[
\sD_s(\sW)
:=
\left\{
v\in\sW\setminus\{0\}:
q_s(v)>q_j(v)\ \text{for every }j\ne s
\right\}.
\]

\begin{theorem}[Inactive minimax interpolators]\label{thm:minimax-interpolation}
	Let \(\bar x\) be a common interpolator of \(f_{\mathrm{GDRO}}\). Suppose there are a linear subspace \(\sW\subseteq\RR^d\) and distinct indices \(s_1,s_2\in[S]\) such that both \(\sD_{s_1}(\sW)\) and \(\sD_{s_2}(\sW)\) are nonempty. Then \(\bar x\) is an inactive global minimizer of \(f_{\mathrm{GDRO}}\).
\end{theorem}

\begin{proof}
	Since every loss is nonnegative and vanishes at $\bar x$, the point $\bar x$ is a global minimizer. Moreover,
	\[
	\nabla\ell_s(\bar x)=\sJ_s^*[R_s(\bar x)]=0
	\qquad\text{for every }s\in[S].
	\]
	Since $f_{\mathrm{GDRO}}$ is a finite maximum of $C^2$ functions, it is locally Lipschitz and subdifferentially regular near $\bar x$, and the subdifferential formula of Section~\ref{subsec:problem-formulation} gives $\partial f_{\mathrm{GDRO}}(\bar x)=\conv\{\nabla\ell_s(\bar x):s\in[S]\}=\{0\}$; in particular, $\bar x$ is critical.
	
	We next identify the locally minimal identifiable set. At $\bar z=F(\bar x)=0$, every outer piece is active and $\nabla g_s(\bar z)=0$, so the critical multiplier set~\eqref{eq::Lambda0} is the entire simplex: $\Lambda_0=\Delta_{[S]}$. In particular, every singleton $\{s\}$ occurs as the support of a critical multiplier, and the corresponding active loci $\sA_{\{s\}}$ already cover the image space, because every $z$ has at least one active piece. Proposition~\ref{prop:lmi-finite-max} therefore gives
	\[
	\sE=\bigcup_{s\in[S]}\sA_{\{s\}}
	=\prod_{s=1}^S\RR^{m_s},
	\qquad
	\sM=F^{-1}(\sE)=\RR^d,
	\]
	To invoke Theorem~\ref{thm:branching-criterion}, we verify its accessibility and incompatibility conditions under the activity hypothesis. For each $s\in[S]$, let $\sR_{\{s\}}:=\{z:I(z)=\{s\}\}$ be the open region on which $g_s$ is uniquely active. Choose the compatible stratification in Proposition~\ref{prop:compatible-stratification} so that every connected component of each nonempty $\sR_{\{s\}}$ is a stratum; this is possible by taking these open components as the top-dimensional strata and stratifying their definable complement compatibly with their closures.
	
	Work under the activity hypothesis. Since $\sM=\RR^d$, we have $T_{\bar x}\sM=\RR^d$ and may take the local chart to be $\gamma(v)=\bar x+v$. Fix $s\in\{s_1,s_2\}$ and choose a unit vector $v_s\in\sD_s(\sW)$. Because $\sD_s(\sW)$ is relatively open, there is a connected relatively open neighborhood $\sV_s$ of $v_s$ in the unit sphere of $\sW$ such that $\cl{\sV_s}\subset\sD_s(\sW)$. Uniformly for $v\in\cl{\sV_s}$,
	\[
	g_s(F(\bar x+tv))-g_j(F(\bar x+tv))
	=t^2\bigl(q_s(v)-q_j(v)\bigr)+o(t^2),
	\qquad j\ne s.
	\]
	The strict inequalities on the compact set $\cl{\sV_s}$ therefore imply that, for some $\eps_s>0$,
	\[
	F(\bar x+tv)\in\sR_{\{s\}}
	\qquad
	\text{for all }v\in\sV_s\text{ and }0<t<\eps_s.
	\]
	The image of the connected set $\sV_s\times(0,\eps_s)$ under $(v,t)\mapsto F(\bar x+tv)$ is connected and hence lies in a single connected component $\sZ_s$ of $\sR_{\{s\}}$. Therefore, $\sZ_s$ is accessible from $\sM$ at $\bar x$ along $\sW$, with accessibility cone ${\sC}_s:=\{\alpha v:\alpha>0,\ v\in\sV_s\}$.
	
	It remains to compare the branch quadratics. Since $\nabla g_s(0)=0$ and $\nabla^2g_s(0)$ is the identity on the $s$-th group and zero on all other groups, definition~\eqref{eq:zeta-main} gives
	\[
	\zeta_{\sZ_s,\sW}(v)=\frac12\|\sJ_s[v]\|^2=q_s(v)
	\qquad\text{for every }v\in\sW.
	\]
	The two accessible strata $\sZ_{s_1}$ and $\sZ_{s_2}$ have incompatible branch quadratics: any $v\in\sD_{s_1}(\sW)$ satisfies $q_{s_1}(v)>q_{s_2}(v)$, and hence $q_{s_1}|_{\sW}\not\equiv q_{s_2}|_{\sW}$. Thus, under the activity hypothesis, conditions~\textup{(a)}--\textup{(b)} of Theorem~\ref{thm:branching-criterion} hold. That theorem rules out any $C^2$ identifiable manifold at $\bar x$, proving that $\bar x$ is inactive.
\end{proof}

For two scenarios, the condition above reduces to a familiar matrix test. Indeed, \(\sD_1(\sW)\) and \(\sD_2(\sW)\) are both nonempty exactly when \(\sJ_1^*\sJ_1-\sJ_2^*\sJ_2\) is indefinite on \(\sW\). A useful sufficient condition is that the two kernels are incomparable: there exist \(v_1\in\ker \sJ_2\setminus\ker \sJ_1\) and \(v_2\in\ker \sJ_1\setminus\ker \sJ_2\). Then \(q_1(v_1)>q_2(v_1)\) and \(q_2(v_2)>q_1(v_2)\).

We next show that this kernel geometry is generic in overparameterized two-group linear regression. Given \(A_s\in\RR^{m_s\times d}\) and \(b_s\in\RR^{m_s}\), consider
\begin{equation}\label{eq:two-group-minimax-regression}
	f_{\mathrm{2G}}(x)
	:=
	\max_{s\in[2]}
	\frac{1}{2m_s}\|A_sx-b_s\|^2.
\end{equation}

\begin{corollary}[Gaussian two-group minimax regression]\label{cor:random-minimax-regression}
	Let \(A_1\in\RR^{m_1\times d}\) and \(A_2\in\RR^{m_2\times d}\) have i.i.d. standard normal entries, and suppose \(m_1,m_2\geq1\) and \(m_1+m_2\leq d\). Then, with probability one, for every \(b_1\in\RR^{m_1}\) and \(b_2\in\RR^{m_2}\), the common interpolators $\{\bar x\in\RR^d:A_1\bar x=b_1,\ A_2\bar x=b_2\}$ form a nonempty affine subspace of dimension $d-m_1-m_2$ and every one of them is an inactive global minimizer of $f_{\mathrm{2G}}$.
\end{corollary}

\begin{proof}
	The stacked matrix \([A_1^\top \; A_2^\top]^\top\) has i.i.d. standard normal entries and $m_1+m_2$ rows. Since $m_1+m_2\leq d$, it has full row rank almost surely. Hence the two systems \(A_1\bar x=b_1\) and \(A_2\bar x=b_2\) admit a common interpolator for every \((b_1,b_2)\), and its solution set is an affine subspace of dimension $d-m_1-m_2$. Full row rank also implies that neither \(\ker A_1\) nor \(\ker A_2\) contains the other: either inclusion would force the rows of one nonempty block to lie in the row span of the other, contradicting full row rank of the stack. Consequently, there exist
	\[
	v_1\in\ker A_2\setminus\ker A_1,
	\qquad
	v_2\in\ker A_1\setminus\ker A_2.
	\]
	At any common interpolator, the scenario Jacobians in~\eqref{eq:minimax-scenario-quadratic} are \(\sJ_s=A_s/\sqrt{m_s}\). Thus \(q_1(v_1)>q_2(v_1)\) and \(q_2(v_2)>q_1(v_2)\), and Theorem~\ref{thm:minimax-interpolation} applies with \(\sW=\RR^d\).
\end{proof}

The mechanism here is particularly transparent. At a common interpolator, all scenario gradients vanish, so the locally minimal identifiable set fills the ambient space. Nevertheless, different scenarios dominate on different open cones and induce incompatible quadratic expansions, precluding a single \(C^2\) local model. The inactive points identified here are global minimizers of a natural finite-max problem. By contrast, the next section turns to robust low-rank recovery, where the identified inactive critical points can be strict saddles rather than global/local minimizers.

\section{Inactive Criticality in Robust Low-Rank Recovery}\label{sec:low-rank-specialization}
We next consider the class of \emph{robust low-rank recovery} problems
\begin{equation}\label{eq:recovery-problem}
\min_{X\in \SS^n}\ \left\|\sA\left(\sL\left(X\right)\right)-y\right\|_1
\quad\text{s.t.}\quad X\succeq 0,\quad \rank(X)\leq k,
\end{equation}
for which the branching criterion takes a particularly geometric form. Here, $\sA:\mathbb{S}^n\to\mathbb{R}^m$ is the linear \emph{measurement operator} defined by $[\sA(Z)]_i=\langle A_i,Z\rangle$ for $i=1,\ldots,m$, where $\{A_i\}_{i=1}^m\subset\mathbb{S}^n$ are known measurement matrices. The linear \emph{design operator} $\sL:\SS^n \to\SS^n$ is determined by the underlying generative model, such as matrix sensing or nonnegative sparse regression. The observations satisfy $y=\sA(\sL(X^\star))+\varepsilon$, where $X^\star\succeq0$ is the unknown ground-truth matrix with $\rank(X^\star)=r\leq k$, and $\varepsilon\in\mathbb{R}^m$ is a sparse vector of outliers. The $\ell_1$ loss provides robustness against these outliers.

Instances of~\eqref{eq:recovery-problem} arise in a wide range of robust statistical recovery applications, including signal processing, recommendation systems, and control; see the surveys by~\citet{wright2022high,chi2019nonconvex}.

Problem~\eqref{eq:recovery-problem} is often solved via the \emph{Burer--Monteiro} factorization~\citep{burer2003nonlinear}, which replaces the constraint set $\{X:X\succeq 0,\,\rank(X)\leq k\}$ with $\{UU^\top:U\in\mathbb{R}^{n\times k}\}$, reducing~\eqref{eq:recovery-problem} to
\begin{equation}\label{eq:recovery-problem-BM}
\min_{U\in\mathbb{R}^{n\times k}}\ \left\|\sA\left(\sL\left(UU^\top\right)\right)-y\right\|_1,
\end{equation}
which is in the finite-max composite form considered in this paper. 
Despite its nonconvex and nonsmooth nature, local-search algorithms are known to converge to the ground truth $X^\star = U^\star{U^\star}^\top$ when applied directly to~\eqref{eq:recovery-problem-BM}. For instance, the subgradient method converges to $X^\star$ for matrix sensing even when $\varepsilon\neq 0$ and the model is severely \emph{overparameterized}, i.e., the search rank dominates the true rank ($k\gg r$)~\citep{ma2023global,diaz2025preconditioned}. Analogous results hold for nonnegative sparse regression~\citep{ma2022blessing,diaz2025preconditioned}.

\citet{ma2025can} showed that the (balanced) ground-truth solutions from $\{U^\star:U^\star{U^\star}^\top=X^\star\}$ in matrix sensing emerge not as local or global minima, but as strict saddles. The convergence of the subgradient method to these points (or equivalently, its inability to escape them) therefore cannot be attributed to local optimality, leading the authors to conjecture that these strict saddles are inactive. In this work, we confirm this conjecture and prove a considerably stronger result:

\begin{quote}
\itshape Under appropriate randomness in the measurements and noise, and with overparameterization, every (balanced) critical true solution in $\{U^\star:U^\star{U^\star}^\top=X^\star\}$ satisfies the branching criterion (Theorem~\ref{thm:branching-criterion}) with high probability. Consequently, these critical true solutions are inactive, admitting no $C^2$ identifiable manifold; in regimes where they are strict saddles, they become inactive strict saddles.
\end{quote}

We formally present our results for two representative classes of robust low-rank recovery: nonnegative sparse regression and matrix sensing.

\medskip
\noindent\emph{\bf Nonnegative sparse regression:}
The goal of nonnegative sparse regression is to recover a nonnegative vector $x^\star\in\mathbb{R}^n_+$ with at most $s\ll n$ nonzero entries from linear measurements $y_i=\inp{a_i}{x^\star}+\varepsilon_i$, where $\varepsilon\in\mathbb{R}^m$ is a sparse outlier vector. A common approach is to introduce the parametrization $x=u\odot u$, where $\odot$ denotes the elementwise product, and apply the subgradient method directly to
\begin{equation}\label{eq::sp-loss}
f_{\mathrm{sp}}(u)
:=\sum_{i=1}^m
\left|\inp{a_i}{u\odot u}-y_i\right|.
\end{equation}
Nonnegative sparse regression is a special case of~\eqref{eq:recovery-problem-BM}: take $k=1$, set $A_i=\diag(a_i)$, where $\diag(a_i)$ denotes the diagonal matrix with diagonal $a_i$, and let $\sL$ be the projection onto the diagonal matrices. Applying our branching criterion under Gaussian measurements yields the following high-probability guarantee.

\begin{theorem}[Nonnegative sparse regression]\label{thm:nonnegative-sparse}
Consider the nonnegative sparse regression problem~\eqref{eq::sp-loss}. Suppose that the measurement vectors $a_1,\ldots,a_m$ have i.i.d. standard normal entries, and that the outlier vector $\varepsilon$ is independent of the measurement vectors and satisfies $\|\varepsilon\|_0\leq pm$, where $0\leq p<1/2$. Let $x^\star\in\mathbb{R}^n_+$ satisfy $\|x^\star\|_0=s$ and $n\geq s+2$. Then there exist constants $c_1,c_2>0$, depending only on $p$, such that, whenever $m\geq c_1s$, with probability at least
\[
1-e^{-c_2s}-2^{-(n-s-1)},
\]
every ground-truth factor $u^\star$ satisfying $u^\star\odot u^\star=x^\star$ is an inactive critical point of $f_{\mathrm{sp}}$.
\end{theorem}

Two features of the theorem are worth emphasizing. First, the assumption $n-s\geq2$ ensures the existence of at least two coordinates outside the support of $x^\star$. These off-support coordinates provide the directions used by the branching criterion to construct incompatible branch quadratics with high probability. The presence of off-support coordinates is not merely a technical requirement; it is, in fact, necessary. Indeed, when $n=s$, an argument analogous to that of~\citet[Theorem~V.2]{aravkin2020global} shows that $f_{\mathrm{sp}}$ is locally sharp away from the singleton $\sM=\{u^\star\}$.\footnote{That is, there exist $\mu>0$ and a neighborhood $\sU$ of $u^\star$ such that $f_{\mathrm{sp}}(u)\geq f_{\mathrm{sp}}(u^\star)+\mu\,\dist(u,\sM)$ for every $u\in\sU$.} Together with the weak convexity of $f_{\mathrm{sp}}$, this sharpness implies that $\sM$ is a locally minimal identifiable set~\citep[Proposition~8.4]{drusvyatskiy2014optimality}. Since a singleton is a $C^2$ manifold, $u^\star$ is therefore an active critical point when $n=s$.
Second, the sample-size requirement $m\geq c_1s$ is information-theoretically optimal up to constant factors. Thus, the order of measurements required to make recovery information-theoretically possible is already sufficient to ensure that the ground-truth factors are inactive critical points with high probability.

\medskip
\noindent\emph{\bf Matrix sensing:}
Matrix sensing is the matrix analogue of nonnegative sparse regression, with low rank playing the role of sparsity in the spectral domain. The goal is to recover a rank-$r$ matrix $X^\star\in\mathbb{R}^{n_1\times n_2}$ from linear measurements $y_i=\inp{\Gamma_i}{X^\star}+\varepsilon_i$, where $\Gamma_i\in\mathbb{R}^{n_1\times n_2}$ have i.i.d. standard normal entries, and $\varepsilon\in\mathbb{R}^m$ is a sparse outlier vector. A common approach is to introduce a low-rank factorization and apply the subgradient method directly to one of the factorized objectives
\begin{align}
	f_{\mathrm{sym}}(U)
	&:=\sum_{i=1}^m
	\left|\inp{\Gamma_i}{UU^\top}-y_i\right|,
	&& U\in\mathbb{R}^{n\times k}, \label{eq::sym-MS}\\
	f_{\mathrm{asym}}(U_1,U_2)
	&:=\sum_{i=1}^m
	\left|\inp{\Gamma_i}{U_1U_2^\top}-y_i\right|,
	&& (U_1,U_2)\in
	\mathbb{R}^{n_1\times k}\times\mathbb{R}^{n_2\times k},\label{eq::asym-MS}
	\end{align}
	where $k\geq r$. The symmetric formulation applies when $n_1=n_2=n$ and $X^\star$ is symmetric and positive semidefinite, whereas the asymmetric formulation accommodates general, possibly rectangular, matrices.
	Both formulations are special cases of~\eqref{eq:recovery-problem-BM}. In the symmetric setting, it suffices to take $A_i=\frac{1}{2}(\Gamma_i+\Gamma_i^\top)$ and let $\sL$ be the identity map. For the asymmetric setting, set $n:=n_1+n_2$ and introduce the stacked factor $U:=[U_1^\top\;U_2^\top]^\top\in\mathbb{R}^{n\times k}$. Define
	\[
	A_i:=\frac{1}{2}
	\begin{bmatrix}
0 & \Gamma_i \\
\Gamma_i^\top & 0
\end{bmatrix},
\qquad
\sL(X):=
\begin{bmatrix}
0 & X_{12} \\
X_{12}^\top & 0
\end{bmatrix},
\]
where $X_{12}\in\mathbb{R}^{n_1\times n_2}$ denotes the upper-right off-diagonal block of $X\in\mathbb{S}^n$. Since the upper-right off-diagonal block of $UU^\top$ is $U_1U_2^\top$, we have
$\inp{A_i}{\sL(UU^\top)}=\inp{\Gamma_i}{U_1U_2^\top}$, which establishes the claimed reduction. Applying our branching criterion to the symmetric and asymmetric formulations, again under Gaussian measurements, yields the following high-probability guarantees.

\begin{theorem}[Symmetric matrix sensing]\label{thm:sym-MS}
Consider the symmetric matrix sensing problem~\eqref{eq::sym-MS}. Suppose that the measurement matrices $\Gamma_1,\ldots,\Gamma_m$ have i.i.d. standard normal entries, and that the outlier vector $\varepsilon$ is independent of the measurement matrices and satisfies $\|\varepsilon\|_0\leq pm$ for some $0\leq p<1/2$. Let $X^\star\succeq 0$ have rank $r$, and suppose that $k\geq r+1$ and $n\geq r+2$. Then there exist constants $c_1,c_2>0$, depending only on $p$, such that, whenever $m\geq c_1nr$, with probability at least
\[
1-e^{-c_2nr}-2^{-(n-r-1)},
\]
every ground-truth factor $U^\star$ satisfying $U^\star {U^\star}^{\top}=X^\star$ is an inactive critical point of $f_{\mathrm{sym}}$.
\end{theorem}

\begin{theorem}[Asymmetric matrix sensing]\label{thm:asym-MS}
Consider the asymmetric matrix sensing problem~\eqref{eq::asym-MS}. Suppose that the measurement matrices $\Gamma_1,\ldots,\Gamma_m$ have i.i.d. standard normal entries, and that the outlier vector $\varepsilon$ is independent of the measurement matrices and satisfies $\|\varepsilon\|_0\leq pm$ for some $0\leq p<1/2$. Let $X^\star\in\RR^{n_1\times n_2}$ have rank $r$, and suppose that $k\geq r+1$, $n_1\geq r+1$, and $n_2\geq r+1$. Then there exist constants $c_1,c_2>0$, depending only on $p$, such that, whenever $m\geq c_1(n_1+n_2)r$, with probability at least
\[
1-e^{-c_2(n_1+n_2)r},
\]
every balanced ground-truth factor pair $(U_1^\star,U_2^\star)$ satisfying $U_1^\star {U_2^\star}^{\top}=X^\star$ and ${U_1^\star}^{\top}U_1^\star={U_2^\star}^{\top}U_2^\star$ is an inactive critical point of $f_{\mathrm{asym}}$.
\end{theorem}

Several observations regarding Theorems~\ref{thm:sym-MS} and~\ref{thm:asym-MS} are in order. First, \citet[Corollary~1]{ma2025can} show that, with high probability, every ground-truth factor $U^\star$ satisfying $U^\star{U^\star}^{\top}=X^\star$ is a strict saddle point of $f_{\mathrm{sym}}$, provided that $cnr\leq m\leq c'nk$ for suitable constants $c,c'>0$. Theorem~\ref{thm:sym-MS} complements this result by establishing that these ground-truth factors are also inactive. Consequently, within this sample-size regime, they are inactive strict saddles with high probability. This identifies inactivity as the geometric mechanism underlying the failure of the subgradient method to escape these strict saddles. An analogous conclusion for asymmetric matrix sensing follows by combining \citet[Corollary~2]{ma2025can} with Theorem~\ref{thm:asym-MS}.
Second, as in nonnegative sparse regression, overparameterization—that is, $k>r$—is necessary for the ground-truth factors to be inactive. Indeed, when $k=r$, \citet[Propositions~4.6 and~4.7]{charisopoulos2021low} establish that $f_{\mathrm{sym}}$ and $f_{\mathrm{asym}}$ grow sharply away from their respective ground-truth solution manifolds. Since these manifolds are $C^2$ smooth, this sharpness implies that the corresponding ground-truth factors are active critical points. Finally, the sample-size requirements in both theorems match the corresponding information-theoretic limits up to constant factors.

\medskip
\noindent\emph{\bf Structure of the proof and geometric intuition:}
The crux of our analysis is to show how the branching criterion manifests itself in robust low-rank recovery. In Section~\ref{subsec:robust-branching-condition}, we reduce the criterion to a deterministic cone-selection test involving the normal cones at the vertices of the clean multiplier polytope. Geometrically, each tangent direction in a hypothetical identifiable manifold generates a projected second-order residual, and the normal cone containing this residual determines the branch that is accessed. If two relatively open sets of tangent directions select normal cones corresponding to distinct vertices, then they access branches with incompatible quadratic expansions, forcing the ground-truth factor to be inactive. In Section~\ref{subsec:high-prob-results}, we show that, after an appropriate whitening, Gaussian measurements make the directions of these projected residuals uniformly distributed on the relevant sphere. Since each pointed normal cone occupies at most half of the sphere, the projected residuals are unlikely to all select the same branch, yielding the desired high-probability guarantees.

\subsection{Branching Criterion and Cone Selection}\label{subsec:robust-branching-condition}
Problem~\eqref{eq:recovery-problem-BM} takes the finite-max composite form of Section~\ref{subsec:problem-formulation}, $f_{\mathrm{LR}} = g\circ F$, where
\begin{align*}
F(U) = \sA\left(\sL\left(UU^\top\right)\right) - y = \sA\left(\sL\left(UU^\top - X^\star\right)\right) - \varepsilon, \qquad g(z) = \|z\|_1.
\end{align*}
Here $F:\mathbb{R}^{n\times k}\to\mathbb{R}^m$ is $C^2$ and definable, and $g(z)=\max_{\sigma\in\{\pm1\}^m}g_\sigma(z)$ with each piece $g_\sigma(z):=\inp{\sigma}{z}$ being linear. Let $\Omega_C := \{i:\varepsilon_i=0\}$ and $\Omega_N := \{i:\varepsilon_i\neq 0\}$ denote the index sets of clean and noisy measurements, with $m_C := |\Omega_C|$ and $m_N := |\Omega_N|$. We write $\sA_C:\mathbb{S}^{n}\to\mathbb{R}^{m_C}$ and $F_C:\mathbb{R}^{n\times k}\to\mathbb{R}^{m_C}$ for the restrictions of $\sA$ and $F$ to the clean measurements, and $\sA_N$, $F_N$ for the corresponding restrictions to the noisy measurements. Analogously, we write $\varepsilon_N\in \mathbb{R}^{m_N}$ to denote the restriction of $\varepsilon$ to the noisy measurements. The clean and noisy Jacobians at a ground-truth factor $U^\star$ are defined as
\begin{align}
\sJ_{C}(U^\star)[V]
&:=DF_C(U^\star)[V]
=\sA_C\left(\sL\left(U^\star V^\top+VU^{\star\top}\right)\right),\label{eq:clean-jacobian}\\
\sJ_{N}(U^\star)[V]
&:=DF_N(U^\star)[V]
=\sA_N\left(\sL\left(U^\star V^\top+VU^{\star\top}\right)\right).
\end{align}
When the base point $U^\star$ is clear from context, we suppress it from the notation and write $\sJ_C[V]$ and $\sJ_N[V]$ in place of $\sJ_C(U^\star)[V]$ and $\sJ_N(U^\star)[V]$, respectively. Likewise, $\sJ_C^*[v]$ and $\sJ_N^*[v]$ denote the actions of the corresponding adjoint linear maps on $v$. 

The following proposition gives the inclusion-minimal object that any $C^2$ identifiable manifold at a critical point $U^\star$ would have to equal locally.
\begin{proposition}[Clean reduction of the identifiable set]\label{prop:clean-reduction}
Fix a critical point $U^\star$ of $f_{\mathrm{LR}}$ satisfying $U^\star{U^\star}^\top=X^\star$. Define the clean multiplier polytope
\begin{equation}\label{eq:clean-multiplier-section5}
	\Lambda_C:=\left\{v\in[-1,1]^{m_C}:\sJ_{C}^*[v]=\sJ_{N}^*[\sign(\varepsilon_N)]\right\},
\end{equation}
Then, letting $\mathcal{S}_v$ denote the sign cone \eqref{eq:Sy-section5} in $\RR^{m_C}$, the set 
\begin{align}\label{eq::Ec-LR}
	\sM=F_C^{-1}(\sE_C),\qquad \sE_C:=\bigcup_{v\in\Lambda_C}\mathcal{S}_v\subset\RR^{m_C}
\end{align}
is the locally minimal identifiable set for $f_{\mathrm{LR}}$ at $U^\star$ for $0\in \partial f_{\mathrm{LR}}(U^\star)$.
\end{proposition}
The proof of Proposition~\ref{prop:clean-reduction} is presented in Appendix~\ref{app:clean-reduction}. It shows that the branching criterion can be studied on $F_C^{-1}(\sE_C)$, where $\sE_C$ is the identifiable set in the clean coordinates before pullback by $F_C$; the corrupted coordinates contribute only a smooth, sign-fixed term to the local expansion of the objective. Consequently, the nonsmoothness of $f$ sits entirely on the clean coordinates. To obtain a more tractable description of $\sE_C$, we introduce
\begin{equation}\label{eq:clean-spaces}
\sR_C := \Im\sJ_C, \qquad \sN_C := \ker\sJ_C^* = \sR_C^\perp.
\end{equation}
in terms of which $\Lambda_C$ in~\eqref{eq:clean-multiplier-section5} can be written as
\[
\Lambda_C = (\bar v + \sN_C)\cap[-1,1]^{m_C},
\]
where $\bar v\in\Lambda_C$ is arbitrary. 
Recall from Example~\ref{ex:l1-outer} that a coordinate $i$ of $\nu\in [-1,1]^{m_C}$ is saturated when $|\nu_i|=1$ and free when $|\nu_i|<1$. Write $\sI_<(\nu):=\{i:|\nu_i|<1\}$ for the free coordinates of $\nu$. 
Next, we impose two geometric conditions on the clean multiplier polytope $\Lambda_C$. Together they make $\Lambda_C$ a full-dimensional affine-cube section whose vertices index the full-dimensional cones of $\sE_C$.
\begin{assumption}[Affine-cube geometry of the multiplier polytope]\label{ass:affine-cube}
The clean multiplier polytope $\Lambda_C$ satisfies:
\begin{enumerate}
	\item[(I)] \emph{(Interior condition)} $\Lambda_C$ meets the interior of the cube, i.e.,  $\Lambda_C\cap(-1,1)^{m_C}\ne\emptyset$;
	\item[(N)] \emph{(Nondegeneracy condition)} For every vertex $\nu$ of $\Lambda_C$, the restriction $\sJ_{{C},\sI_<(\nu)}: \RR^{n\times k}\to \RR^{|\sI_<(\nu)|}$ of $\sJ_C$ to the free coordinates of $\nu$ has full rank, i.e., $\rank \sJ_{{C},\sI_<(\nu)} = \rank \sJ_{{C}}$.
\end{enumerate}
\end{assumption}
In Section~\ref{subsec:high-prob-results}, we show that under appropriate randomness in the model, Assumption~\ref{ass:affine-cube} holds with high probability. Under these assumptions, our next proposition shows that the set $\sE_C$ decomposes into the sign cones $\mathcal{S}_\nu$ over the vertices of $\Lambda_C$, which are equivalently the normal cones of $[-1,1]^{m_C}$ at the vertices of $\Lambda_C$, where each normal cone has dimension $\dim \sN_C$.

\begin{proposition}[Structure of $\sE_C$]\label{prop:fan-structure}
The set $\sE_C$ of \eqref{eq::Ec-LR} is a polyhedral fan:
\begin{equation}\label{eq:EC-vertex-union}
	\sE_C = \bigcup_{\nu\in\operatorname{Vert}(\Lambda_C)}\mathcal{S}_\nu,
\end{equation}
with distinct vertices $\nu$ yielding distinct sign cones $\mathcal{S}_\nu$. Moreover:
\begin{itemize}
	\item Under condition~\textnormal{(I)}, $\Lambda_C$ is a full-dimensional polytope of dimension $\dim\sN_C$, and $N_{\Lambda_C}(v) = \sR_C + \mathcal{S}_v$ for every $v\in\Lambda_C$.
	\item Under conditions~\textnormal{(I)} and~\textnormal{(N)}, every cone $\mathcal{S}_\nu$ for $\nu\in\operatorname{Vert}(\Lambda_C)$ is full-dimensional with $\dim\mathcal{S}_\nu=\dim\sN_C$, and the union~\eqref{eq:EC-vertex-union} is irredundant: no $\mathcal{S}_\nu$ is contained in another.
\end{itemize}
\end{proposition}

\begin{proof}
Recall the construction $\sE_C=\bigcup_{v\in\Lambda_C}\mathcal{S}_v$ from Proposition~\ref{prop:clean-reduction}, where $\mathcal{S}_v$ is the sign cone of $v$, or equivalently the normal cone to the cube at $v$,
\[
\mathcal{S}_{v}:=\{z\in\RR^{m_C}:z_i=0\text{ if }|v_i|<1,\ z_i\ge0\text{ if }v_i=1,\ z_i\le0\text{ if }v_i=-1\}.
\]
For two vertices $\nu,\nu'$, the intersection $\mathcal{S}_\nu\cap\mathcal{S}_{\nu'}$ is the set of $z$ supported on the coordinates $\{i: |\nu_i|=|\nu'_i|=1\text{ and }\nu_i=\nu'_i\}$. This is a common face of the two cones, and every face of $\mathcal{S}_\nu$ arises in the same way. The cones $\mathcal{S}_\nu$ together with their faces are therefore a polyhedral fan. To establish \eqref{eq:EC-vertex-union}, take $v\in\Lambda_C$ and let $\nu$ be a vertex of the smallest face of $\Lambda_C$ containing $v$. Passing from $v$ to $\nu$ only saturates further coordinates, so the saturated set of $v$ lies in that of $\nu$ and $\mathcal{S}_v\subseteq\mathcal{S}_\nu$, which gives \eqref{eq:EC-vertex-union}. To see distinct vertices give distinct cones, note that if $\mathcal{S}_\nu=\mathcal{S}_{\nu'}$, $\nu$ and $\nu'$ have the same saturated coordinates and signs. A vertex has no nonzero direction of $\sN_{C}$ along its free coordinates, since such a direction would place it inside a segment of $\Lambda$, so its saturated coordinates determine it and $\nu=\nu'$.

Now assume condition \textnormal{(I)}, and fix $\bar v\in\Lambda_C\cap(-1,1)^{m_C}$. As $\bar v$ is interior to the cube and $\Lambda_C = (\bar v + \sN_C)\cap[-1,1]^{m_C}$, $\Lambda_C$ contains a relative neighborhood of $\bar v$ in $\bar v+\sN_{C}$. Hence, it is full-dimensional there, of dimension $\dim\sN_{C}$. For the normal cone formula, let $v\in\Lambda_C$ be arbitrary. Since $\bar v+\sN_{C}$ meets the interior of the cube, the normal-cone sum rule \citep[Corollary 23.8.1]{rockafellar1997convex} applies at $v$: the normal cone to $\bar v+\sN_{C}$ is $\sN_{C}^{\perp}=\sR_{C}$ and the normal cone to the cube at $v$ is the sign cone $\mathcal{S}_v$, so $N_{\Lambda_C}(v)=N_{\bar v+\sN_{C}}(v)+N_{[-1,1]^{m_C}}(v)=\sR_{C}+\mathcal{S}_v$.

Now assume conditions~\textnormal{(I)} and~\textnormal{(N)}, and let $\nu$ be a vertex of $\Lambda_C$. By condition~\textnormal{(N)}, the restriction $\sJ_{C,\sI_<(\nu)}$ has rank equal to $\rank\sJ_C$, which requires $|\sI_<(\nu)|\geq\rank\sJ_C$. For the reverse inequality, we claim $\nu$ has at least $m_C-\rank\sJ_C$ saturated coordinates, i.e., $|\sI_<(\nu)|\leq\rank\sJ_C$. Indeed, if $\nu$ had fewer than $m_C-\rank\sJ_C$ saturated coordinates, there would exist a nonzero direction $h\in\sN_C$ supported on the free coordinates of $\nu$, placing $\nu$ in the interior of a segment of $\Lambda_C$ and contradicting that $\nu$ is a vertex. Hence $|\sI_<(\nu)|=\rank\sJ_C$, and therefore
\[
\dim\mathcal{S}_\nu = m_C - |\sI_<(\nu)| = m_C - \rank\sJ_C = \dim\sN_C.
\]
Finally, if $\mathcal{S}_\nu\subseteq\mathcal{S}_{\nu'}$ for two vertices $\nu,\nu'$, then $\mathcal{S}_\nu$ is a face of $\mathcal{S}_{\nu'}$ of equal dimension, hence equals it, giving $\nu=\nu'$. The union~\eqref{eq:EC-vertex-union} is therefore irredundant.
\end{proof}

In light of Proposition~\ref{prop:fan-structure}, the set $\sE_C=\bigcup_{\nu\in\operatorname{Vert}(\Lambda_C)}\mathcal{S}_\nu$ forms a polyhedral fan and admits a natural stratification: the top-dimensional strata are the relative interiors $\relint\mathcal{S}_\nu$ of the maximal cones, one per vertex $\nu$ of $\Lambda_C$, and lower-dimensional strata appear along their common faces. The branching criterion then compares the second-order behavior across these top-dimensional strata. Figure~\ref{fig:ec-example-2d} illustrates the geometry of $\Lambda_C$, $\mathcal{S}_\nu$, and $\sE_C$ for $\dim\sN_C=2$ with $m_C=3$.

\begin{figure}[htbp]
\centering
\includegraphics{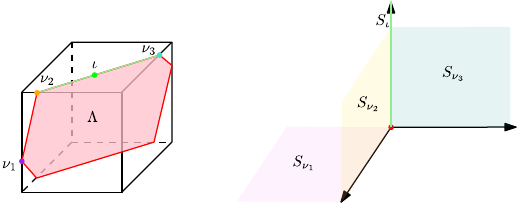}
\caption{The case $m_C=3$, $\dim\sN_C=2$. \emph{Left:} the clean multiplier polytope $\Lambda_C$ is a hexagon obtained by intersecting $\bar v+\sN_C$ with $[-1,1]^3$, satisfying Assumption~\ref{ass:affine-cube}. Three of its six vertices are $\nu_1$ (purple), $\nu_2$ (orange), and $\nu_3$ (cyan). \emph{Right:} by Proposition~\ref{prop:fan-structure}, $\sE_C=\bigcup_\nu\mathcal{S}_\nu$ over the six vertices; the cones of $\nu_1$, $\nu_2$, $\nu_3$ are shown in matching colors, each two-dimensional ($\dim\mathcal{S}_\nu=\dim\sN_C=2$). Adjacent cones share a lower-dimensional face: for a point $\iota$ (green) on the open edge $\nu_2\nu_3$, the cone $\mathcal{S}_\iota=\mathcal{S}_{\nu_2}\cap\mathcal{S}_{\nu_3}$ (green) is their common boundary ray.}
\label{fig:ec-example-2d}
\end{figure}

To apply the branching criterion (Theorem~\ref{thm:branching-criterion}), we work under the activity hypothesis and show that it forces two incompatible accessible branches. Specifically, we must determine (a) which cones $\mathcal{S}_\nu$ are accessible from $\sM$ at $U^\star$ along directions in some subspace $\sW\subseteq T_{U^\star}\sM$, and (b) whether their branch quadratics are incompatible. Condition~(b) turns out to be easy: any two distinct strata $\relint\mathcal{S}_{\nu_1}$ and $\relint\mathcal{S}_{\nu_2}$ yield incompatible branch quadratics. Condition~(a) is more delicate because it requires characterizing how curves on $\sM$ through $U^\star$ behave after being pushed forward through $F_C$. As we show below, accessibility reduces to a cone-selection test. The following lemma characterizes $T_{U^\star}\sM$ under the activity hypothesis.
\begin{lemma}[Tangent space of the clean identifiable set]\label{lem:tangent-clean-null-equality}
Fix a critical point $U^\star$ of $f_{\mathrm{LR}}$ satisfying $U^\star{U^\star}^\top=X^\star$. 
Under Assumption~\ref{ass:affine-cube}, assume $\sM=F_C^{-1}(\sE_C)$ is a $C^2$ identifiable manifold for $f_{\mathrm{LR}}$ at $U^\star$ for $0$. Then,
\[
T_{U^\star}\sM=\ker \sJ_C.
\]
\end{lemma}
We defer the proof of this lemma to Appendix~\ref{app::lemma-tangent} as it requires tools we shall develop later.\footnote{One might expect this to follow directly from the standard characterization of the tangent space as the kernel of a local defining map~\citep[Proposition~5.38]{lee2012introduction}. However, we cannot assume $\sJ_C$ to be a local defining map for $\sM$, so this standard result does not apply directly.}
Under the activity hypothesis, Section~\ref{subsec:general-second-order} locally represents $\sM$, along each tangent direction $V\in T_{U^\star}\sM$, by a curve starting from $U^\star$ with velocity $V$. Consider such a curve
\begin{align}\label{eq::curve}
\gamma_V(t)=U^\star+tV+t^2Z+o(t^2)
\end{align}
with some second-order correction $Z$ (cf.\ Lemma~\ref{lem:chart-second-order}). 
Defining the clean second-order residual 
\begin{equation}\label{eq:clean-quadratic}
\sQ_{C}(V)
:=\tfrac12 D^2F_C(U^\star)[V,V]
=\sA_C\bigl(\sL\left(VV^\top\right)\bigr)
\end{equation}
and passing the curve $\gamma_V(t)$ through $F_C$, we obtain
\begin{align}
F_C(\gamma_V(t)) &= F_C(U^\star)+t\sJ_C[V]+t^2\left(\sQ_{C}(V)+\sJ_{C}[Z]\right)+o(t^2)\nonumber\\
&=t^2\left(\sQ_{C}(V)+\sJ_{C}[Z]\right)+o(t^2)\in\sE_C,\label{eq::Fc-gamma}
\end{align}
where we use $F_C(U^\star)=0$ and $\sJ_C[V]=0$, the latter following from $V\in\ker\sJ_C$ by Lemma~\ref{lem:tangent-clean-null-equality}. Since $\sJ_C[Z]\in\sR_C$, the leading coefficient of $F_C(\gamma_V(t))$ lies in the affine subspace $\sQ_C(V)+\sR_C$ and in $\sE_C$. The curvature of $\sM$ along $\gamma$ enters only through $\sJ_C[Z]\in\sR_C$. The next proposition shows that this curvature term is immaterial: the affine subspace $\sQ_C(V)+\sR_C$ meets $\sE_C$ in a single point, determined by $\sQ_C(V)$ alone.

\begin{proposition}[Second-order cone selection]\label{prop:unique-slice}
Under Assumption~\ref{ass:affine-cube}, for every $a\in\RR^{m_C}$, the affine slice $(a+\sR_C)\cap\sE_C$ is a single point, denoted $\sS_{\Lambda_C}(a)$:
\begin{equation}\label{eq:selection-map}
	\{\sS_{\Lambda_C}(a)\}:=(a+\sR_{C})\cap \sE_C .
\end{equation}
The selection map $\sS_{\Lambda_C}(a)$ depends only on the projection $P_{\sN_{C}}a$, so it descends to a map $\sN_{C}\to\sE_C$. This descended map is a two-sided inverse of the projection $P_{\sN_{C}}$ restricted to $\sE_C$, so $P_{\sN_{C}}$ restricts to a bijection $\sE_C\to\sN_{C}$. Moreover $\sS_{\Lambda_C}$ is positively homogeneous: $\sS_{\Lambda_C}(\lambda a)=\lambda\,\sS_{\Lambda_C}(a)$ for all $\lambda\ge0$.
\end{proposition}

The proof of Proposition~\ref{prop:unique-slice} is presented in Appendix~\ref{app:unique-slice}.
It implies that $\sS_{\Lambda_C}(\sQ_C(V))$ depends only on the projection $P_{\sN_C}\sQ_C(V)$. Geometrically, this projection flattens the polyhedral fan $\sE_C$ along the first-order range $\sR_C=\Im\sJ_C$. By Proposition~\ref{prop:unique-slice}, $P_{\sN_C}$ restricts to a bijection $\sE_C\to\sN_C$ with inverse $\sS_{\Lambda_C}$, so no information is lost under projection. Each cone $\mathcal{S}_\nu$ maps to its image $P_{\sN_C}\mathcal{S}_\nu$, one per vertex $\nu$ of $\Lambda_C$, and these images partition $\sN_C$, meeting only along their boundaries.
Since $\Lambda_C$ lies in the affine subspace $\bar v+\sN_C$, let $N_{\Lambda_C}^{\sN_C}(\nu)\subset\sN_C$ denote the normal cone of $\Lambda_C$ at $\nu$ relative to $\bar v+\sN_C$. The normal cone formula of Proposition~\ref{prop:fan-structure} gives
\[
P_{\sN_C}\mathcal{S}_\nu = N_{\Lambda_C}^{\sN_C}(\nu),
\]
so flattening identifies each maximal cone of $\sE_C$ with a normal cone of $\Lambda_C$ in $\sN_C$, as illustrated in Figure~\ref{fig:projection-slice}. Consequently, determining which cone $\relint\mathcal{S}_\nu$ contains $\sS_{\Lambda_C}(\sQ_C(V))$ is equivalent to identifying which normal cone $N_{\Lambda_C}^{\sN_C}(\nu)$ contains $P_{\sN_C}\sQ_C(V)$, as formally established in the next lemma.

\begin{figure}[htbp]
\centering
\includegraphics{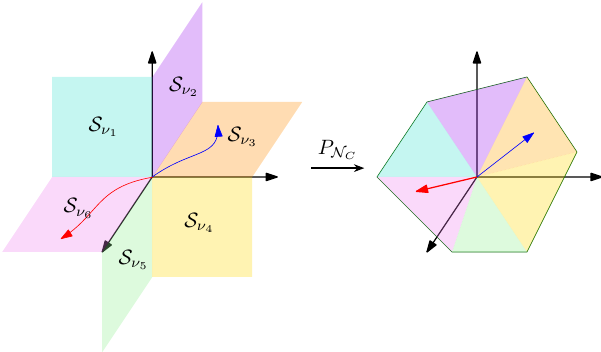}
\caption{Cone selection via flattening, continuing the hexagonal example of Figure~\ref{fig:ex_polyhedral_fan_hexagon}. \emph{Left:} the polyhedral fan $\sE_C\subset\RR^{m_C}$ with maximal cones $\mathcal{S}_{\nu_1},\dots,\mathcal{S}_{\nu_6}$, one per vertex $\nu_i$ of $\Lambda_C$, colored as in Figure~\ref{fig:ex_polyhedral_fan_hexagon}. Two tangent directions $V_1,V_2\in\ker\sJ_C$ generate curves $\gamma_{V_1},\gamma_{V_2}$ of the form \eqref{eq::curve}; by Proposition~\ref{prop:unique-slice}, their images $F_C(\gamma_{V_1})$ (blue) and $F_C(\gamma_{V_2})$ (red) leave the origin along the rays spanned by the selected points $\sS_{\Lambda_C}(\sQ_C(V_1))\in\relint\mathcal{S}_{\nu_3}$ and $\sS_{\Lambda_C}(\sQ_C(V_2))\in\relint\mathcal{S}_{\nu_6}$. \emph{Right:} the image of $\sE_C$ under the projection $P_{\sN_C}$ (arrow between the panels), inside $\sN_C$, with $\Lambda_C$ overlaid (dark green, translated from $\bar v+\sN_C$ to $\sN_C$). Each cone $\mathcal{S}_{\nu_i}$ flattens to the normal cone $N_{\Lambda_C}^{\sN_C}(\nu_i)$ at the corresponding vertex, drawn in the matching color, and these cover $\sN_C$. By Lemma~\ref{lem:branch-correspondence}, the cone selected along $V_j$ is the one whose normal cone contains the projected direction $P_{\sN_C}\sQ_C(V_j)$: the blue and red arrows depict $P_{\sN_C}\sQ_C(V_1)\in\relint N_{\Lambda_C}^{\sN_C}(\nu_3)$ and $P_{\sN_C}\sQ_C(V_2)\in\relint N_{\Lambda_C}^{\sN_C}(\nu_6)$, respectively, reducing cone selection from $\RR^{m_C}$ to the lower-dimensional space $\sN_C$.}
\label{fig:projection-slice}
\end{figure}
\begin{lemma}[Projected cone selection]\label{lem:branch-correspondence}
Under Assumption~\ref{ass:affine-cube}, let $\nu$ be a vertex of $\Lambda_C$ and $a\in\RR^{m_C}$. Then
\[
\sS_{\Lambda_C}(a)\in\relint \mathcal{S}_\nu
\qquad\Longleftrightarrow\qquad
P_{\sN_{C}}a\in\relint N_{\Lambda_C}^{\sN_{C}}(\nu).
\]
\end{lemma}

\begin{proof}
By Proposition~\ref{prop:fan-structure}, $N_\Lambda(\nu)=\sR_{C}+\mathcal{S}_\nu$. Therefore, $P_{\sN_{C}}\mathcal S_\nu
=N_{\Lambda_C}(\nu)\cap\sN_{C}
=N_{\Lambda_C}^{\sN_{C}}(\nu)$. By the linearity of $P_{\sN_{C}}$, we have
\[
P_{\sN_{C}} \Span \mathcal{S}_\nu = \Span P_{\sN_{C}} \mathcal{S}_\nu = \Span N_{\Lambda_C}^{\sN_{C}}(\nu),
\]
implying that $P_{\sN_{C}}$ is surjective. Additionally, note that $P_{\sN_{C}}$ is injective on $\sE_C$, and hence on $\mathcal{S}_\nu\subseteq \sE_C$, due to Proposition~\ref{prop:unique-slice}. Since $\mathcal{S}_\nu$ is a convex cone, $\Span \mathcal{S}_\nu=\mathcal{S}_\nu-\mathcal{S}_\nu$. If $h=z_1-z_2\in\Span \mathcal{S}_\nu$ satisfies $P_{\sN_{C}} h=0$, then $P_{\sN_{C}} z_1=P_{\sN_{C}} z_2$, and the injectivity of $P_{\sN_{C}}$ on $\mathcal{S}_\nu$ gives $z_1=z_2$, hence $h=0$. So $P_{\sN_{C}}$ is injective on $\Span \mathcal{S}_\nu$, and therefore a linear isomorphism onto $\Span N_{\Lambda_C}^{\sN_C}(\nu)$. Thus, it maps the relative interior of $\mathcal{S}_\nu$ onto the relative interior of $N_{\Lambda_C}^{\sN_C}(\nu)$. That is, for any $z\in\sE_C$,
\[
z\in\relint\mathcal S_\nu
\quad\Longleftrightarrow\quad
P_{\sN_{C}}z
\in\relint N_{\Lambda_C}^{\sN_{C}}(\nu).
\]
Now take $z=\sS_\Lambda(a)$. Since
$\sS_{\Lambda_C}(a)\in a+\sR_{C}$, $P_{\sN_{C}}\sS_{\Lambda_C}(a)
=P_{\sN_{C}}a.$ Substituting this identity into the preceding equivalence concludes the proof. 
\end{proof}

Armed with these results, we now show that, under the activity hypothesis, the two conditions in Theorem~\ref{thm:branching-criterion} follow from cone selection: accessibility of the strata (condition~(a)) and incompatibility of their branch quadratics (condition~(b)). We begin with condition~(a). The following lemma shows that $P_{\sN_C}\sQ_C(V)\in\relint N_{\Lambda_C}^{\sN_C}(\nu)$ for $V$ ranging over a relatively open cone in $\ker\sJ_C$ is sufficient to guarantee accessibility of the stratum $\relint\mathcal{S}_\nu$.

\begin{lemma}[Accessibility of a vertex cone]\label{lem:second-order-accessibility}
Fix a critical point $U^\star$ of $f_{\mathrm{LR}}$ satisfying $U^\star{U^\star}^\top=X^\star$, and suppose that $\sM=F_C^{-1}(\sE_C)$ is a $C^2$ identifiable manifold for $f_{\mathrm{LR}}$ at $U^\star$ for $0$. Under Assumption~\ref{ass:affine-cube}, suppose there exists a vertex $\nu$ of $\Lambda_C$, a linear subspace $\sW\subseteq\ker\sJ_{C}$, and a nonempty relatively open cone $\sC\subseteq \sW$ such that
\[
P_{\sN_{C}}\sQ_{C}(V)	\in\relint N_{\Lambda_C}^{\sN_{C}}(\nu),\quad \forall V\in \sC.
\]
Then, $ \relint\mathcal{S}_\nu$ is accessible from $\sM$ at $U^\star$ along $\sW$.
\end{lemma}

\begin{proof}
Let $\gamma$ be the chart in~\eqref{eq:general-chart}, and fix a unit vector $V\in\sC$. By Lemma~\ref{lem:tangent-clean-null-equality}, $V\in T_{U^\star}\sM=\ker\sJ_C$. Hence, using the homogeneity of $Z_{\sM}$ and the definition~\eqref{eq:clean-quadratic},
\begin{align*}
F_C(\gamma(tV))
&=t^2\bigl(\sQ_C(V)+\sJ_C[Z_{\sM}(V)]\bigr)+o(t^2)\\
&=:t^2\eta_C(V)+o(t^2).
\end{align*}
Here $\eta_C(V):=(\eta_{\sM}(V))_{i\in\Omega_C}$ is the clean block of the coefficient in~\eqref{eq:eta-main}.
For $t>0$, set $z_t:=t^{-2}F_C(\gamma(tV))$. Since $\gamma(tV)\in\sM=F_C^{-1}(\sE_C)$ and $\sE_C$ is a cone, we have $z_t\in\sE_C$. Moreover, $z_t\to\eta_C(V)$ as $t\downarrow0$. Proposition~\ref{prop:fan-structure} expresses $\sE_C$ as a finite union of closed polyhedral cones, so $\sE_C$ is closed and therefore $\eta_C(V)\in\sE_C$. On the other hand,
\[
\eta_C(V)-\sQ_C(V)=\sJ_C[Z_{\sM}(V)]\in\sR_C.
\]
Consequently,
\[
\eta_C(V)\in\bigl(\sQ_C(V)+\sR_C\bigr)\cap\sE_C.
\]
Proposition~\ref{prop:unique-slice} now yields
\begin{equation}\label{eq:eta-clean-selection}
\eta_C(V)=\sS_{\Lambda_C}(\sQ_C(V)).
\end{equation}
By the cone-selection hypothesis and Lemma~\ref{lem:branch-correspondence},
\[
\eta_C(V)\in\relint\mathcal{S}_\nu.
\]
Because the maximal cones of $\sE_C$ form a finite polyhedral fan, every point of $\sE_C$ sufficiently close to $\eta_C(V)$ also lies in $\relint\mathcal{S}_\nu$. Indeed, the relative boundary of $\mathcal{S}_\nu$ together with all other maximal cones is a finite union of closed polyhedral sets that does not contain $\eta_C(V)$. Since $z_t\in\sE_C$ and $z_t\to\eta_C(V)$, it follows that $z_t\in\relint\mathcal{S}_\nu$ for all sufficiently small $t>0$. Finally, relative interiors of cones are invariant under positive scaling, and hence
\[
F_C(\gamma(tV))=t^2z_t\in\relint\mathcal{S}_\nu
\]
for all sufficiently small $t>0$. Since $V\in\sC$ was arbitrary, Definition~\ref{def:accessible-stratum} applied to the reduced representation shows that $\relint\mathcal{S}_\nu$ is accessible from $\sM$ at $U^\star$ along $\sW$.
\end{proof}

Having established accessibility of the strata, it remains to verify that the branch quadratics are incompatible on them. As we show next, this incompatibility is in fact automatic: any two distinct strata necessarily carry incompatible branch quadratics.
To state the main result, define the \emph{boundary sign vector} $\sigma(\nu)\in\{-1,0,1\}^{m_C}$ by $\sigma(\nu)_i = \operatorname{sign}(\nu_i)$ on the saturated coordinates ($|\nu_i|=1$) and $\sigma(\nu)_i = 0$ on the free coordinates ($|\nu_i|<1$). Its support is the set of saturated coordinates and its zero set is $\sI_<(\nu)$. On the sign cone at $\nu$, it computes the $\ell_1$ norm,
\begin{equation}\label{eq:sigma-support}
\|z\|_1=\inp{\sigma(\nu)}{z}\qquad\text{for }z\in\relint \mathcal{S}_\nu.
\end{equation}

\begin{theorem}[Branching criterion via cone selection]\label{thm:spectrahedral-deterministic}
Fix a critical point $U^\star$ of $f_{\mathrm{LR}}$ satisfying $U^\star{U^\star}^\top=X^\star$. 
Under Assumption~\ref{ass:affine-cube}, suppose there exist two distinct vertices $\nu_1, \nu_2$ of $\Lambda_C$, a linear subspace $\sW\subseteq\ker\sJ_{C}$, and two nonempty relatively open cones $\sC_1, \sC_2\subseteq \sW$ such that
\[
P_{\sN_{C}}\sQ_{C}(V)	\in\relint N_{\Lambda_C}^{\sN_{C}}(\nu_j),\quad \forall V\in \sC_j\quad j=1,2.
\]
Then $f_{\mathrm{LR}}$ admits no $C^2$ identifiable manifold at $U^\star$ for $0\in \partial f(U^\star)$; that is, $U^\star$ is inactive.
\end{theorem}

\begin{proof}

By Theorem~\ref{thm:branching-criterion}, it suffices to show that the activity hypothesis implies accessibility of two strata and incompatibility of their branch quadratics. We therefore work under the activity hypothesis throughout the remainder of the proof.

Since $P_{\sN_C}\sQ_C(V) \in \relint N_{\Lambda_C}^{\sN_C}(\nu_j)$ for $V \in \sC_j$, Lemma~\ref{lem:second-order-accessibility} shows that the strata $\sZ_j := \relint \mathcal{S}_{\nu_j}$, $j = 1,2$, are accessible. It remains to prove the incompatibility of the branch quadratics on $\sZ_1$ and $\sZ_2$. Fix a piece $\tau_j$ active on $\sZ_j$, i.e.\ $\tau_j \in \sI_{\sZ_j}$, and let $\zeta_{\sZ_j,\sW} := \zeta_{\tau_j}|_{\sW}$ denote the corresponding branch quadratic, as in Lemma~\ref{lem:stratum-quadratic}, with $\zeta_{\tau_j}$ defined in~\eqref{eq:zeta-main}. Since $f_{\mathrm{LR}} = g \circ F$ with $g$ a finite maximum of linear functions, the Hessian term for $g_{\tau_j}$ in~\eqref{eq:zeta-main} vanishes, so
\[
\zeta_{\tau_j}(V)=\inp{\tau_j}{\eta_{\sM}(V)}=\inp{\tau_{j,C}}{\eta_{C}(V)}+ \underbrace{\inp{-\sign(\varepsilon_N)}{\eta_N(V)}}_{:=q_N(V)},
\]
where $\tau_{j,C} := (\tau_j)_{i \in \Omega_C}$ and $\eta_C := (\eta_{\sM})_{i \in \Omega_C}$ are the restrictions of $\tau_j$ and $\eta_{\sM}$ to the clean indices, and $\eta_N := (\eta_{\sM})_{i \in \Omega_N}$ is the restriction to the noisy indices. Note that $q_N$ is independent of the choice of active piece $\tau_j$. By~\eqref{eq::Fc-gamma} and Proposition~\ref{prop:unique-slice},
\[
\eta_C(V) = \sQ_C(V)+\sJ_C[Z] = \sS_{\Lambda_C}(\sQ_C(V)),
\]
The assumption $P_{\sN_{C}}\sQ_{C}(V)	\in\relint N_{\Lambda_C}^{\sN_{C}}(\nu_j)$ for $V\in \sC_j$ combined with Lemma~\ref{lem:branch-correspondence} implies $\sS_{\Lambda_C}(\sQ_C(V))\in \relint\mathcal{S}_{\nu_j}$ for $V\in \sC_j$, and hence $\eta_C(V) \in \relint\mathcal{S}_{\nu_j}$ for $V\in \sC_j$. In particular $\eta_C(V)$ vanishes on the free coordinates of $\nu_j$, while $\tau_{j,C}$ agrees with $\sigma(\nu_j)$ on the saturated ones. Consequently,
\[
\inp{\tau_{j,C}}{\eta_C(V)} = \inp{\sigma(\nu_j)}{\eta_C(V)} =: \psi_{\nu_j}(V), \qquad \forall\, V \in \sC_j,\ j=1,2,
\]
and therefore
\[
\zeta_{\tau_j}(V) = q_N(V) + \psi_{\nu_j}(V), \qquad \forall\, V \in \sC_j,\ j = 1,2.
\]
Since $\zeta_{\sZ_j,\sW}$ and $\zeta_{\tau_j}$ are both quadratic forms on $\sW$ agreeing on the nonempty relatively open cone $\sC_j$, Lemma~\ref{lem:quadratic-open-cone} shows they agree on all of $\sW$:
\[
\zeta_{\sZ_j,\sW}(V) = q_N(V) + \psi_{\nu_j}(V), \qquad \forall\, V \in \sW,\ j = 1,2.
\]
Hence $\zeta_{\sZ_1,\sW} - \zeta_{\sZ_2,\sW} = \psi_{\nu_1} - \psi_{\nu_2}$ on $\sW$, and it suffices to show $\psi_{\nu_1} \not\equiv \psi_{\nu_2}$ on $\sW$. For $V \in \sC_1$, since $\eta_C(V) \in \relint \mathcal{S}_{\nu_1}$, its support equals the saturated coordinates of $\nu_1$ with matching sign, so
\[
\psi_{\nu_1}(V) = \inp{\sigma(\nu_1)}{\eta_C(V)} = \|\eta_C(V)\|_1, \qquad \forall\, V \in \sC_1.
\]
Since $\sigma(\nu_2) \in \{-1,0,+1\}^{m_C}$, by H\"older's inequality,
\[
\psi_{\nu_2}(V) = \inp{\sigma(\nu_2)}{\eta_C(V)} \leq \|\eta_C(V)\|_1, \qquad \forall\, V \in \sC_1,
\]
with equality only if $\sigma(\nu_2)_i = \sign(\eta_C(V)_i)$ for every $i \in \supp(\eta_C(V))$, that is, only if every coordinate saturated by $\nu_1$ is also saturated by $\nu_2$ with the same sign, forcing $\eta_C(V) \in \mathcal{S}_{\nu_2}$. But $\mathcal{S}_{\nu_1}$ and $\mathcal{S}_{\nu_2}$ are distinct maximal cones of the polyhedral fan due to Proposition~\ref{prop:fan-structure}, so they meet only along a common proper face, and $\relint\mathcal{S}_{\nu_1}$ is disjoint from $\mathcal{S}_{\nu_2}$. Since $\eta_C(V) \in \relint\mathcal{S}_{\nu_1}$ for $V \in \sC_1$, we conclude $\eta_C(V) \notin \mathcal{S}_{\nu_2}$, and thus $\psi_{\nu_2}(V) < \psi_{\nu_1}(V)$ strictly for all $V \in \sC_1$.
As $\sC_1 \subseteq \sW$ is nonempty, this shows $\psi_{\nu_1} \not\equiv \psi_{\nu_2}$ on $\sW$. Thus the activity hypothesis implies the two incompatible accessible branches required by Theorem~\ref{thm:branching-criterion}, completing the proof.
\end{proof}

Theorem~\ref{thm:spectrahedral-deterministic} reduces the branching criterion to a cone-selection test in $\sN_{\mathrm c}$. The next subsection verifies that this test succeeds with high probability under appropriate random models.

\subsection{Probabilistic Analysis}\label{subsec:high-prob-results}

In this subsection, we show that the cone-selection test—and hence the branching criterion—succeeds with high probability for robust low-rank recovery under the random models considered here. Before turning to the formal analysis, we develop geometric intuition for why this should occur. This intuition applies directly to nonnegative sparse regression and symmetric matrix sensing. Asymmetric matrix sensing requires an analogous but slightly sharper argument, which we present in Appendix~\ref{app::cone-select-asym-hp}.

Suppose for now that the true solution $U^\star$ is a critical point of $f_{\mathrm{LR}}$ and that Assumption~\ref{ass:affine-cube} (the interior and nondegeneracy conditions) holds. Suppose further, purely for the sake of intuition, that there exist $q\geq 2$ distinct directions $V_1,\ldots,V_q\in\ker\sJ_C$ whose projected second-order residuals $P_{\sN_C}\sQ_C(V_1),\ldots,P_{\sN_C}\sQ_C(V_q)$ are spherically symmetric on $\sN_C$, i.e., their normalized directions are uniformly distributed on the unit sphere. Under this idealized randomness assumption, the cone-selection test succeeds unless all of these projected residuals lie in the relative interior of a single normal cone $N_{\Lambda_C}^{\sN_C}(\nu)$ for some vertex $\nu$ of $\Lambda_C$.\footnote{By Theorem~\ref{thm:spectrahedral-deterministic}, the cone-selection test in fact requires $P_{\sN_C}\sQ_C(V)\in\relint N_{\Lambda_C}^{\sN_C}(\nu)$ not just for a single $V$, but for all $V$ in some relatively open cone $\sC$. We suppress this distinction here for the sake of intuition.}

To estimate the probability of this event, first recall from Proposition~\ref{prop:fan-structure} that the clean multiplier polytope $\Lambda_C$ is full-dimensional in $\sN_C$. Hence its vertex normal cones $N_{\Lambda_C}^{\sN_C}(\nu)$, $\nu\in\operatorname{Vert}(\Lambda_C)$ are pointed~\citep[Theorem~2.3.2]{cox2024toric}. A pointed cone is always contained in some half-space. Consequently, for any two indices $1\leq i<j\leq q$, the probability that both $P_{\sN_C}\sQ_C(V_i)$ and $P_{\sN_C}\sQ_C(V_j)$ fall into the \emph{same} normal cone $N_{\Lambda_C}^{\sN_C}(\nu)$ is at most $1/2$. Consequently, the probability that all $q$ projected residuals land in the same cone is at most $2^{-(q-1)}$, and the cone selection test succeeds whenever this does not happen, giving
\[
\mathbb{P}(\text{cone-selection test succeeds}) \geq 1 -2^{-(q-1)}.
\]
This intuition is formalized in the following proposition. 
\begin{proposition}[Spherically symmetric projected residuals imply successful cone selection]
\label{lem:isotropic-branch-occupancy}
Fix a critical point $U^\star$ of $f_{\mathrm{LR}}$ satisfying $U^\star{U^\star}^\top=X^\star$. Suppose that Assumption~\ref{ass:affine-cube} holds and that
$\dim\sN_C\geq 1$. Let $q\geq 2$, and suppose there exist distinct
nonzero directions $V_1,\ldots,V_q\in\ker\sJ_C$ such that, conditional
on $\sJ_C$ and $\Lambda_C$, the projected residuals $P_{\sN_C}\sQ_C(V_1),\ldots,P_{\sN_C}\sQ_C(V_q)$
are independent and have continuous, spherically symmetric
distributions on $\sN_C$. Then, conditional on $\sJ_C$ and
$\Lambda_C$, with probability at least $1-2^{-(q-1)},$
the cone-selection test in
Theorem~\ref{thm:spectrahedral-deterministic} succeeds. Specifically,
there exist two distinct vertices $\nu_1,\nu_2$ of $\Lambda_C$, a
linear subspace $\sW\subseteq\ker\sJ_C$, and two nonempty relatively
open cones $\sC_1,\sC_2\subseteq\sW$ such that
\[
P_{\sN_C}\sQ_C(V)
\in
\relint N_{\Lambda_C}^{\sN_C}(\nu_j)
\qquad
\text{for every }V\in\sC_j,\quad j=1,2.
\]
\end{proposition}
The proof is presented in Appendix~\ref{app:isotropic-branch-occupancy}.
For this result to materialize, three conditions must hold: (1) the true solution $U^\star$ is a critical point of $f_{\mathrm{LR}}$; (2) Assumption~\ref{ass:affine-cube} (nonempty relative interior and nondegeneracy of $\Lambda_C$) holds; and (3) the projected residuals $P_{\sN_C}\sQ_C(V_1),\ldots,P_{\sN_C}\sQ_C(V_q)$ are independent and spherically symmetric. 

Condition~(1) has already been established in the literature for all three models considered here. We now show that condition~(2) can be verified through a \emph{whitening procedure}. Define the linearized design map at $U^\star$ by
\[
\sT({U^\star}):\mathbb{R}^{n\times k}\to\mathbb{S}^n,
\qquad
\sT({U^\star})[V]
:=\sL\left(U^\star V^\top+V{U^\star}^\top\right).
\]
When the base point $U^\star$ is clear from context, we suppress it and simply write $\sT$. The definitions of the clean and noisy Jacobians then give $\sJ_C=\sA_C\circ\sT$ and $\sJ_N=\sA_N\circ\sT$, and hence, $\sJ_C^*=\sT^*\circ\sA_C^*$ and $\sJ_N^*=\sT^*\circ\sA_N^*$.
In particular, both adjoint Jacobians take values in $\Im\sT^*$. We write
\[
\dim\Im\sT^*=\dim\Im\sT=:d_\star.
\]
\begin{assumption}[Whitening operator]\label{ass:whitening}
There exists a linear isomorphism $\sH:\Im\sT^*\to\mathbb{R}^{d_\star}$, called the \emph{whitening operator}, such that the whitened adjoint Jacobians $\widetilde{\sJ}_C^* = \sH \circ \sJ_C^*$ and $\widetilde{\sJ}_N^* = \sH \circ \sJ_N^*$ satisfy
\[
\widetilde{\sJ}_C^*[e_i] \sim N(0,I_{d_\star}) \ \text{ i.i.d.\ for } i\in\Omega_C,
\qquad
\widetilde{\sJ}_N^*[e_j] \sim N(0,I_{d_\star}) \ \text{ i.i.d.\ for } j\in\Omega_N,
\]
where $e_i$ denotes the $i$-th standard basis vector of $\mathbb{R}^{m_C}$ or $\mathbb{R}^{m_N}$, as appropriate.
\end{assumption}
Our next lemma shows that this whitening operator exists for Gaussian measurement operators $\sA$, which underlie symmetric and asymmetric matrix sensing as well as nonnegative sparse regression.
\begin{lemma}[Existence of a whitening operator]
\label{lem:whitening-operator}
Fix a ground-truth factor $U^\star$ and let
$\sT$ be its associated linearized design map. Suppose that
$A_1,\ldots,A_m$ are independent of the outlier vector $\varepsilon$.
Suppose further that there exist a linear isometry
\[
\iota_\star:\Im\sT\longrightarrow\mathbb{R}^{d_\star},
\qquad
d_\star:=\dim\Im\sT,
\]
and a positive-definite matrix $\Sigma\succ 0$
such that
\[
\iota_\star\!\left(P_{\Im\sT}A_1\right),\ldots,
\iota_\star\!\left(P_{\Im\sT}A_m\right)
\stackrel{\mathrm{i.i.d.}}{\sim}N(0,\Sigma).
\]
Then Assumption~\ref{ass:whitening} holds. In particular, a valid
whitening operator is
\[
\sH
:=
\Sigma^{-1/2}
\circ\iota_\star
\circ
\left(\left.\sT^*\right|_{\Im\sT}\right)^{-1}
:
\Im\sT^*\longrightarrow\mathbb{R}^{d_\star}.
\]
Moreover, the hypothesis above holds for the Gaussian measurement
ensembles underlying nonnegative sparse regression, symmetric matrix
sensing, and asymmetric matrix sensing.
\end{lemma}
Its proof is presented in Appendix~\ref{app::whitening-operator}.
To see the utility of the whitening procedure, note that since $\sH$ is an isomorphism from $\Im\sT^*$ onto $\mathbb{R}^{d_\star}$ and both $\Im\sJ_C^*$ and $\Im\sJ_N^*$ belong to $\Im\sT^*$, applying $\sH$ preserves the solution set defining $\Lambda_C$:
\begin{align*}
\Lambda_C
&= \left\{v\in[-1,1]^{m_C} : \sJ_C^*[v] = \sJ_N^*[\sign(\varepsilon_N)]\right\} \\
&= \left\{v\in[-1,1]^{m_C} : \sH\circ\sJ_C^*[v] = \sH\circ\sJ_N^*[\sign(\varepsilon_N)]\right\} \\
&= \left\{v\in[-1,1]^{m_C} : \widetilde{\sJ}_C^*[v] = \widetilde{\sJ}_N^*[\sign(\varepsilon_N)]\right\}.
\end{align*}
The latter characterization makes it considerably easier to establish the nonempty relative interior and nondegeneracy of $\Lambda_C$, owing to the isotropic Gaussian structure of $\widetilde{\sJ}_C^*$ and $\widetilde{\sJ}_N^*$, as we show next.

\begin{proposition}[Interior and nondegeneracy under whitening]\label{prop:interior-point-event}
Suppose Assumption~\ref{ass:whitening} holds, $m\geq c_1 d_\star$ for some constant $c_1>2$, and $m_C>m/2$. Then, with probability at least $1-\exp(-c_2 d_\star)$ for some constant $c_2>0$, Assumption~\ref{ass:affine-cube} holds:
\begin{enumerate}
	\item[(I)] $\Lambda_C$ meets the interior of the cube, i.e., $\Lambda_C\cap(-1,1)^{m_C}\neq\emptyset$;
	\item[(N)] for every vertex $\nu$ of $\Lambda_C$, the restriction $\sJ_{C,\sI_<(\nu)}$ of $\sJ_C$ to the free coordinates of $\nu$ has full rank, i.e., $\rank\sJ_{C,\sI_<(\nu)} = \rank\sJ_C$.
\end{enumerate}
\end{proposition}
The proof is given in Appendix~\ref{app::nondeg-hp}. With Proposition~\ref{prop:interior-point-event} in hand, the proofs of Theorems~\ref{thm:nonnegative-sparse}, \ref{thm:sym-MS}, and~\ref{thm:asym-MS} separate into common and model-specific steps. For a fixed ground-truth factor $U^\star$, criticality follows from \citet[Corollaries~1 and~2]{ma2025can} and the corresponding diagonal specialization for nonnegative sparse regression. Lemma~\ref{lem:whitening-operator} and Proposition~\ref{prop:interior-point-event} then verify Assumption~\ref{ass:affine-cube} with high probability.

For nonnegative sparse regression and symmetric matrix sensing, Proposition~\ref{lem:isotropic-branch-occupancy} reduces the remaining work to constructing distinct nonzero directions $V_1,\ldots,V_q\in\ker\sJ_C$ such that, conditional on $\sJ_C$ and $\Lambda_C$, the projected residuals $P_{\sN_C}\sQ_C(V_1),\ldots,P_{\sN_C}\sQ_C(V_q)$ are independent and spherically symmetric on $\sN_C$. The asymmetric model uses the same common steps, but its final cone-selection argument relies on a pair of complementary directions. Table~\ref{tab:robust-model-dictionary} summarizes these model-specific constructions. Each displayed direction lies in $\ker\sJ_C$ because it acts either on an off-support coordinate or along a null direction of the ground-truth factor. Once inactivity is established for a fixed factor, Lemma~\ref{lem:orthogonal-invariance} extends it to all (balanced) ground-truth factors. The full model-specific proofs are given in Appendices~\ref{app::cone-select-sp-hp}, \ref{app::cone-select-sym-hp}, and~\ref{app::cone-select-asym-hp}.

\begin{center}
\begin{minipage}{\linewidth}
	\centering
	\small
	\begin{tabular}{@{}>{\raggedright\arraybackslash}p{.21\linewidth}>{\raggedright\arraybackslash}p{.17\linewidth}>{\raggedright\arraybackslash}p{.54\linewidth}@{}}
		\toprule
		Model & $d_\star$ & Kernel directions used for cone selection \\
		\midrule
		Nonnegative sparse regression & $s$ & Let $S=\supp(u^\star)$ and enumerate $S^c=\{i_a\}_{a=1}^{n-s}$. Set $V_a=e_{i_a}$. \\
		Symmetric matrix sensing & $nr-\frac{r(r-1)}2$ & Choose a unit vector $w\in\ker U^\star$ and an orthonormal basis $\{u_j\}_{j=1}^{n-r}$ of $\col(U^\star)^\perp$. Set $V_j=u_jw^\top$. \\
		Asymmetric matrix sensing & $r(n_1+n_2-r)$ & Choose unit vectors $w\in\ker U_1^\star=\ker U_2^\star$, $u\in\col(U_1^\star)^\perp$, and $v\in\col(U_2^\star)^\perp$. Set $V_1=(uw^\top,0)$ and $V_2=(0,vw^\top)$. \\
		\bottomrule
	\end{tabular}
	\captionof{table}{Model-specific dimensions and directions used in the cone-selection argument. In each row, the displayed directions belong to $\ker\sJ_C$.}
	\label{tab:robust-model-dictionary}
\end{minipage}
\end{center}

\section{Conclusion}\label{sec:conclusion}

We studied when critical points of finite-max composite problems fail to admit a $C^2$ identifiable manifold. After characterizing the locally minimal identifiable set, we derived a branching criterion based on a second-order incompatibility: if two strata accessible along a common family of tangent directions induce different branch quadratics, then no identifiable manifold can exist. This criterion detects inactivity in settings where first-order conditions are inconclusive.

Our applications show that inactivity is not merely pathological. It arises almost surely at common interpolators of random finite-scenario minimax problems and with high probability in overparameterized robust low-rank recovery, encompassing both global minimizers and strict saddles. These results suggest that finite identification and the resulting smooth local dynamics cannot be taken for granted in modern nonsmooth problems. Developing algorithmic guarantees that do not rely on identifiable manifolds, and identifying weaker geometric structures that remain available at inactive points, are natural directions for future work.

\section*{Acknowledgments}
This research is supported, in part, by the NSF CAREER grant CCF-2337776 and ONR grant N00014-26-1-2074.

\bibliography{ref}

\appendix
\section{Derivation of the Identifiable Set in Example~\ref{ex:l1-outer}}\label{app::example-l1}
Our goal is to show that the minimal identifiable set
\begin{equation}\nonumber
\sM_{\bar v} := \bigcup_{\lambda\in\Lambda_{\bar v}}\{z\in\RR^m : \supp\lambda\subseteq I(z)\},
\qquad
\Lambda_{\bar v} := \Bigl\{\lambda\in\Delta_{\sI} : \bar v = \textstyle\sum_\sigma\lambda_\sigma\sigma\Bigr\}
\end{equation}
coincides with
\begin{equation}\nonumber
\mathcal{S}_{\bar v} := \{z\in\RR^m : z_i=0\ \text{if}\ |\bar v_i|<1,\ z_i\geq 0\ \text{if}\ \bar v_i=1,\ z_i\leq 0\ \text{if}\ \bar v_i=-1\}.
\end{equation}

Fix a coordinate $i$. Since $\lambda\in\Delta_{\sI}$ is a probability distribution over $\sI=\{\pm1\}^m$, the constraint $\bar v_i = \sum_\sigma\lambda_\sigma\sigma_i$ expresses $\bar v_i$ as a weighted average of $\sigma_i\in\{-1,+1\}$. Setting $\alpha_i^+ := \sum_{\sigma\in\sI_i^+}\lambda_\sigma$, $\alpha_i^- := \sum_{\sigma\in\sI_i^-}\lambda_\sigma$, and $\sI_i^\pm := \{\sigma\in\sI:\sigma_i=\pm1\}$, 
we have $\alpha_i^++\alpha_i^-=1$ and $\bar v_i = \alpha_i^+-\alpha_i^-$, so
\begin{equation}\label{eq:supp_and_sign}
|\bar v_i| = 1 \iff \supp\lambda\subseteq\sI_i^{\sign\bar v_i}, \quad |\bar v_i| < 1 \iff \supp\lambda\cap\sI_i^+\neq\varnothing\ \text{and}\ \supp\lambda\cap\sI_i^-\neq\varnothing.
\end{equation}

\textbf{Proof of $\sM_{\bar v}\subseteq\mathcal{S}_{\bar v}$.} Suppose $z\in\sM_{\bar v}$ and pick $\lambda\in\Lambda_{\bar v}$ with $\supp\lambda\subseteq I(z)$. For coordinate $i$, if $z_i>0$ then every active sign pattern has $\sigma_i=+1$, so $\supp\lambda\subseteq\sI_i^+$, which by~\eqref{eq:supp_and_sign} forces $\bar v_i=1$. Likewise, $z_i<0$ forces $\bar v_i=-1$. Finally, if $|\bar v_i|<1$ then by~\eqref{eq:supp_and_sign} the support of $\lambda$ must intersect both $\sI_i^+$ and $\sI_i^-$, which requires both signs to be active at $z_i$, hence $z_i=0$. Thus $z\in\mathcal{S}_{\bar v}$.

\textbf{Proof of $\mathcal{S}_{\bar v}\subseteq\sM_{\bar v}$.} Suppose $z\in\mathcal{S}_{\bar v}$. Since $\bar v\in[-1,1]^m=\conv\{\pm1\}^m$, there exists $\lambda\in\Lambda_{\bar v}$; we claim $\supp\lambda\subseteq I(z)$ for any such $\lambda$. Indeed, fix a sign pattern $\sigma\in\supp\lambda$ and a coordinate $i$. If $z_i\neq0$, then $\bar v_i=\sign(z_i)$ by definition of $\mathcal{S}_{\bar v}$, so~\eqref{eq:supp_and_sign} gives $\sigma_i=\sign(z_i)$, meaning $\sigma$ is active at $z_i$. If $z_i=0$, both signs are active at $z_i$ regardless of $\sigma_i$. Hence $\sigma\in I(z)$, so $\supp\lambda\subseteq I(z)$ and $z\in\sM_{\bar v}$.

\section{Omitted Proofs}
\subsection{Proof of Lemma~\ref{lem:orthogonal-invariance}}\label{app:orthogonal-invariance}
Define the lifted transformation
\[
\widetilde{\mathcal{T}}:\sX\times\RR\to\sX\times\RR,
\qquad
\widetilde{\mathcal{T}}(x,\alpha)
:=(\mathcal{T}x,\alpha).
\]
Since $\mathcal{T}$ is orthogonal and $h\circ\mathcal{T}=h$, the map $\widetilde{\mathcal{T}}$ is orthogonal and leaves the epigraph of
$h$ invariant: $\widetilde{\mathcal{T}}(\epi h)=\epi h$.  The limiting normal cone is equivariant under orthogonal transformations. Hence
\[
N_{\epi h}\bigl(\mathcal{T}x,h(\mathcal{T}x)\bigr)
=
\widetilde{\mathcal{T}}\,
N_{\epi h}\bigl(x,h(x)\bigr).
\]
Since $\widetilde{\mathcal{T}}(v,-1)=(\mathcal{T}v,-1)$, the definition of the limiting subdifferential gives
\[
\begin{aligned}
v\in\partial h(x)
&\iff
(v,-1)\in N_{\epi h}\bigl(x,h(x)\bigr)\\
&\iff
(\mathcal{T}v,-1)
\in
N_{\epi h}\bigl(\mathcal{T}x,h(\mathcal{T}x)\bigr)\\
&\iff
\mathcal{T}v\in\partial h(\mathcal{T}x).
\end{aligned}
\]
Therefore, $\partial h(\mathcal{T}x) =\mathcal{T}\bigl(\partial h(x)\bigr)$. In particular, since $\mathcal{T}0=0$,
\[
0\in\partial h(x)
\iff
0\in\partial h(\mathcal{T}x),
\]
which proves the invariance of criticality.

We next prove the invariance of identifiable sets. Suppose that
$\sM$ is identifiable for $h$ at $x$ for $v$, and consider sequences $y_\nu\to\mathcal{T}x$ and $w_\nu\to\mathcal{T}v$ with $w_\nu\in\partial h(y_\nu).$
Set $x_\nu:=\mathcal{T}^{-1}y_\nu$ and $v_\nu:=\mathcal{T}^{-1}w_\nu$. By continuity of $\mathcal{T}^{-1}$, we have $x_\nu\to x$ and
$v_\nu\to v$. The subdifferential equivariance established above also gives $v_\nu\in\partial h(x_\nu).$
Identifiability of $\sM$ therefore implies $x_\nu\in\sM$ for all
sufficiently large $\nu$, and hence $y_\nu=\mathcal{T}x_\nu\in\mathcal{T}\sM$
for all sufficiently large $\nu$. Thus, $\mathcal{T}\sM$ is
identifiable for $h$ at $\mathcal{T}x$ for $\mathcal{T}v$. Applying
the same argument to $\mathcal{T}^{-1}$ proves the converse.

It remains to verify that the manifold and restricted-smoothness properties are also preserved. Suppose that $\sM$ is a $C^2$ manifold around $x$. Locally, it can be written as $\sM=\Phi^{-1}(0)$ for some $C^2$ map $\Phi$ whose derivative is surjective. Around $\mathcal{T}x$, we then have
\[
\mathcal{T}\sM
=
\bigl(\Phi\circ\mathcal{T}^{-1}\bigr)^{-1}(0).
\]
Because $\mathcal{T}^{-1}$ is a linear isomorphism, the derivative of $\Phi\circ\mathcal{T}^{-1}$ remains surjective. Hence $\mathcal{T}\sM$ is a $C^2$ embedded manifold. Moreover, $\mathcal{T}|_{\sM}:\sM\to\mathcal{T}\sM$ is a $C^\infty$ diffeomorphism, and the invariance of $h$ yields
\[
h|_{\mathcal{T}\sM}
=
\bigl(h|_{\sM}\bigr)
\circ
\bigl(\mathcal{T}|_{\sM}\bigr)^{-1}.
\]
Thus, $h|_{\sM}$ is $C^2$ if and only if $h|_{\mathcal{T}\sM}$ is $C^2$. It follows that $\sM$ is a $C^2$ identifiable manifold at $x$ for $v$ if and only if $\mathcal{T}\sM$ is a $C^2$ identifiable manifold at $\mathcal{T}x$ for $\mathcal{T}v$. The invariance of inactivity follows immediately.
\qed

\subsection{Proof of Proposition~\ref{prop:clean-reduction}}\label{app:clean-reduction}
Let $\widehat\Lambda
:=\{v\in\partial \|F(U^\star)\|_1:DF(U^\star)^*[v]=0\}$
be the full multiplier set at $U^\star$ and let $\widehat{\mathcal{S}}_v\subseteq \mathbb{R}^m$ be the sign cone \eqref{eq:Sy-section5}. By Proposition~\ref{prop:lmi-finite-max} applied to $g=\|\cdot\|_1$, the set
\[
\widehat\sM
:=\bigcup_{v\in\widehat\Lambda}F^{-1}(\widehat{\mathcal{S}}_v)
\]
is locally minimal identifiable for $f_{\mathrm{LR}}$ at $U^\star$ for $0\in \partial f_{\mathrm{LR}}(U^\star)$.
For corrupted indices $i\in\Omega_{N}$, $F_i(U^\star)=-\eps_i\ne0$. Hence, after shrinking to a neighborhood $\sO$ of $U^\star$, the corrupted signs remain fixed:
\[
\sign (F_N(U))=-\sign(\eps_N)
\qquad\text{for all }U\in\sO.
\]
Therefore the corrupted block of every multiplier $v\in\widehat\Lambda$ is fixed at $\sign (F_N(U))=-\sign(\eps_N)$. Writing $v=(\bar v,-\sign(\eps_N))$ with $\bar v\in\RR^{m_C}$, the criticality equation becomes
\[
\sJ_{C}^*[\bar v]-\sJ_N^*[\sign(\eps_N)]=0,
\]
so the clean block $\bar v$ lies in $\Lambda_C$ as defined in \eqref{eq:clean-multiplier-section5}. Since the corrupted signs are already pinned to $-\sign(\eps_N)$ on $\sO$, membership $F(U)\in\widehat{\mathcal{S}}_v$ depends only on the clean indices; as a result, $F(U)\in\widehat{\mathcal{S}}_v$ if and only if $F_C(U)\in\mathcal{S}_{\bar v}$. Consequently,
\[
\widehat\sM\cap\sO
=F_C^{-1}\left(\bigcup_{v\in\Lambda_C}\mathcal{S}_v\right)\cap\sO
=F_C^{-1}(\sE_C)\cap\sO,
\]
which is the claimed locally minimal identifiable set.\qed

\subsection{Proof of Lemma~\ref{lem:tangent-clean-null-equality}}\label{app::lemma-tangent}

\begin{proof}
We prove each inclusion separately.

\emph{Establishing $T_{U^\star}\sM\subseteq\ker\sJ_C$.}
Let $V\in T_{U^\star}\sM$ and, for an open interval $I$ containing $0$, let $\gamma_V:I\to\sM$ be a $C^1$ curve with $\gamma_V(0)=U^\star$ and $\dot\gamma_V(0)=V$. Since $\sM=F_C^{-1}(\sE_C)$, we have $F_C(\gamma_V(t))\in\sE_C$ for all $t\in I$. Since $\sE_C$ is a closed cone, dividing by $t>0$ preserves membership, so $\frac{1}{t}F_C(\gamma_V(t))\in\sE_C$. Letting $t\downarrow 0$ and using $F_C(U^\star)=0$ gives
\[
\sJ_C[V] = \lim_{t\downarrow 0}\frac{F_C(\gamma_V(t))-F_C(\gamma_V(0))}{t} =\lim_{t\downarrow 0}\frac{F_C(\gamma_V(t))-F_C(U^\star)}{t}  = \lim_{t\downarrow 0}\frac{F_C(\gamma_V(t))}{t} \in \sE_C.
\]
On the other hand, $\sJ_C[V]\in\Im\sJ_C=\sR_C$ by definition. By Proposition~\ref{prop:unique-slice} applied with $a=0$, the intersection $\sE_C\cap\sR_C=\{0\}$. Hence $\sJ_C[V]=0$, so $V\in\ker\sJ_C$.

\emph{Establishing $\ker\sJ_C\subseteq T_{U^\star}\sM$.}
Let $V\in\ker\sJ_C$. We must find a $C^1$ curve $\gamma_V:I\to\sM$ with $\gamma_V(0)=U^\star$ and $\dot\gamma_V(0)=V$, which by definition of $\sM$ requires $F_C(\gamma_V(t))\in\sE_C$ for all $t\in I$.
Consider the straight-line curve $t\mapsto U^\star+tV$. Since $F_C(U^\star)=0$ and $\sJ_C[V]=0$,  we have
\[
a_V(t) := F_C(U^\star+tV) = t^2\sQ_C(V),
\]
where $\sQ_C(V)$ is defined as \eqref{eq:clean-quadratic}.
This need not lie in $\sE_C$, so we correct it. By Proposition~\ref{prop:unique-slice}, for any $a\in\RR^{m_C}$ there is a unique point $\sS_{\Lambda_C}(a)\in\sE_C$ with $\sS_{\Lambda_C}(a)-a\in\sR_C$. Define the residual correction
\[
\Gamma(a) := \sS_{\Lambda_C}(a) - a \in \sR_C.
\]
We seek a correction $\Delta(t)=O(t^2)$ such that $\gamma_V(t):=U^\star+tV+\Delta(t)$ satisfies $F_C(\gamma_V(t))\in\sE_C$. Specifically, we write
\[
F_C(\gamma_V(t)) = a_V(t) + \sJ_C[\Delta(t)] + \rho_{V,\Delta}(t),
\]
where $\rho_{V,\Delta}(t):=F_C(U^\star+tV+\Delta(t))-a_V(t)-\sJ_C[\Delta(t)]$ is the remainder. Moreover, due to the $C^2$-smoothness of $F$, we obtain $\|\rho_{V,\Delta}(t)\| = O(|t|\|\Delta(t)\|+\|\Delta(t)\|^2)$

For $F_C(\gamma_V(t))$ to lie in $\sE_C$, we need
\[
\sJ_C[\Delta(t)] = \Gamma\bigl(a_V(t)+\rho_{V,\Delta}(t)\bigr).
\]
Since $\Gamma$ takes values in $\sR_C=\Im\sJ_C$, and $\sJ_C$ restricted to $\sL:=(\ker\sJ_C)^\perp$ is a bijection onto $\sR_C$, every right-hand side can be realized by a unique $\Delta(t)\in\sL$. This gives the fixed-point equation
\[
\Delta(t) = \underbrace{(\sJ_C|_\sL)^{-1}\Bigl(\Gamma\bigl(a_V(t)+\rho_{V,\Delta}(t)\bigr)\Bigr)}_{=:\sT_t(\Delta(t))}.
\]

We show $\sT_t$ has a fixed point in the ball $\mathbb{B}_t:=\{\Delta\in\sL:\|\Delta\|\leq Mt^2\}$ for sufficiently large $M>0$. By Proposition~\ref{prop:unique-slice} and Lemma~\ref{lem:branch-correspondence}, $\sS_{\Lambda_C}$ is piecewise linear, hence Lipschitz with $\Gamma(0)=0$. For $\Delta(t)\in\mathbb{B}_t$, the remainder satisfies $\rho_{V,\Delta}(t)=O(t^3)$, so $a_V(t)+\rho_{V,\Delta}(t)=O(t^2)$, and Lipschitzness of $\Gamma$ gives $\Gamma(a_V(t)+\rho_{V,\Delta}(t))=O(t^2)$. Choosing $M$ large enough and taking $|t|$ small enough, we have $\sT_t(\mathbb{B}_t)\subseteq \mathbb{B}_t$, that is, $\sT_t$ maps $\mathbb{B}_t$ to itself. To see that $\sT_t$ is a contraction on $\mathbb{B}_t$, for $\Delta_1(t),\Delta_2(t)\in\mathbb B_t$, the definition of
$\mathcal T_t$, the Lipschitzness of $\Gamma$, and $C^2$-smoothness of $F_C$ give
\[
\begin{aligned}
	\|\mathcal T_t(\Delta_1(t))-\mathcal T_t(\Delta_2(t))\|
	& = 
	O\Big(
	\|\rho_{V,\Delta_1}(t)-\rho_{V,\Delta_2}(t)\|\Big)  \\
	& = O\Big(\big(|t|+\|\Delta_1(t)\| + \|\Delta_2(t)\|\big) \|\Delta_1(t)-\Delta_2(t)\|\Big) \\
	&=O(|t|+Mt^2)\|\Delta_1(t)-\Delta_2(t)\|.
\end{aligned}
\]
which is strictly less than $\|\Delta_1(t)-\Delta_2(t)\|$ for small $|t|$. By the Banach fixed-point theorem, there exists $\Delta(t)\in\mathbb{B}_t$ with $\Delta(t)=\sT_t(\Delta(t))$ and $\|\Delta(t)\|\leq Mt^2=O(t^2)$, so $F_C(U^\star+tV+\Delta(t))\in\sE_C$ by construction.

We have constructed points $U^\star+tV+\Delta(t)\in\sM$ with $\Delta(t)=O(t^2)$, but we have not shown $\Delta(t)$ is smooth. Instead, we use the local defining map of $\sM$. Let $\Phi$ be a local $C^2$ defining map for $\sM$ near $U^\star$, so $\sM=\Phi^{-1}(0)$ locally. Since $U^\star+tV+\Delta(t)\in\sM$, we have $\Phi(U^\star+tV+\Delta(t))=0$. Taylor expanding and using $\|\Delta(t)\|=O(t^2)$ gives
\[
0 = t\,D\Phi(U^\star)[V] + D\Phi(U^\star)[\Delta(t)] + O(t^2) = t\,D\Phi(U^\star)[V] + O(t^2).
\]
Dividing by $t$ and letting $t\downarrow 0$ yields $D\Phi(U^\star)[V]=0$, which means $V\in T_{U^\star}\sM$. Since $V\in\ker\sJ_C$ was arbitrary, $\ker\sJ_C\subseteq T_{U^\star}\sM$.
Combining both steps, $T_{U^\star}\sM=\ker\sJ_C$.
\end{proof}

\subsection{Proof of Proposition~\ref{prop:unique-slice}}\label{app:unique-slice}
We first prove that the intersection $(a+\sR_{C})\cap \sE_C$ is non-empty. Due to Proposition~\ref{prop:fan-structure}, $\Lambda_C$ is a full-dimensional polytope in, and contained in, the affine subspace $\bar v + \sN_C$. Therefore, the union of the normal cones of $\Lambda_C$ at its vertices covers $\sN_C$, that is, $\sN_C\subseteq \bigcup_{\nu\in\operatorname{Vert}(\Lambda_C)}N_{\Lambda_C}(\nu)$ \citep[Proposition 2.3.6]{cox2024toric}. Hence $P_{\sN_{C}}a$ lies in $\sN_{C}\cap N_{\Lambda_C}(\nu)$ for at least one vertex $\nu$. By Proposition~\ref{prop:fan-structure}, $N_{\Lambda_C}(\nu)=\sR_{C}+\mathcal{S}_\nu$, so there are $r\in\sR_{C}$ and $s\in\mathcal{S}_\nu$ with $P_{\sN_{C}}a=r+s$. Since $a-P_{\sN_{C}}a\in\sR_{C}$, we have $s-a\in\sR_{C}$, and therefore $s\in(a+\sR_{C})\cap\sE_C$, establishing its nonemptiness. 

Next, we show that $(a+\sR_{C})\cap \sE_C$ is a unique point.
\begin{itemize}
\item \emph{Uniqueness within a vertex cone.} For a vertex $\nu$, we write $\sI_\nu:=\sI_<(\nu)$ for its free coordinates and $P_{\sI_\nu}$ as the projection onto those coordinates. Suppose $z,z'\in(a+\sR_{C})\cap \mathcal{S}_\nu$, then $z-z'\in\sR_{C}$. Moreover, since $\sJ_{{C},\sI_\nu} = P_{\sI_\nu}\circ \sJ_C$ has rank equal to $\dim\sR_{C}$, the restricted projection $P_{\sI_\nu}|_{\sR_C}:\sR_{C}\to\RR^{\sI_\nu}$ is injective. Since $z-z'\in\sR_{C}$ and $P_{\sI_\nu}(z-z')=0$, we have $z=z'$.
\item \emph{Uniqueness across vertex cones.} Take
\[
z_1\in(a+\sR_{C})\cap \mathcal{S}_{\nu_1},
\qquad
z_2\in(a+\sR_{C})\cap \mathcal{S}_{\nu_2}.
\]
Then $z_1-z_2\in\sR_{C}$, while $\nu_1-\nu_2\in\sN_{C}$, so
\[
0=\inp{z_1-z_2}{\nu_1-\nu_2}
=\bigl(\inp{z_1}{\nu_1}-\inp{z_1}{\nu_2}\bigr)+\bigl(\inp{z_2}{\nu_2}-\inp{z_2}{\nu_1}\bigr).
\]
Note that $z_1\in \mathcal{S}_{\nu_1}$; hence, due to the definition of $\sS_{\nu_1}$ in \eqref{eq:Sy-section5}, $\inp{z_1}{\nu_1}- \inp{z_1}{\nu_2}\geq 0$. Similarly,  $\inp{z_2}{\nu_2}-\inp{z_2}{\nu_1}\geq 0$. Combined with the above equality, $\inp{z_1}{\nu_1}=\inp{z_1}{\nu_2}$ and $\inp{z_2}{\nu_2}=\inp{z_2}{\nu_1}$. Therefore, $z_1,z_2\in \mathcal{S}_{\nu_1}\cap \mathcal{S}_{\nu_2}$, and in particular $z_1,z_2\in (a+\sR_{C})\cap \mathcal{S}_{\nu_1}$. From the uniqueness within a vertex cone, we have $z_1=z_2$. Finally $a-a'\in\sR_{C}$ gives $a+\sR_{C}=a'+\sR_{C}$, so $\sS_{\Lambda_C}$ depends only on $P_{\sN_{C}}a$.
\end{itemize}

It remains to show that $\sS_{\Lambda_C}$ is a two-sided inverse of $P_{\sN_C}|_{\sE_C}:\sE_C\to\sN_C$, and that it is positively homogeneous.
We first verify both compositions $\sS_{\Lambda_C}\circ P_{\sN_C}|_{\sE_C}$ and $P_{\sN_C}|_{\sE_C} \circ \sS_{\Lambda_C}$ are identity. Let $z\in\sE_C$. Since $z\in(z+\sR_C)\cap\sE_C$, uniqueness gives $\sS_{\Lambda_C}(z)=z$. Hence $\sS_{\Lambda_C}(P_{\sN_C}z)=\sS_{\Lambda_C}(z)=z$, since, as established earlier, $\sS_{\Lambda_C}(z)$ depends only on $P_{\sN_C}z$. Therefore, $\sS_{\Lambda_C}\circ P_{\sN_C}|_{\sE_C}$ is identity. Next, let $w\in\sN_C$. By definition, $\sS_{\Lambda_C}(w)\in(w+\sR_C)\cap\sE_C$, so $\sS_{\Lambda_C}(w)-w\in\sR_C$. Projecting onto $\sN_C$ and using $\sR_C\perp\sN_C$ gives $P_{\sN_C}\sS_{\Lambda_C}(w)=w$. Therefore, $P_{\sN_C}|_{\sE_C}\circ\sS_{\Lambda_C}$ is identity.

To establish positive homogeneity, note that both $\sE_C$ and $\sR_C$ are cones, hence invariant under scaling by $\lambda>0$. Thus $\lambda\sS_{\Lambda_C}(a)\in(\lambda a+\sR_C)\cap\sE_C$, and uniqueness gives $\sS_{\Lambda_C}(\lambda a)=\lambda\sS_{\Lambda_C}(a)$. At $\lambda=0$, we have $\sS_{\Lambda_C}(0)=0$ since $0\in\sR_C\cap\sE_C$. This completes the proof.\qed

\subsection{Proof of Proposition~\ref{lem:isotropic-branch-occupancy}}\label{app:isotropic-branch-occupancy}
Let $\mathcal O_q$ be the event that there are distinct indices $\ell_1,\ell_2\in[q]$ and distinct vertices $\nu_1,\nu_2\in\operatorname{Vert}(\Lambda_C)$ such that
\[
P_{\sN_C}\sQ_C(V_{\ell_j})
\in\relint N_{\Lambda_C}^{\sN_C}(\nu_j),
\qquad j=1,2.
\]
We first show that the following bound holds almost surely
\begin{equation}\label{eq::Oq}
\PP\bigl(\mathcal O_q\mid\sJ_C,\Lambda_C\bigr)
\ge 1-2^{-(q-1)}.
\end{equation}
By Proposition~\ref{prop:fan-structure}, $\Lambda_C$ is full-dimensional in its affine space. Hence its vertex normal cones $N_{\Lambda_C}^{\sN_C}(\nu)$, $\nu\in\operatorname{Vert}(\Lambda_C)$, are full-dimensional and pointed, cover $\sN_C$, and have pairwise disjoint relative interiors~\citep[Theorem~2.3.2 and Proposition~2.3.6]{cox2024toric}. Their boundaries therefore have spherical measure zero.
For each $\nu\in\operatorname{Vert}(\Lambda_C)$, set
\[
\alpha_\nu
:=\PP\left\{
P_{\sN_C}\sQ_C(V_1)
\in\relint N_{\Lambda_C}^{\sN_C}(\nu)
\,\middle|\,\sJ_C,\Lambda_C\right\}.
\]
Spherical symmetry and continuity identify $\alpha_\nu$ with the normalized spherical measure of $N_{\Lambda_C}^{\sN_C}(\nu)$, so
\[
\sum_{\nu\in\operatorname{Vert}(\Lambda_C)}\alpha_\nu=1,
\qquad
0<\alpha_\nu\le\frac12.
\]
Indeed, $0<\alpha_{\nu}$ follows from the full-dimensionality of each normal cone, while $\alpha_\nu\leq \frac{1}{2}$ follows from pointedness: the spherical section of $N_{\Lambda_C}^{\sN_C}(\nu)$ is disjoint from its antipodal image, and the two sets have the same spherical measure.

Up to the boundary-null event, $\mathcal O_q^c$ is the event that all $q$ residuals lie within the same $\relint N_{\Lambda_C}^{\sN_C}(\nu)$ for some $\nu\in\operatorname{Vert}(\Lambda_C)$. Conditional independence therefore gives
\[
\begin{aligned}
\PP\bigl(\mathcal O_q^c\mid\sJ_C,\Lambda_C\bigr)
&=\sum_{\nu\in\operatorname{Vert}(\Lambda_C)}\alpha_\nu^q
\le\left(\max_{\nu\in\operatorname{Vert}(\Lambda_C)}\alpha_\nu\right)^{q-1}
\sum_{\nu\in\operatorname{Vert}(\Lambda_C)}\alpha_\nu
\le 2^{-(q-1)}.
\end{aligned}
\]
Taking the complementary event proves \eqref{eq::Oq}.

It remains to verify that this event implies the success of the cone-selection test of Theorem~\ref{thm:spectrahedral-deterministic}. Choose the distinct indices $\ell_1,\ell_2\in[q]$ and distinct vertices $\nu_1,\nu_2$ supplied by the event $\mathcal O_q$, so that
\[
P_{\sN_C}\sQ_C(V_{\ell_j})
\in\relint N_{\Lambda_C}^{\sN_C}(\nu_j),
\qquad j=1,2,
\]
and set $\sW:=\Span\{V_{\ell_1},V_{\ell_2}\}$. For $j=1,2$, define
\[
\sC_j:=\left\{V\in\sW\setminus\{0\}:
P_{\sN_C}\sQ_C(V)
\in\relint N_{\Lambda_C}^{\sN_C}(\nu_j)\right\}.
\]
Each $\sC_j$ is nonempty because it contains $V_{\ell_j}$. It is relatively open in $\sW\setminus\{0\}$ because $V\mapsto P_{\sN_C}\sQ_C(V)$ is continuous and $\relint N_{\Lambda_C}^{\sN_C}(\nu_j)$ is open in $\sN_C$. It is also a cone because this map is homogeneous of degree two. Thus $\sW,\sC_1,\sC_2,\nu_1,\nu_2$ satisfy the conditions of Theorem~\ref{thm:spectrahedral-deterministic}, thereby completing the proof.\qed

\subsection{Proof of Lemma~\ref{lem:whitening-operator}}\label{app::whitening-operator}
Since $\ker\sT^*=(\Im\sT)^\perp$,
we have
\[
\ker\!\left(\left.\sT^*\right|_{\Im\sT}\right)
=
\Im\sT\cap(\Im\sT)^\perp
=
\{0\}.
\]
Thus $\left.\sT^*\right|_{\Im\sT}$ is injective. Moreover,
\[
\dim\Im\sT
=
\rank\sT
=
\rank\sT^*
=
\dim\Im\sT^*,
\]
so $\left.\sT^*\right|_{\Im\sT}:
\Im\sT\longrightarrow\Im\sT^*$
is a linear isomorphism. Since $\iota_\star$ and
$\Sigma^{-1/2}$ are also linear isomorphisms, so is $\sH$.
For every $A\in\mathbb{S}^n$, the orthogonal decomposition $A=P_{\Im\sT}A+P_{(\Im\sT)^\perp}A$ and the identity $(\Im\sT)^\perp=\ker\sT^*$ give $\sT^*[A]=\sT^*\!\left[P_{\Im\sT}A\right].$ Consequently,
\[
\left(\left.\sT^*\right|_{\Im\sT}\right)^{-1}
\sT^*[A]
=
P_{\Im\sT}A.
\]
In particular, for every measurement index $i$,
\[
\begin{aligned}
\sH\bigl(\sT^*[A_i]\bigr)
=
\Sigma^{-1/2}
\iota_\star
\left(\left.\sT^*\right|_{\Im\sT}\right)^{-1}
\sT^*[A_i] =
\Sigma^{-1/2}
\iota_\star\!\left(P_{\Im\sT}A_i\right).
\end{aligned}
\]
The hypothesis therefore yields
\[
\sH\bigl(\sT^*[A_1]\bigr),\ldots,
\sH\bigl(\sT^*[A_m]\bigr)
\stackrel{\mathrm{i.i.d.}}{\sim}
N(0,I_{d_\star}).
\]
Recall that $\sJ_C^*[e_i]=\sT^*[A_i]$ for $i\in\Omega_C$ and $\sJ_N^*[e_j]=\sT^*[A_j]$ for $j\in\Omega_N$.
Conditional on $\varepsilon$, the index sets $\Omega_C$ and
$\Omega_N$ are fixed. Since $\varepsilon$ is independent of the
measurement matrices, conditioning on $\varepsilon$ does not alter
their joint distribution. Hence
\[
\bigl\{\sH\circ\sJ_C^*[e_i]:i\in\Omega_C\bigr\}
\quad\text{and}\quad
\bigl\{\sH\circ\sJ_N^*[e_j]:j\in\Omega_N\bigr\}
\]
are independent standard Gaussian vectors in
$\mathbb{R}^{d_\star}$. This proves
Assumption~\ref{ass:whitening}.

It remains to verify the final assertion. In each of the three models,
the measurement matrices are isotropic Gaussian random elements,
up to a deterministic scaling, on a linear subspace $\sG$ containing
$\Im\sT$:

\begin{itemize}
\item for nonnegative sparse regression, $\sG$ is the subspace of
diagonal matrices and $A_i=\diag(a_i)$ is standard Gaussian on
$\sG$;

\item for symmetric matrix sensing, $\sG=\mathbb{S}^n$ and
$A_i\sim\mathrm{GOE}(n)$ is standard Gaussian on $\mathbb{S}^n$
under the Frobenius inner product;

\item for asymmetric matrix sensing, $\sG$ is the subspace of
symmetric matrices with zero diagonal blocks, and
\[
A_i=\frac12
\begin{bmatrix}
	0&\Gamma_i\\
	\Gamma_i^\top&0
\end{bmatrix}
\]
is Gaussian on $\sG$ with covariance
$\frac12 I_{\sG}$.
\end{itemize}

Since $\Im\sT\subseteq\sG$, orthogonal projection onto $\Im\sT$
shows that
\[
\iota_\star\!\left(P_{\Im\sT}A_i\right)
\stackrel{\mathrm{i.i.d.}}{\sim}
N(0,cI_{d_\star}),
\]
where $c=1$ for nonnegative sparse regression and symmetric matrix
sensing, and $c=\frac12$ for asymmetric matrix sensing. Thus the
hypothesis holds with $\Sigma=cI_{d_\star}$ in all three cases.\qed

\subsection{Proof of Proposition~\ref{prop:interior-point-event}}\label{app::nondeg-hp}

\emph{Proof of Condition (I).} 
Write $h:=\widetilde{\sJ}_N^*[\sign(\varepsilon_N)]$ and $\sZ_C:=\{\widetilde{\sJ}_C^*[v] : v\in[-1,1]^{m_C}\}$. Since $\Lambda_C=\{v\in[-1,1]^{m_C}:\widetilde{\sJ}_C^*[v]=h\}$, showing $\Lambda_C\cap(-1,1)^{m_C}\neq\emptyset$ reduces to showing $h$ lies in the interior of $\sZ_C$. It suffices to exhibit $\tau>0$ such that the ball $\sB(\tau):=\{z:\|z\|_2<\tau\}$ satisfies $\sB(\tau)\subseteq\sZ_C$ and $h\in\sB(\tau)$. Define 
\[
\tau:=\inf_{\|u\|_2=1}\left\|\widetilde{\sJ}_C [u]\right\|_1
\]
Suppose, for contradiction, that some $z\in\sB(\tau)$ satisfies $z\notin\sZ_C$. Since $\sZ_C$ is closed and convex, the separating hyperplane theorem gives a unit vector $u$ with
\[
\inp{u}{z} > \sup_{w\in\sZ_C}\inp{u}{w}.
\]
But
\[
\inp{u}{z} \leq \|z\|_2 < \tau \leq \left\|\widetilde{\sJ}_C [u]\right\|_1 = \sup_{\|y\|_\infty\leq 1}\inp{\widetilde{\sJ}_C [u]}{y} = \sup_{w\in\sZ_C}\inp{u}{w},
\]
where the middle equality writes the $\ell_1$ norm as a supremum over the dual ball and the last equality uses $\sZ_C=\widetilde{\sJ}_C^*([-1,1]^{m_C})$. This contradicts the strict inequality above, so $\sB(\tau)\subseteq\sZ_C$.

Next, we show $\|h\|<\tau$. Note for $m_N=0$, we have $h = 0$, and since $\tau > 0$ by the above construction, this inequality holds trivially. Let $\mu:=\Ex|Z|=\sqrt{2/\pi}$ for $Z\sim N(0,1)$. Since $\widetilde{\sJ}_C$ is an $m_C\times d_\star$ standard Gaussian matrix, \citet[Theorem~9.6.3]{vershynin2018high} gives
\[
\Ex\tau \geq \mu m_C - c\sqrt{m_C d_\star}
\]
for a universal constant $c>0$. Under the assumption $m_C\geq c_2 d_\star$ with $c_2:=4c^2/\mu^2$, this simplifies to $\Ex\tau\geq(\mu/2)m_C$.
The map $X\mapsto\inf_{\|u\|_2=1}\|Xu\|_1$ is $\sqrt{m_C}$-Lipschitz with respect to the Frobenius norm, so by concentration for Lipschitz functions of Gaussian matrices~\citep[Theorem~5.2.2]{vershynin2018high},
\[
\PP\{\tau \geq \Ex\tau - t\} \geq 1 - \exp\left(-\frac{t^2}{2m_C}\right).
\]
Setting $t=(\mu/4)m_C$ and using $\Ex\tau\geq(\mu/2)m_C$ gives
\begin{equation}
\PP\left\{\tau > \frac{\mu}{4}m_C\right\} \geq 1 - \exp\left(-\frac{m_C}{32}\right).\nonumber
\end{equation}
Since $\widetilde{\sJ}_N^*$ is an $d_\star\times m_N$ standard Gaussian matrix and $h=\widetilde{\sJ}_N^*[\sign(\varepsilon_N)]$ , we have $h\sim N(0,m_N I_{d_\star})$. As $x\mapsto\|x\|_2$ is $1$-Lipschitz, the same concentration argument (applied to the Gaussian vector $h/\sqrt{m_N}\sim N(0,I_{d_\star})$) yields
\begin{equation}
\PP\left\{\|h\|_2 < 2\sqrt{m_N d_\star}\right\} \geq 1 - \exp\left(-\frac{d_\star}{2}\right).\nonumber
\end{equation}
If $m_C\geq c_3 d_\star$ with $c_3:=64/\mu^2$, then $(\mu/4)m_C\geq 2\sqrt{m_N d_\star}$ (using $m_N\leq m_C$, since $m_C>m/2$). A union bound gives
\[
\PP(\|h\|_2<\tau) \geq 1 - \exp\left(-\frac{m_C}{32}\right) - \exp\left(-\frac{d_\star}{2}\right),
\]
provided $m_C\geq\max\{c_2,c_3\}d_\star$. On this event, $h\in\sB(\tau)\subseteq\sZ_C$, so $\Lambda_C\cap(-1,1)^{m_C}\neq\emptyset$, completing the proof of Condition (I).

\medskip
\noindent\emph{Proof of Condition~(N).}
Since \(m_C>m/2\) and \(m\geq c_1d_\star\) with \(c_1>2\), we have
\(m_C\geq d_\star\). Hence the Gaussian matrix
\(\widetilde{\sJ}_C^*\) has rank \(d_\star\) almost surely. Because
the whitening operator \(\sH\) is an isomorphism and $\widetilde \sJ_C^* = \sH\circ \sJ_C^*$,
\[
\rank\sJ_C
=
\rank\sJ_C^*
=
\rank\widetilde{\sJ}_C^*
=
d_\star
\qquad\text{almost surely}.
\]
We first establish that, almost surely,
\begin{equation}\label{eq:vertex-free-lower-bound}
|\sI_<(v)|\geq d_\star
\qquad
\text{for every }v\in\Lambda_C.
\end{equation}
Fix $\sD\subseteq[m_C]$ with $|\sD|<d_\star$. Let $\widetilde{\sJ}_{C,\sD}$ denote the restriction of $\widetilde{\sJ}_C$ to the coordinates in $\sD$, and write $\sD^c:=[m_C]\setminus\sD$.  Define
\[
\mathcal B_{\sD,\tau}
:=
\left\{
\begin{array}{c}
\text{there exists }v\in\Lambda_C
\text{ such that }\sI_<(v)=\sD \text{ and }v_{\sD^c}=\tau
\end{array}
\right\}.
\]
The event on which~\eqref{eq:vertex-free-lower-bound} fails is exactly
\begin{equation}\label{eq:vertex-free-lower-bound-failure}
\bigl\{\text{there exists }v\in\Lambda_C
\text{ with }|\sI_<(v)|<d_\star\bigr\}
=\bigcup_{\substack{\sD\subseteq[m_C]\\|\sD|<d_\star}}
\ \bigcup_{\tau\in\{\pm1\}^{\sD^c}}
\mathcal B_{\sD,\tau}.
\end{equation}
This is a finite union, so it suffices to prove $\PP(\mathcal B_{\sD,\tau})=0$ for every fixed choice of $\sD$ and $\tau$.
Fix such a pair $(\sD,\tau)$. Recall $h=\widetilde{\sJ}_N^*[\sign(\varepsilon_N)]$. On the event $\mathcal B_{\sD,\tau}$, the free coordinates $v_{\sD}$ of the corresponding point $v\in\Lambda_C$ satisfy
\begin{equation}\label{eq:vertex-face-intersection}
\widetilde{\sJ}_{C,\sD}^*[v_{\sD}]
=h-\widetilde{\sJ}_{C,\sD^c}^*[\tau].
\end{equation}
Let $e_1,\ldots,e_{m_C}$ be the standard basis of $\RR^{m_C}$. Since $|\sD|<d_\star\le m_C$, the complement $\sD^c$ is nonempty; fix $j\in\sD^c$. Equation~\eqref{eq:vertex-face-intersection} can hold only if
\[
\widetilde{\sJ}_C^*[e_j]
\in \tau_j\left(
h-\sum_{i\in\sD^c\setminus\{j\}}
\tau_i\widetilde{\sJ}_C^*[e_i]
\right)
+\Im\widetilde{\sJ}_{C,\sD}^* =:\mathscr L_{\sD,\tau,j}.
\]
Let $\mathscr F_j :=\sigma\bigl(h,\,\widetilde{\sJ}_C^*[e_i]: i\in[m_C]\setminus\{j\}\bigr)$ be the $\sigma$-algebra generated by the random variables $h$ and all columns of $\widetilde{\sJ}_C^*$ except the $j$th.
Since $j\notin\sD$, conditional on $\mathscr F_j$, the set $\mathscr L_{\sD,\tau,j}$ is a fixed affine subspace of $\RR^{d_\star}$ of dimension at most $|\sD|<d_\star$. Due to Assumption~\ref{ass:whitening}, $\widetilde{\sJ}_C^*[e_j]\sim N(0, I_{d_\star})$ and is independent of $\mathscr F_j$. Therefore,
\[
\PP\bigl(\mathcal B_{\sD,\tau}\mid\mathscr F_j\bigr)
\le
\PP\bigl\{\widetilde{\sJ}_C^*[e_j]
\in\mathscr L_{\sD,\tau,j}\mid\mathscr F_j\bigr\}
=0
\qquad\text{almost surely}.
\]
Taking expectations gives $\PP(\mathcal B_{\sD,\tau})=0$. Since this holds for every term in the finite union~\eqref{eq:vertex-free-lower-bound-failure}, the union has probability zero, proving~\eqref{eq:vertex-free-lower-bound}.

Now fix a vertex \(\nu\) of \(\Lambda_C\) and set $\sD:=\sI_<(\nu).$
We claim that $\widetilde{\sJ}_{C,\sD}^*$
is injective. Otherwise, there exists a nonzero vector
\(u_{\sD}\in\ker\widetilde{\sJ}_{C,\sD}^*\). Extend it to
\(u\in\mathbb{R}^{m_C}\) by setting \(u_{\sD^c}=0\). Since
\(|\nu_i|<1\) for every \(i\in\sD\), there exists \(\epsilon>0\)
such that
\[
\nu+\epsilon u,\ \nu-\epsilon u\in[-1,1]^{m_C}.
\]
Moreover, $\widetilde{\sJ}_C^*[u]=\widetilde{\sJ}_{C,\sD}^*[u_{\sD}]=0.$
It follows that $	\widetilde{\sJ}_C^*[\nu\pm\epsilon u]=\widetilde{\sJ}_C^*[\nu]=h.$
Thus, \(\nu+\epsilon u\) and \(\nu-\epsilon u\) are two distinct
points of \(\Lambda_C\) whose midpoint is \(\nu\), contradicting the
assumption that \(\nu\) is a vertex. Therefore,
\(\widetilde{\sJ}_{C,\sD}^*\) is injective, which implies $	|\sD|\leq d_\star.$
Combining this with \eqref{eq:vertex-free-lower-bound} gives $|\sD|=d_\star.$
Consequently,
\[
\rank\widetilde{\sJ}_{C,\sD}^*
=
\rank\widetilde{\sJ}_{C,\sD}
=
d_\star.
\]
Since \(\sH\) is an isomorphism on the image of
\(\sJ_{C,\sD}^*\), we conclude that
\[
\rank\sJ_{C,\sD}
=
\rank\sJ_{C,\sD}^*
=
\rank\widetilde{\sJ}_{C,\sD}^*
=
d_\star
=
\rank\sJ_C.
\]
Since this argument applies to every vertex
\(\nu\), Condition~\emph{(N)} holds almost surely.

Finally, combining the high-probability event for
Condition~\emph{(I)} with the almost-sure event for
Condition~\emph{(N)}, and adjusting the universal constants, gives
\[
\PP\left\{
\text{Conditions~\emph{(I)} and~\emph{(N)} hold}
\right\}
\geq
1-\exp(-c_2d_\star).
\]
This completes the proof.\qed

\subsection{Proof of Theorem~\ref{thm:nonnegative-sparse}}\label{app::cone-select-sp-hp}
Fix a ground-truth factor $u^\star$, and set $S:=\supp(u^\star)=\supp(x^\star)$ and $q:=n-s\geq2$.

\emph{Criticality and affine-cube geometry.}
Using the argument of \citet[Corollary~1]{ma2025can}, specialized to diagonal Gaussian measurement matrices, $u^\star$ is a critical point of $f_{\mathrm{sp}}$ with probability at least $1-e^{-c s}$ whenever $m\geq c_1s$, where $c,c_1>0$ depend only on $p$. For nonnegative sparse regression, the linearized design map is $\sT: \RR^n\to \SS^n$ with $\sT[V]=2\diag(u^\star\odot V)$. Its image is the space of diagonal matrices supported on $S$, and therefore $d_\star=\dim\Im\sT=s$. Lemma~\ref{lem:whitening-operator} and Proposition~\ref{prop:interior-point-event} then imply that Assumption~\ref{ass:affine-cube} holds with probability at least $1-e^{-c's}$, after increasing $c_1$ if necessary. A union bound, followed by an adjustment of constants, shows that with probability at least $1-e^{-c_2s}$, the factor $u^\star$ is critical and Assumption~\ref{ass:affine-cube} holds simultaneously.

\emph{Isotropic kernel directions.}
Enumerate $S^c = [n]\backslash S$ as $\{i_1,\ldots,i_q\}$ and, for each $a\in[q]$, set $V_a:=e_{i_a}$. Since $u^\star_{i_a}=0$,
\[
\sT[V_a]
=
2\diag(u^\star\odot V_a)
=
0,
\]
and hence $V_a\in\ker\sJ_C$. Moreover,
\[
\sQ_C(V_a)
=
\sA_C(V_aV_a^\top)
=
\bigl(a_{\ell,i_a}\bigr)_{\ell\in\Omega_C}.
\]
Thus $\sQ_C(V_1),\ldots,\sQ_C(V_q)$ are independent standard Gaussian vectors in $\RR^{m_C}$. More importantly, this remains true after conditioning on $\sJ_C$ and $\Lambda_C$. To see this, note that $\sJ_C$ depends only on the on-support coordinates $\{a_{\ell,S}:\ell\in\Omega_C\}$, while $\Lambda_C$ is determined by these coordinates, their noisy counterparts $\{a_{\ell,S}:\ell\in\Omega_N\}$, and $\varepsilon$. In contrast, $\sQ_C(V_a)$ depends only on the off-support coordinates $\{a_{\ell,i_a}:\ell\in\Omega_C\}$. Independence of the Gaussian coordinates and independence of $\varepsilon$ from the measurements therefore give the claimed conditional independence.

Since $\sN_C$ is determined by $\sJ_C$, it follows that, conditional on $\sJ_C$ and $\Lambda_C$,
\[
P_{\sN_C}\sQ_C(V_1),\ldots,
P_{\sN_C}\sQ_C(V_q)
\]
are independent standard Gaussian vectors on $\sN_C$; in particular, their distributions are continuous and spherically symmetric. Moreover, $m_C>m/2>s$ after increasing $c_1$ if necessary, and $\rank\sJ_C=s$ almost surely, so $\dim\sN_C=m_C-s\geq1$. All hypotheses of Proposition~\ref{lem:isotropic-branch-occupancy} are therefore satisfied, implying that, conditional on $\sJ_C$ and $\Lambda_C$, the cone-selection test succeeds with probability at least $1-2^{-(q-1)}$. Theorem~\ref{thm:spectrahedral-deterministic} then implies that $u^\star$ is inactive.

\emph{Probability bound and uniformity.}
Combining the preceding high-probability events and recalling that $q=n-s$ gives
\[
\PP\{u^\star\text{ is an inactive critical point}\}
\geq
1-e^{-c_2s}-2^{-(n-s-1)}.
\]
Finally, any other ground-truth factor $\widehat u^\star$ satisfying $\widehat u^\star\odot\widehat u^\star=x^\star$ has the form $\widehat u^\star=Du^\star$ for a diagonal sign matrix $D$. The map $u\mapsto Du$ is orthogonal and leaves $f_{\mathrm{sp}}$ invariant. Lemma~\ref{lem:orthogonal-invariance} therefore shows that, on the same event, every ground-truth factor is an inactive critical point. No union bound over the ground-truth factors is needed. This proves the theorem. \qed

\subsection{Proof of Theorem~\ref{thm:sym-MS}}\label{app::cone-select-sym-hp}
The proof follows the same structure as that of Theorem~\ref{thm:nonnegative-sparse}. Fix a ground-truth factor $U^\star$ satisfying
$U^\star{U^\star}^\top=X^\star$, and set $q:=n-r\geq2$.

\emph{Criticality and affine-cube geometry.}
By \citet[Corollary~1]{ma2025can}, $U^\star$ is a critical point of $f_{\mathrm{sym}}$ with probability at least $1-e^{-cnr}$ whenever $m\geq c_1nr$, where $c,c_1>0$ depend only on $p$. In the symmetric model, $\sL$ is the identity and the linearized design map is $\sT[V]=U^\star V^\top+V{U^\star}^\top.$
To compute the dimension of its image, choose an orthogonal matrix $Q\in\RR^{k\times k}$ such that $U^\star Q=(\bar U,0)$ for some full-rank
$\bar U\in\RR^{n\times r}$. Then
\[
\Im\sT
=\{ U^\star QQ^\top V^\top+VQQ^\top{U^\star}^\top:V\in\RR^{n\times k}\} = 
\{\bar U H^\top+H\bar U^\top:H\in\RR^{n\times r}\}.
\]
The kernel of $H\mapsto\bar U H^\top+H\bar U^\top$ is $\{\bar U K:K^\top=-K\}$. The rank-nullity theorem gives
\[
d_\star=\dim\Im\sT=nr-\frac{r(r-1)}2.
\]
Lemma~\ref{lem:whitening-operator} and Proposition~\ref{prop:interior-point-event} now imply that Assumption~\ref{ass:affine-cube} holds with probability at least $1-e^{-c'd_\star}$, after increasing $c_1$ if necessary. Since $d_\star\geq nr/2$, a union bound and an adjustment of constants show that, with probability at least $1-e^{-c_2nr}$, the factor $U^\star$ is critical and Assumption~\ref{ass:affine-cube} holds simultaneously.

\emph{Isotropic kernel directions.}
Choose a unit vector $w\in\ker U^\star$, which exists because $k\geq r+1$, and let $u_1,\ldots,u_q$ be an orthonormal basis of $\col(U^\star)^\perp$. For each $j\in[q]$, set $V_j:=u_jw^\top.$ Since $U^\star w=0$,
\[
\sT[V_j] =U^\star V_j^\top+V_j{U^\star}^\top =0,
\]
and hence $V_j\in\ker\sJ_C$. Moreover, because $\|w\|=1$,
\[
\sQ_C(V_j) =\sA_C(V_jV_j^\top) =\sA_C(u_ju_j^\top) =\bigl(\langle\Gamma_\ell,u_ju_j^\top\rangle\bigr)_{\ell\in\Omega_C}.
\]
The matrices $u_ju_j^\top$, $j\in[q]$, form an orthonormal family in $(\Im\sT)^\perp$. Gaussian isotropy therefore shows that $\sQ_C(V_1),\ldots,\sQ_C(V_q)$ are independent standard Gaussian vectors in $\RR^{m_C}$. This conclusion remains valid after conditioning on $\sJ_C$ and $\Lambda_C$: the pair $(\sJ_C,\Lambda_C)$ is determined by the clean and corrupted measurements restricted to $\Im\sT$, together with $\varepsilon$, whereas the vectors $\sQ_C(V_j)$ depend on mutually orthogonal measurement components in $(\Im\sT)^\perp$. Independence of orthogonal Gaussian components and independence of $\varepsilon$ from the measurements give the claim.

Since $\sN_C$ is determined by $\sJ_C$, it follows that, conditional on $\sJ_C$ and $\Lambda_C$,
\[
P_{\sN_C}\sQ_C(V_1),\ldots, P_{\sN_C}\sQ_C(V_q)
\]
are independent standard Gaussian vectors on $\sN_C$; in particular, their distributions are continuous and spherically symmetric. Furthermore, $m_C\geq(1-p)m$. After increasing $c_1$ if necessary, Gaussian injectivity on the fixed space $\Im\sT$ gives $\rank\sJ_C=d_\star$ almost surely and
\[
\dim\sN_C=m_C-d_\star\geq1.
\]
All hypotheses of Proposition~\ref{lem:isotropic-branch-occupancy} are therefore satisfied. Conditional on $\sJ_C$ and $\Lambda_C$, with probability at least $1-2^{-(q-1)}$, the cone-selection test in Theorem~\ref{thm:spectrahedral-deterministic} succeeds. That theorem then implies that $U^\star$ is inactive.

\emph{Probability bound and uniformity.}
Combining the preceding high-probability events and recalling that $q=n-r$ gives
\[
\PP\{U^\star\text{ is an inactive critical point}\} \geq 1-e^{-c_2nr}-2^{-(n-r-1)}.
\]
Finally, every other ground-truth factor $\widehat U^\star$ satisfying $\widehat U^\star(\widehat U^\star)^\top=X^\star$ has the form $\widehat U^\star=U^\star Q$ for some orthogonal $Q\in\RR^{k\times k}$. The map $U\mapsto UQ$ is orthogonal and leaves $f_{\mathrm{sym}}$ invariant. Lemma~\ref{lem:orthogonal-invariance} therefore shows that, on the same event, every ground-truth factor is an inactive critical point. No union bound over the ground-truth factors is needed. This proves the theorem. \qed

\subsection{Proof of Theorem~\ref{thm:asym-MS}}\label{app::cone-select-asym-hp}
Rather than constructing $q\geq2$ distinct directions $V_1,\ldots,V_q\in\ker\sJ_C$ with spherically symmetric projected second-order residuals and invoking the resulting success probability $1-2^{-(q-1)}$, we use a sharper argument. We construct two directions $V_1,V_2\in\ker\sJ_C$ whose projected second-order residuals take the form $g$ and $-g$, where $g\neq0$ almost surely. Since the boundaries of the vertex normal cones have measure zero and each such cone is pointed, $g$ and $-g$ lie in the relative interiors of two distinct vertex normal cones almost surely. Consequently, the cone-selection test succeeds almost surely, improving the probability $1-2^{-(q-1)}$ to one.

\begin{proposition}[Two-sector branching implies successful cone selection]
\label{prop:two-sector-branching}
Fix a critical point $U^\star$ of $f_{\mathrm{LR}}$ satisfying $U^\star{U^\star}^\top=X^\star$. Suppose that Assumption~\ref{ass:affine-cube} holds and that $\dim\sN_C\geq 1$. Suppose further that there exist linearly independent directions $V_1,V_2\in\ker\sJ_C$ and a random vector $g\in\sN_C$ such that, conditional on $\sJ_C$ and $\Lambda_C$, the law of $g$ is absolutely continuous with respect to Lebesgue measure on $\sN_C$ and, almost surely,
\begin{equation}\label{eq:two-sector-residual-main}
	P_{\sN_C}\sQ_C(aV_1+bV_2)=ab\,g
	\qquad\text{for all }a,b\in\RR.
\end{equation}
Then, conditional on $\sJ_C$ and $\Lambda_C$, the cone-selection test in Theorem~\ref{thm:spectrahedral-deterministic} succeeds almost surely. Specifically, there exist two distinct vertices $\nu_1,\nu_2$ of $\Lambda_C$, a linear subspace $\sW\subseteq\ker\sJ_C$, and two nonempty relatively open cones $\sC_1,\sC_2\subseteq\sW$ such that
\[
P_{\sN_C}\sQ_C(V)
\in
\relint N_{\Lambda_C}^{\sN_C}(\nu_j)
\qquad
\text{for all }V\in\sC_j,\quad j=1,2.
\]
\end{proposition}
\begin{proof}
By Proposition~\ref{prop:fan-structure}, $\Lambda_C$ is full-dimensional in its affine space. Hence its vertex normal cones $N_{\Lambda_C}^{\sN_C}(\nu)$, $\nu\in\operatorname{Vert}(\Lambda_C)$, are full-dimensional and pointed, cover $\sN_C$, and have pairwise disjoint relative interiors~\citep[Theorem~2.3.2 and Proposition~2.3.6]{cox2024toric}. The union of their boundaries therefore has Lebesgue measure zero in $\sN_C$. By the absolute continuity of the conditional law of $g$, almost surely both $g$ and $-g$ are nonzero and avoid these boundaries. Hence there are unique vertices $\nu_1,\nu_2\in\operatorname{Vert}(\Lambda_C)$ such that $g\in\relint N_{\Lambda_C}^{\sN_C}(\nu_1)$ and $-g\in\relint N_{\Lambda_C}^{\sN_C}(\nu_2)$.
The vertices are distinct; otherwise the pointed cone $N_{\Lambda_C}^{\sN_C}(\nu_1)$ would contain both $g$ and $-g$, and hence the nonzero line $\Span{g}$, which is impossible.
Set $\sW:=\Span\{V_1,V_2\}$ and define
\[
\sC_1:=\{aV_1+bV_2:ab>0\},
\qquad
\sC_2:=\{aV_1+bV_2:ab<0\}.
\]
Because $V_1$ and $V_2$ are linearly independent, $\sC_1$ and $\sC_2$ are nonempty relatively open cones in $\sW$. For every $V=aV_1+bV_2\in\sC_1$, equation~\eqref{eq:two-sector-residual-main} yields
\[
P_{\sN_C}\sQ_C(V)
=
ab\,g
\in
\relint N_{\Lambda_C}^{\sN_C}(\nu_1),
\]
because $ab>0$. Similarly, for every $V=aV_1+bV_2\in\sC_2$, we have $ab<0$ and therefore
\[
P_{\sN_C}\sQ_C(V)
=
|ab|(-g)
\in
\relint N_{\Lambda_C}^{\sN_C}(\nu_2).
\]
Thus $\sW$, $\sC_1$, $\sC_2$, $\nu_1$, and $\nu_2$ satisfy the cone-selection conditions in Theorem~\ref{thm:spectrahedral-deterministic}. This completes the proof.
\end{proof}
Using Proposition~\ref{prop:two-sector-branching} in lieu of Proposition~\ref{lem:isotropic-branch-occupancy}, we prove Theorem~\ref{thm:asym-MS}.
\begin{proof}[Proof of Theorem~\ref{thm:asym-MS}]
The proof follows the same structure as those of Theorems~\ref{thm:nonnegative-sparse} and~\ref{thm:sym-MS}. Fix a balanced ground-truth factor pair $(U_1^\star,U_2^\star)$ satisfying $U_1^\star{U_2^\star}^\top=X^\star$ and ${U_1^\star}^\top U_1^\star={U_2^\star}^\top U_2^\star$, and introduce the stacked factor
\[
U^\star
:=
\begin{bmatrix}
	U_1^\star\\
	U_2^\star
\end{bmatrix}
\in\RR^{(n_1+n_2)\times k}.
\]
The objective $f_{\mathrm{asym}}(U_1,U_2)$ is precisely the robust low-rank objective~\eqref{eq:recovery-problem-BM} evaluated at the stacked factor $U=[U_1^\top\ U_2^\top]^\top$.

\emph{Criticality and affine-cube geometry.}
By \citet[Corollary~2]{ma2025can}, the stacked factor $U^\star$, equivalently the pair $(U_1^\star,U_2^\star)$, is a critical point with probability at least $1-e^{-c(n_1+n_2)r}$ whenever $m\geq c_1(n_1+n_2)r$, where $c,c_1>0$ depend only on $p$. For a perturbation
\[
H
=
\begin{bmatrix}
	H_1\\
	H_2
\end{bmatrix}
\in\RR^{(n_1+n_2)\times k},
\]
the linearized design map defined in Section~\ref{sec:low-rank-specialization} is
\[
\sT[H]
=
\sL\left(U^\star H^\top+H{U^\star}^\top\right)
=
\begin{bmatrix}
	0 & U_1^\star H_2^\top+H_1{U_2^\star}^\top\\
	U_2^\star H_1^\top+H_2{U_1^\star}^\top & 0
\end{bmatrix}.
\]
Write $\sT_{12}[H]
:=
U_1^\star H_2^\top+H_1{U_2^\star}^\top$
for its upper-right block. Since $\sT[H]$ is uniquely determined by $\sT_{12}[H]$, we have $\dim\Im\sT=\dim\Im\sT_{12}$.
Because the factor pair is balanced and $X^\star$ has rank $r$, the common Gram matrix ${U_1^\star}^\top U_1^\star={U_2^\star}^\top U_2^\star$ has rank $r$. Choose an orthogonal matrix $Q\in\RR^{k\times k}$ such that $U_1^\star Q=(\bar U_1,0)$ and $U_2^\star Q=(\bar U_2,0)$,
where $\bar U_1\in\RR^{n_1\times r}$ and $\bar U_2\in\RR^{n_2\times r}$ have full column rank. Since right multiplication by $Q$ is a linear isomorphism on the perturbation space,
\[
\Im\sT_{12}
=
\left\{
\bar U_1H_2^\top+H_1\bar U_2^\top:
H_1\in\RR^{n_1\times r},\ H_2\in\RR^{n_2\times r}
\right\}.
\]
The kernel of the map $(H_1,H_2)\mapsto\bar U_1H_2^\top+H_1\bar U_2^\top$ is $\left\{
(\bar U_1K,-\bar U_2K^\top):K\in\RR^{r\times r}
\right\}.$
The rank-nullity theorem therefore gives
\[
d_\star
=\dim\Im\sT
=\dim\Im\sT_{12}
=r(n_1+n_2)-r^2
=r(n_1+n_2-r).
\]
Lemma~\ref{lem:whitening-operator} and Proposition~\ref{prop:interior-point-event} now imply that Assumption~\ref{ass:affine-cube} holds with probability at least $1-e^{-c'd_\star}$, after increasing $c_1$ if necessary. Since $r\leq\min\{n_1,n_2\}$, we have $d_\star\geq\frac12(n_1+n_2)r$. A union bound and an adjustment of constants thus show that, with probability at least $1-e^{-c_2(n_1+n_2)r}$, the stacked factor $U^\star$ is critical and Assumption~\ref{ass:affine-cube} holds simultaneously.

\emph{Two-sector kernel directions.}
The balancing identity gives $\ker U^\star=\ker U_1^\star=\ker U_2^\star.$
Choose unit vectors $w\in\ker U^\star$, $u\in\col(U_1^\star)^\perp$, and $v\in\col(U_2^\star)^\perp$,
which exist because $k\geq r+1$, $n_1\geq r+1$, and $n_2\geq r+1$. Define the stacked directions
\[
V_1
:=
\begin{bmatrix}
	uw^\top\\
	0_{n_2\times k}
\end{bmatrix}	\in\RR^{(n_1+n_2)\times k},
\qquad
V_2
:=
\begin{bmatrix}
	0_{n_1\times k}\\
	vw^\top
\end{bmatrix}
\in\RR^{(n_1+n_2)\times k}.
\]
The directions $V_1$ and $V_2$ are linearly independent. Moreover,
\[
\sT_{12}[V_1]
=u(U_2^\star w)^\top
=0,
\qquad
\sT_{12}[V_2]
=(U_1^\star w)v^\top
=0.
\]
Hence $V_1,V_2\in\ker\sT\subseteq\ker\sJ_C$. For $a,b\in\RR$, the two blocks of $aV_1+bV_2$ are $au w^\top$ and $bv w^\top$, and therefore
\[
\sL\left((aV_1+bV_2)(aV_1+bV_2)^\top\right)
=
ab
\begin{bmatrix}
	0 & uv^\top\\
	vu^\top & 0
\end{bmatrix}.
\]
It follows that the clean second-order residual satisfies
\[
\sQ_C(aV_1+bV_2)
=ab\,z_C,
\qquad
z_C
:=
\bigl(\langle\Gamma_\ell,uv^\top\rangle\bigr)_{\ell\in\Omega_C}.
\]
Furthermore, $uv^\top\in(\Im\sT_{12})^\perp$, because ${U_1^\star}^\top u=0$ and ${U_2^\star}^\top v=0$. Consequently, Gaussian isotropy shows that $z_C$ is a standard Gaussian vector in $\RR^{m_C}$ independent of the measurement components that determine $\sJ_C$ and $\sJ_N$. Since $\varepsilon$ is independent of the measurements, $z_C$ is independent of $(\sJ_C,\sJ_N,\varepsilon)$ and therefore of $(\sJ_C,\Lambda_C)$.
Since $m_C\geq(1-p)m$, after increasing $c_1$ if necessary we have $m_C>d_\star$. Gaussian injectivity on the fixed space $\Im\sT$ then gives $\rank\sJ_C=d_\star$ almost surely, and hence
\[
\dim\sN_C
=m_C-d_\star
\geq1.
\]
Because $\sN_C$ is determined by $\sJ_C$, it follows that, conditional on $\sJ_C$ and $\Lambda_C$, $g:=P_{\sN_C}z_C$
is a standard Gaussian vector on $\sN_C$ and therefore has a density there. Projecting the second-order identity onto $\sN_C$ yields
\[
P_{\sN_C}\sQ_C(aV_1+bV_2)
=ab\,g
\qquad
\text{for all }a,b\in\RR.
\]
Thus all hypotheses of Proposition~\ref{prop:two-sector-branching} are satisfied. That proposition shows that the cone-selection conditions in Theorem~\ref{thm:spectrahedral-deterministic} hold almost surely, and the theorem implies that $U^\star$, equivalently $(U_1^\star,U_2^\star)$, is inactive.

\emph{Probability bound and uniformity.}
Combining the preceding events gives
\[
\PP\left\{
U^\star
\text{ is an inactive critical point}
\right\}
\geq
1-e^{-c_2(n_1+n_2)r}.
\]
Finally, let
\[
\widehat U^\star
:=
\begin{bmatrix}
	\widehat U_1^\star\\
	\widehat U_2^\star
\end{bmatrix}
\]
be any other balanced ground-truth factor. Any two balanced factorizations of $X^\star$ differ by a common right-orthogonal transformation, so $\widehat U^\star=U^\star\widehat Q$ for some orthogonal matrix $\widehat Q\in\RR^{k\times k}$. The map $U\mapsto U\widehat Q$ is orthogonal and leaves the stacked objective, and hence $f_{\mathrm{asym}}$, invariant. Lemma~\ref{lem:orthogonal-invariance} therefore shows that, on the same event, every balanced ground-truth factor is an inactive critical point. No union bound over the balanced factorizations is needed. This proves the theorem.
\end{proof}

\section{Details of the Low-Rank Matrix-Sensing Experiment}\label{app:low-rank-experiment}

Figure~\ref{fig:small-random-tilt-dynamics} uses the symmetric matrix-sensing formulation~\eqref{eq::sym-MS} with exact measurements. We set $n=k=100$, $r=2$, and $m=200$. In each of $50$ independent trials, we generate an orthogonal matrix $Q$ from a standard Gaussian matrix, let $Q_r$ contain its first two columns, and define 
\[
X^\star=Q_r\diag(10,1)Q_r^\top.
\]
Thus, any ground-truth factor $U^\star$ with $U^\star{U^\star}^\top=X^\star$ yields $\|U^\star\|_F^2/\|U^\star\|_{\mathrm{op}}^2=1.1$. The sensing matrices $\Gamma_i$ have i.i.d. standard normal entries, and $y_i=\langle\Gamma_i,X^\star\rangle$.
We compare the $\ell_1$-loss
\[
\phi_0(U)=\frac{1}{m}\sum_{i=1}^m
\left|\left\langle\Gamma_i,UU^\top-X^\star\right\rangle\right|
\]
with $\phi_\lambda(U)=\phi_0(U)+\lambda\langle G,U\rangle$, where $\lambda=0.20$ and $G=\widetilde G/\|\widetilde G\|_F$ for a standard normal matrix $\widetilde G$. The paired trajectories start from the same initialization $U_0=U^\star+E$, where the entries of $E$ are independently drawn from $N(0,(5\cdot10^{-5})^2)$. For $\beta\in\{0,\lambda\}$, we perform $5000$ iterations
\begin{align*}
U_{t+1}^{(\beta)} &=U_t^{(\beta)}-\frac{0.01}{\sqrt{t+1}}\,g_t^{(\beta)},\\
g_t^{(\beta)} &=\frac{1}{m}\sum_{i=1}^m
\sign\!\left(\left\langle\Gamma_i,U_t^{(\beta)}{U_t^{(\beta)}}^\top-X^\star\right\rangle_F\right)
(\Gamma_i+\Gamma_i^\top)U_t^{(\beta)}+\beta G,
\end{align*}
Within each trial, the two trajectories share the ground truth, measurements, and initialization and differ only by the linear perturbation. We record the relative recovery error and effective rank every ten iterations and report all trajectories together with their coordinate-wise medians in Figure~\ref{fig:small-random-tilt-dynamics}. All simulations use NumPy's \texttt{default\_rng}, with seed $2+1009j$ for trial $j=0,\ldots,49$.
\end{document}